\pdfoutput=1
\documentclass[12pt]{amsart}
\usepackage{lmodern}
\usepackage{ifthen}
\usepackage{amsfonts}
\usepackage{amsxtra}
\usepackage{amssymb}
\usepackage{mathdots}
\usepackage{array}
\usepackage[margin=0.8in]{geometry}
\usepackage{xcolor}
\definecolor{cite}{rgb}{0.30,0.60,1.00}
\definecolor{url}{rgb}{0.00,0.00,0.80}
\definecolor{link}{rgb}{0.40,0.10,0.20}
\usepackage[colorlinks,linkcolor=link,urlcolor=url,citecolor=cite,pagebackref,breaklinks]{hyperref}
\usepackage{bbm}
\usepackage{mathtools}
\usepackage{mathrsfs}
\usepackage{appendix}
\usepackage[all]{xy}
\usepackage[lite,abbrev,msc-links,alphabetic]{amsrefs}

\usepackage{graphicx}
\usepackage{multirow}
\usepackage{pstricks}
\usepackage{pst-pdf}
\usepackage{enumitem}
\usepackage{pifont}
\usepackage{marvosym}

\usepackage[OT2,T1]{fontenc}
\DeclareSymbolFont{cyrletters}{OT2}{wncyr}{m}{n}
\DeclareMathSymbol{\Sha}{\mathalpha}{cyrletters}{"58}

\usepackage{graphicx}

\makeatletter
\providecommand*{\Dashv}{%
  \mathrel{%
    \mathpalette\@Dashv\vDash
  }%
}
\newcommand*{\@Dashv}[2]{%
  \reflectbox{$\m@th#1#2$}%
}
\makeatother

\setitemize[0]{leftmargin=20pt}
\setenumerate[0]{leftmargin=30pt}

\numberwithin{equation}{section}

\theoremstyle{plain}
\newtheorem{proposition}{Proposition}[section]

\newtheorem{corollary}[proposition]{Corollary}
\newtheorem{lem}[proposition]{Lemma}
\newtheorem{theorem}[proposition]{Theorem}

\theoremstyle{definition}
\newtheorem{definition}[proposition]{Definition}

\theoremstyle{remark}
\newtheorem{remark}[proposition]{Remark}
\newtheorem{example}[proposition]{Example}

\renewcommand{\b}[1]{\mathbf{#1}}
\renewcommand{\c}[1]{\mathcal{#1}}
\renewcommand{\d}[1]{\mathbb{#1}}
\newcommand{\f}[1]{\mathfrak{#1}}
\renewcommand{\r}[1]{\mathrm{#1}}
\newcommand{\s}[1]{\mathscr{#1}}
\renewcommand{\sf}[1]{\mathsf{#1}}
\renewcommand{\(}{\left(}
\renewcommand{\)}{\right)}
\newcommand{\res}{\mathbin{|}}
\newcommand{\ol}[1]{\overline{#1}{}}
\newcommand{\wt}[1]{\widetilde{#1}{}}
\newcommand{\ul}{\underline}
\renewcommand{\leq}{\leqslant}
\renewcommand{\geq}{\geqslant}

\newcommand{\bG}{\b G}

\newcommand{\cA}{\c A}
\newcommand{\cB}{\c B}
\newcommand{\cC}{\c C}
\newcommand{\cD}{\c D}

\newcommand{\cH}{\c H}

\newcommand{\cM}{\c M}
\newcommand{\cN}{\c N}
\newcommand{\cO}{\c O}

\newcommand{\cR}{\c R}
\newcommand{\cS}{\c S}

\newcommand{\cU}{\c U}

\newcommand{\cZ}{\c Z}

\newcommand{\dA}{\d A}

\newcommand{\dC}{\d C}

\newcommand{\dQ}{\d Q}
\newcommand{\dR}{\d R}

\newcommand{\dZ}{\d Z}

\newcommand{\fC}{\f C}
\newcommand{\fD}{\f D}

\newcommand{\fT}{\f T}

\newcommand{\fV}{\f V}

\newcommand{\fa}{\f a}
\newcommand{\fb}{\f b}
\newcommand{\fc}{\f c}

\newcommand{\fg}{\f g}
\newcommand{\fh}{\f h}

\newcommand{\ft}{\f t}

\newcommand{\fz}{\f z}

\newcommand{\rI}{\r I}

\newcommand{\rL}{\r L}

\newcommand{\rR}{\r R}

\newcommand{\rU}{\r U}

\newcommand{\rc}{\r c}
\newcommand{\rd}{\,\r d}

\newcommand{\sP}{\s P}

\newcommand{\sR}{\s R}
\newcommand{\sS}{\s S}
\newcommand{\sT}{\s T}

\newcommand{\sfH}{\sf H}

\newcommand{\sfW}{\sf W}
\newcommand{\sfX}{\sf X}
\newcommand{\sfY}{\sf Y}

\newcommand{\tS}{\mathtt{S}}
\newcommand{\tT}{\mathtt{T}}

\newcommand{\bbs}{s}

\newcommand{\ad}{\r{ad}}
\newcommand{\Ad}{\r{Ad}}

\newcommand{\cusp}{\r{cusp}}

\newcommand{\der}{\r{der}}

\newcommand{\expo}{\r{exp}}

\newcommand{\Mul}{\r{Mul}}

\newcommand{\rss}{\r{rs}}
\newcommand{\RE}{\r{Re}\,}

\newcommand{\Temp}{\r{Temp}}
\newcommand{\tr}{\r{tr}}

\DeclareMathOperator{\As}{As}
\DeclareMathOperator{\Aut}{Aut}

\DeclareMathOperator{\BC}{BC}

\DeclareMathOperator{\End}{End}

\DeclareMathOperator{\Gal}{Gal}

\DeclareMathOperator{\GL}{GL}

\DeclareMathOperator{\Ind}{Ind}

\DeclareMathOperator{\Ker}{ker}

\DeclareMathOperator{\Lie}{Lie}

\DeclareMathOperator{\Mat}{Mat}

\DeclareMathOperator{\PGL}{PGL}

\DeclareMathOperator{\Res}{Res}

\DeclareMathOperator{\SL}{SL}

\begin{document}

\title[Isolation of cuspidal spectrum]
{Isolation of cuspidal spectrum, with application to the Gan--Gross--Prasad conjecture}

\author{Rapha\"{e}l Beuzart-Plessis}
\address{Aix Marseille Universit\'{e}, CNRS, I2M, Marseille, France}
\email{raphael.beuzart-plessis@univ-amu.fr}

\author{Yifeng Liu}
\address{Institute for Advanced Study in Mathematics, Zhejiang University, Hangzhou 310058, China}
\email{liuyf0719@zju.edu.cn}

\author{Wei Zhang}
\address{Department of Mathematics, Massachusetts Institute of Technology, Cambridge MA 02139, United States}
\email{weizhang@mit.edu}

\author{Xinwen Zhu}
\address{Division of Physics, Mathematics and Astronomy, California Institute of Technology, Pasadena CA 91125, United States}
\email{xzhu@caltech.edu}

\date{\today}

\begin{abstract}
  We introduce a new technique for isolating components on the spectral side of the trace formula. By applying it to the Jacquet--Rallis relative trace formula, we complete the proof of the global Gan--Gross--Prasad conjecture and its refinement Ichino--Ikeda conjecture for $\rU(n)\times\rU(n+1)$ in the stable case.
\end{abstract}

\maketitle

\tableofcontents

\section{Introduction}

\subsection{Isolation of cuspidal spectrum}

Let $F$ be a number field and $G$ a connected reductive group over $F$. In this subsection, we describe a general method for annihilating the non-cuspidal spectrum, while keeping certain prescribed representations of $G(\dA_F)$. We fix an open compact subgroup $K$ of $G(\dA_F^\infty)$. Denote by $\cS(G(\dA_F))_K$ the space of bi-$K$-invariant Schwartz test functions on $G(\dA_F)$, that is, we allow the archimedean component to be a Schwartz function rather than just a compactly supported smooth function. Recall that a Schwartz function on $G(F_\infty)$ is a smooth function $f$ such that $Df$ is bounded for every algebraic differential operator $D$ on $G(F_\infty)$.\footnote{Readers should not confuse it with a Harish-Chandra Schwartz function, which is defined by a much weaker condition.} Then $\cS(G(\dA_F))_K$ is an algebra under convolution and acts continuously on the $L^2$-spectrum $L^2(G(F)\backslash G(\dA_F)/K)$ via the right regular representation $\rR$.

\begin{theorem}[see Theorem \ref{th:isolation} for a more general version]\label{th:isolation_pre}
Let $\pi=\otimes_v\pi_v$ be an irreducible admissible representation of $G(\dA_F)$ with unitary automorphic central character $\omega$. Suppose that there does not exist a pair $(P,\sigma)$ where $P$ is a proper parabolic subgroup of $G$ (defined over $F$) and $\sigma$ is a cuspidal automorphic representation of $M(\dA_F)$ with $M$ the Levi quotient of $P$, such that $\pi_v$ is a constituent of $\Ind^G_P(\sigma_v)$ for all but finitely many places $v$ of $F$. Then there is a multiplier $\mu\star$ of the algebra $\cS(G(\dA_F))_K$ such that for every $f\in\cS(G(\dA_F))_K$,
\begin{enumerate}
  \item $\rR(\mu\star f)$ maps $L^2(G(F)\backslash G(\dA_F)/K,\omega)$ into $L^2_\cusp(G(F)\backslash G(\dA_F)/K,\omega)_\pi$, where the latter is the $\pi$-nearly isotypic subspace of $L^2_\cusp(G(F)\backslash G(\dA_F)/K,\omega)$, that is, the direct sum of $K$-invariants of irreducible subrepresentations $\tilde\pi$ such that $\tilde\pi_v\simeq\pi_v$ for all but finitely many places $v$ of $F$;

  \item $\pi(\mu\star f)=\pi(f)$.
\end{enumerate}
\end{theorem}

Recall that a multiplier of a complex linear algebra $\cS$ is a complex linear operator $\mu\star\colon\cS\to\cS$ that commutes with both left and right multiplications. Finding multipliers is most interesting when $\cS$, such as $\cS(G(\dA_F))_K$, is noncommutative and non-unital.

\begin{remark}\label{re:isolation}
We have
\begin{enumerate}
  \item Theorem \ref{th:isolation_pre}, together with its application to the stable case of the Gan--Gross--Prasad conjecture discussed in the next subsection, breaks the longstanding impression that to understand every (non-CAP) cuspidal automorphic representation, one has to understand the full $L^2$-spectrum on the spectral side of the (relative) trace formula.

  \item When $G$ is anisotropic modulo center in Theorem \ref{th:isolation_pre}, $L^2(G(F)\backslash G(\dA_F)/K,\omega)$ coincides with $L^2_\cusp(G(F)\backslash G(\dA_F)/K,\omega)$, which is a Hilbert direct sum of subspaces of the form $L^2_\cusp(G(F)\backslash G(\dA_F)/K,\omega)_{\pi'}$ for irreducible (cuspidal) automorphic representations $\pi'$ of $G(\dA_F)$ with central character $\omega$ and nonzero $K$-invariants up to near equivalence. In this case, our theorem implies that one can use multipliers to modify $f\in\cS(G(\dA_F))_K$ so that the effect is the same as composing the projection map to an arbitrarily given factor $L^2_\cusp(G(F)\backslash G(\dA_F)/K,\omega)_\pi$.

  \item It is crucial that in Theorem \ref{th:isolation_pre}, we work with the algebra $\cS(G(\dA_F))_K$ of Schwartz test functions. The method does not work if one works with compactly supported test functions.

  \item In Theorem \ref{th:isolation_pre}, we do not even require $\pi$ to be cuspidal automorphic. For example, it is possible that $\pi^K\neq\{0\}$ but $L^2_\cusp(G(F)\backslash G(\dA_F)/K,\omega)_\pi=0$. Then the theorem provides a uniform way to modify $f\in\cS(G(\dA_F))_K$, without changing its action on $\pi^K$, but annihilating the entire $L^2$-spectrum.
\end{enumerate}
\end{remark}

In fact, we also obtain a result for isolating general cuspidal components of the $L^2$-spectrum. For simplicity, here we only state the theorem for $G=\Res_{F'/F}\GL_n$ for a finite extension $F'/F$. For such $G$, we have the result on the classification of automorphic representations \cite{JS81}*{Theorem~4.4}.

\begin{theorem}[special case of Theorem \ref{th:isolation_tri}]\label{th:isolation_tri_pre}
In the situation of Theorem \ref{th:isolation_pre}, suppose that $G=\Res_{F'/F}\GL_n$ and that $\pi$ is an irreducible constituent of $\Ind_P^G(\sigma)$ for a parabolic subgroup $P$ of $G$ and a cuspidal automorphic representation $\sigma$ of the Levi quotient of $P$. Then there is a multiplier $\mu\star$ of the algebra $\cS(G(\dA_F))_K$ such that for every $f\in\cS(G(\dA_F))_K$,
\begin{enumerate}
  \item $\rR(\mu\star f)$ maps $L^2(G(F)\backslash G(\dA_F)/K,\omega)$ into $L^2_{(M,\sigma)}(G(F)\backslash G(\dA_F)/K,\omega)$, where the latter is the cuspidal component of $L^2(G(F)\backslash G(\dA_F)/K,\omega)$ associated to $(M,\sigma)$;

  \item $\pi(\mu\star f)=\pi(f)$.
\end{enumerate}
\end{theorem}

In the process of proving Theorem \ref{th:isolation_pre} and Theorem \ref{th:isolation_tri_pre}, we need a sufficient supply for multipliers of the archimedean component of $\cS(G(\dA_F))_K$. Now we temporarily switch the notation so that $G=\ul{G}(\dR)$ for a connected reductive algebraic group $\ul{G}$ over $\dR$. Let $\fh^*$ be the real vector space spanned by the weight lattice of the abstract Cartan group of $\ul{G}_\dC$, and $\sfW$ the Weyl group of $\ul{G}_\dC$, which acts on $\fh^*$. By the Harish-Chandra isomorphism, the infinitesimal character $\chi_\pi$ of an irreducible admissible representation $\pi$ of $G$ gives rise to a $\sfW$-orbit in $\fh^*_\dC$. In Definition \ref{de:multiplier}, we will define a space $\cM_{\theta\cup\{1\}}(\fh_\dC^*)$ of holomorphic functions on $\fh_\dC^*$ satisfying certain growth conditions, stable under the action of $\sfW$. Let $\cS(G)$ be the convolution algebra of Schwartz functions on $G$. The following is our theorem on multipliers of $\cS(G)$.

\begin{theorem}[Theorem \ref{th:multiplier}]\label{th:multiplier_pre}
For every function $\mu\in\cM_{\theta\cup\{1\}}(\fh_\dC^*)^\sfW$, there is a unique linear operator $\mu\star\colon\cS(G)\to\cS(G)$, such that
\[
\pi(\mu\star f)=\mu(\chi_\pi)\cdot\pi(f)
\]
holds for every $f\in\cS(G)$ and every irreducible admissible representation $\pi$ of $G$. In particular, $\mu\star$ is a multiplier of $\cS(G)$.
\end{theorem}

Now it is a good time to explain the crucial difference between $\cS(G)$ and $C^\infty_c(G)$. As pointed out in the final remark of the article \cite{Del84}, if $G$ has no compact factors, then the only $\sfW$-invariant holomorphic functions on $\fh_\dC^*$ that give rise to (continuous) multipliers of $C^\infty_c(G)$ are polynomials, that is, elements in the center of the universal enveloping algebra. Even if one considers only the subalgebra $C^\infty_c(G)_{(K)}$ of bi-$K$-finite compactly supported smooth functions for a fixed maximal compact subgroup $K$ of $G$, the $\sfW$-invariant holomorphic functions on $\fh_\dC^*$ that give rise to (continuous) multipliers of $C^\infty_c(G)_{(K)}$ have to be of exponential type, a property not required for $\cS(G)$. The removal of the restriction of being of exponential type will vastly increase the collection of multipliers, making it possible to obtain results like Theorem \ref{th:isolation_pre}, as long as the Casimir eigenvalues of cuspidal automorphic representations are distributed in a certain discrete way, while the latter is indeed fulfilled by a result of Donnelly \cite{Don82}.

\subsection{Application to the Gan--Gross--Prasad conjecture}

In this subsection, we describe the results obtained by applying Theorem \ref{th:isolation_pre} and its variants to the Jacquet--Rallis relative trace formula \cites{JR11,Zyd}. Let $E/F$ be a quadratic extension of number fields, with $\rc$ the Galois involution. Let $n\geq 1$ be an integer.

\begin{definition}\label{de:hermitian}
Let $\Pi$ be an isobaric automorphic representation of $\GL_n(\dA_E)$.
\begin{enumerate}
  \item We say that $\Pi$ is \emph{conjugate self-dual} if its contragredient $\Pi^\vee$ is isomorphic to $\Pi\circ\rc$.

  \item We say that $\Pi$ is \emph{hermitian} if $\Pi$ is an isobaric sum of mutually non-isomorphic cuspidal automorphic representations $\Pi_1,\dots,\Pi_d$ in which each factor $\Pi_i$ is conjugate self-dual and satisfies that $L(s,\Pi_i,\As^{(-1)^n})$ is regular at $s=1$.
\end{enumerate}
\end{definition}

Let $V$ be a (nondegenerate) hermitian space over $E$ of rank $n$ (with respect to the involution $\rc$). Put $G\coloneqq\rU(V)$, which is a reductive group over $F$.

\begin{definition}\label{de:weak}
Let $\pi$ be an irreducible admissible representation of $G(\dA_F)$. We say that an isobaric automorphic representation $\Pi$ of $\GL_n(\dA_E)$ is a \emph{weak automorphic base change} of $\pi$ if for all but finitely many places $v$ of $F$ split in $E$, the (split) local base change of $\pi_v$ is isomorphic to $\Pi_v$. If weak automorphic base change of $\pi$ exists, then it is unique up to isomorphism by \cite{Ram}*{Theorem~A}; and we denote it by $\BC(\pi)$.
\end{definition}

\begin{remark}\label{re:weak}
In most of the literatures, weak automorphic base change of $\pi$ requires the local-global compatibility for all but finitely many places, so in some sense our notion of weak automorphic base change should really be ``very weak automorphic base change,'' though by the endoscopic classification for unitary groups \cites{Art13,Mok15,KMSW}, we now know that these two definitions are equivalent. The reason we use this weaker notion is that we want to make our argument independent of the endoscopic theory for unitary groups. In fact, under our weaker notion of automorphic base change, we can prove, as a byproduct of the Jacquet--Rallis relative trace formulae, the following result:
\begin{itemize}
  \item The weak automorphic base change (in the sense of Definition \ref{de:weak}) of $\pi$ exists as long as there exist infinitely many places $v$ of $F$ split in $E$ such that $\pi_v$ is generic.
\end{itemize}
Our proof is completely free of using endoscopic trace formulae. The method can be used to show the local-global compatibility at all places where $\pi_v$ is unramified as well, but the argument will implicitly relies on \cites{Mok15,KMSW}. See Remark \ref{re:base_change} for more details.
\end{remark}

The following theorem confirms the global Gan--Gross--Prasad conjecture \cite{GGP12} for $\rU(n)\times\rU(n+1)$ in the stable case completely, which improves previous results in \cites{Zha1,Xue,BP1}.

\begin{theorem}\label{th:ggp}
Let $\Pi_n$ and $\Pi_{n+1}$ be hermitian cuspidal automorphic representations of $\GL_n(\dA_E)$ and $\GL_{n+1}(\dA_E)$, respectively. Then the following two statements are equivalent.
\begin{enumerate}
  \item We have $L(\tfrac{1}{2},\Pi_n\times\Pi_{n+1})\neq 0$.

  \item There exist
    \begin{itemize}
      \item a hermitian space $V_n$ over $E$ of rank $n$, which gives another hermitian space $V_{n+1}\coloneqq V_n\oplus E.e$ over $E$ of rank $n+1$ in which $e$ has norm $1$,

      \item for $m=n,n+1$, an irreducible subrepresentation $\pi_m\subseteq\cA_\cusp(G_m)$ of $G_m(\dA_F)$, where $G_m\coloneqq\rU(V_m)$, satisfying $\BC(\pi_m)\simeq\Pi_m$, and a cusp form $\varphi_m\in\pi_m$,
    \end{itemize}
    such that
    \[
    \sP(\varphi_n,\varphi_{n+1})\coloneqq\int_{G_n(F)\backslash G_n(\dA_F)}\varphi_n(h)\varphi_{n+1}(h)\rd h\neq 0
    \]
    where $\r{d}h$ is the Tamagawa measure on $G_n(\dA_F)$.
\end{enumerate}
\end{theorem}

Before our current work, the above theorem was known under the restriction that both $\Pi_n$ and $\Pi_{n+1}$ are supercuspidal at some prime of $E$ that is split over $F$. However, removing this (last) local restriction is crucial for arithmetic application to certain motives like symmetric powers of elliptic curves; see \cite{LTXZZ}.

Using the above theorem, we can obtain the following nonvanishing result on central $L$-values, which improves \cite{Zha1}*{Theorem~1.2}.

\begin{theorem}\label{th:nonvanishing}
Let $\Pi_{n+1}$ be a hermitian cuspidal automorphic representation of $\GL_{n+1}(\dA_E)$. Then there exists a hermitian cuspidal automorphic representation $\Pi_n$ of $\GL_n(\dA_E)$ such that
\[
L(\tfrac{1}{2},\Pi_n\times\Pi_{n+1})\neq 0.
\]
\end{theorem}

The last application is the Ichino--Ikeda conjecture \cite{II10}, which is a refinement of the Gan--Gross--Prasad conjecture by giving an explicit formula for $|\sP(\varphi_n,\varphi_{n+1})|^2$. In the case of $\rU(n)\times\rU(n+1)$, it is formulated in \cite{Har14}*{Conjecture~1.2} (see also \cite{Zha2}*{Conjecture~1.1}). The following theorem confirms the Ichino--Ikeda conjecture for $\rU(n)\times\rU(n+1)$ in the stable case completely, which improves previous results in \cites{Zha2,BP1,BP2}.

\begin{theorem}\label{th:ii}
Let the situation be as in (2) of Theorem \ref{th:ggp}. If moreover $\pi_n$ and $\pi_{n+1}$ are both everywhere tempered, then the Ichino--Ikeda conjecture holds for $\pi_n$ and $\pi_{n+1}$.
\end{theorem}

The proofs of Theorem \ref{th:ggp}, Theorem \ref{th:nonvanishing}, and Theorem \ref{th:ii} will be given in Subsection \ref{ss:ggp}.

\begin{remark}\label{re:endo}
The results on isolating components of the $L^2$-spectrum obtained in this article can vastly simplify the computation on the spectral side toward the endoscopic case of the Gan--Gross--Prasad and the Ichino--Ikeda conjectures for $\rU(n)\times\rU(n+1)$ as well. Indeed, on the unitary side, by Theorem \ref{th:isolation} and Theorem \ref{th:abc}, it suffices to understand the $\pi$-nearly isotypic subspace of the $L^2$-spectrum, which is contained in the \emph{cuspidal} spectrum; on the general linear side, by Theorem \ref{th:isolation_tri}, it suffices to understand the \emph{single} cuspidal component of the $L^2$-spectrum that corresponds to $\BC(\pi)$.\footnote{In fact, during the referee process of the current article, the endoscopic case has already been worked out in \cite{BPCZ}.}
\end{remark}

\begin{remark}
Using similar ideas, one can improve the results on the Gan--Gross--Prasad conjecture and the Ichino--Ikeda conjecture for $\rU(n)\times\rU(n)$ previously obtained by Hang Xue \cites{Xue14,Xue16}, based on the relative trace formulae developed in \cite{Liu14}, after establishing analogous results in \cites{BP2,CZ,Zyd}.
\end{remark}

\subsection{Strategy of proofs}
\label{ss:strategy}

The main part of the article has three sections, responsible for three clusters of results: existence of multipliers of Schwartz convolution algebras, isolation of cuspidal components in the $L^2$-spectrum, confirmation of the Gan--Gross--Prasad conjecture in the stable case and related results, respectively. In this subsection, we will briefly explain the strategy of proving these results.

In Section \ref{ss:2}, we aim to prove Theorem \ref{th:multiplier_pre}. First, we give a slightly more precise description of the space $\cM_{\theta\cup\{1\}}(\fh^*_\dC)$. Let $(\sfX^*,\Phi,\sfX_*,\Phi^\vee)$ be the root datum of $\ul{G}_\dC$ so that $\fh^*=\sfX^*\otimes_\dZ\dR$. For every automorphism $\vartheta$ of $(\sfX^*,\Phi,\sfX_*,\Phi^\vee)$ of order at most two, we can formulate two types of growth conditions -- \emph{moderate growth} and \emph{rapid decay} with respect to $\vartheta$ -- for holomorphic functions on $\fh^*_\dC$ (see Definition \ref{de:multiplier} for the accurate definition). In particular, when $\vartheta=1$, they coincide with the usual notions of moderate growth and rapid decay on vertical strips, respectively. The (inner form class) of $\ul{G}$ gives rise to a set $\theta$ of automorphisms of $(\sfX^*,\Phi,\sfX_*,\Phi^\vee)$ of order at most two, stable under $\sfW$-conjugation. Then we define $\cM_{\theta\cup\{1\}}(\fh^*_\dC)$ (resp.\ $\cN_{\theta\cup\{1\}}(\fh^*_\dC)$) to be the space of holomorphic functions on $\fh^*_\dC$ that have moderate growth (resp.\ rapid decay) with respect to all elements in $\theta\cup\{1\}$. It will be clear from the definition (and suggested by the terminology) that $\cN_{\theta\cup\{1\}}(\fh^*_\dC)\subseteq\cM_{\theta\cup\{1\}}(\fh^*_\dC)$. We will first prove Theorem \ref{th:multiplier_pre} for the smaller space $\cN_{\theta\cup\{1\}}(\fh^*_\dC)^\sfW$ (Proposition \ref{pr:multiplier}), and then use a limit process to pass to $\cM_{\theta\cup\{1\}}(\fh^*_\dC)^\sfW$.

For $\cN_{\theta\cup\{1\}}(\fh^*_\dC)^\sfW$, note that $\cN_{\theta\cup\{1\}}(\fh^*_\dC)$ is contained in $\cN_{\{1\}}(\fh^*_\dC)$, the space of holomorphic functions on $\fh^*_\dC$ that have rapid decay on vertical strips. Essentially by a result of Delorme (Proposition \ref{pr:delorme}), elements in $\cN_{\{1\}}(\fh^*_\dC)^\sfW$ will give multipliers for the subalgebra $\cS(G)_{(K)}$ of $\cS(G)$ of bi-$K$-finite Schwartz functions. Therefore, to construct $\mu\star f$ for $\mu\in\cN_{\theta\cup\{1\}}(\fh^*_\dC)^\sfW$ and $f\in\cS(G)$, we may choose a sequence $\{f_n\}\subseteq C^\infty_c(G)_{(K)}$ approaching $f$ in $\cS(G)$, and show that $\{\mu\star f_n\}$ converges. It is a crucial observation that the Fr\'echet topology of $\cS(G)$ is also induced from the semi-norms $f\mapsto\|Df\|_{L^2}$ for all algebraic differential operators $D$ on $G$. Thus, it suffices for us to show that $\|D(\mu\star -)\|_{L^2}$ is continuous on the subspace $C^\infty_c(G)_{(K)}$ with respect to the subspace topology for all $D$. When $D$ is invariant, we have $D(\mu\star f)=\mu\star(Df)$, hence the continuity is clear. It remains to consider the case where $D$ is a polynomial on $G$, or equivalently, a matrix coefficient of a finite dimensional algebraic representation of $G$. In this case, we can show that there are finitely many pairs $(S_i,L_i)$, depending on the matrix coefficient presentation of $D$ only, in which $S_i$ is a linear operator of $\cN_{\theta\cup\{1\}}(\fh^*_\dC)^\sfW$ and $L_i$ is a continuous linear operator of $\cS(G)$ preserving $C^\infty_c(G)_{(K)}$, such that $D(\mu\star f)=\sum_iS_i(\mu)\star L_i(f)$. Thus, we obtain the desired continuity for $\|D(\mu\star -)\|_{L^2}$, hence Theorem \ref{th:multiplier_pre} for the smaller space $\cN_{\theta\cup\{1\}}(\fh^*_\dC)^\sfW$ is confirmed. However, in order to pass to $\cM_{\theta\cup\{1\}}(\fh^*_\dC)^\sfW$, we will need more precise bounds on the linear operators $S_i$ with respect to a certain natural family of semi-norms on $\cN_{\theta\cup\{1\}}(\fh^*_\dC)$, as stated in Lemma \ref{le:multi_3}.

In Section \ref{ss:3}, we prove all results related to the isolation of the $L^2$-spectrum. To make the discussion here more explicit, we restrict ourselves to the number field $F=\dQ$ and the group $G=\PGL_{n,\dQ}$ with $n\geq 2$, and only consider Theorem \ref{th:isolation_pre} with $\pi$ cuspidal automorphic. In particular, $Z=1$ and we will omit the necessarily trivial character $\omega$ in the following notation. For every tuple $\ul{n}=(n_1,n_2,\dots)$ of integers $0<n_1\leq n_2\leq \cdots$ summing up to $n$, we let $M_{\ul{n}}$ be the standard diagonal Levi subgroup of $G$ of block sizes $\ul{n}$ and $\sfW_{\ul{n}}$ the subgroup of the Weyl group stabilizing $M_{\ul{n}}$. Then we have the coarse Langlands decomposition
\[
L^2(G(\dQ)\backslash G(\dA)/K)=\widehat\bigoplus_{(M,\sigma)}L^2_{(M,\sigma)}(G(\dQ)\backslash G(\dA)/K),
\]
where $M=M_{\ul{n}}$ for some $\ul{n}$ and $\sigma$ is a cuspidal automorphic representation of $M(\dA)$ up to twist and $\sfW_{\ul{n}}$-conjugation. Our goal, in view of the strong multiplicity one property, is to find $\mu$ such that for every $f\in\cS(G(\dA))_K$, $\mu\star f$ annihilates all components but $L^2_{(G,\pi)}(G(\dQ)\backslash G(\dA)/K)$ and maintains the action of $f$ on $\pi$.

For the first step, it is not hard to construct a function $\mu_\infty^0\in\cM_{\theta\cup\{1\}}(\fh^*_\dC)^\sfW$ (notation with respect to the real reductive group $G\otimes_\dQ\dR$) satisfying $\mu_\infty^0(\chi_{\pi_\infty})=1$ and the following condition: there exists a finite set $\fT$ of $K_\infty$-types, where $K_\infty$ is the standard maximal compact subgroup of $G(\dR)=\PGL_n(\dR)$, such that for every $f\in\cS(G(\dR))$, $\mu\star f$ annihilates the component $L^2_{(M,\sigma)}(G(\dQ)\backslash G(\dA)/K)$ if it has no $K_\infty$-types in $\fT$. This will exclude all but finitely many components $L^2_{(M,\sigma)}(G(\dQ)\backslash G(\dA)/K)$ with $M\in\{M_{(1,\dots,1)},M_{(n)}\}$. However, for other $M$, there are still infinitely many components remaining.

The second step is to annihilate all but finitely many components $L^2_{(M,\sigma)}(G(\dQ)\backslash G(\dA)/K)$ with a $K_\infty$-type in $\fT$ for every $M$. To explain the idea, we consider the simplest nontrivial case where $M=M_{(1,\dots,1,2)}$ (hence $n\geq 3$). It is easy to see that there exists a finite set $\fT_M$ of $(K_\infty\cap M(\dR))$-types and an open compact subgroup $K_M\subseteq M(\dA^\infty)$ such that if $L^2_{(M,\sigma)}(G(\dQ)\backslash G(\dA)/K)$ has a $K_\infty$-type in $\fT$, then $\sigma$ must have a $(K_\infty\cap M(\dR))$-type in $\fT_M$ and nontrivial $K_M$-invariants. Note that there are infinitely many such $\sigma$ up to twist and $\sfW_{(1,\dots,1,2)}$-conjugation! However, we observe that the Casimir operator for the derived subgroup of $M$, which is simply $\SL_{2,\dQ}$, gives a polynomial function $\lambda$ on $\fh^*_\dC$. It is a well-known result of Harish-Chandra that for any given $\lambda_0\in\dC$, there are only finitely many cuspidal automorphic representations $\sigma$ of $M(\dA)$ up to twist that have a $(K_\infty\cap M(\dR))$-type in $\fT_M$ and nontrivial $K_M$-invariants, such that $\lambda(\chi_{\sigma_\infty})=\lambda_0$. By another well-known result on the distribution of Casimir eigenvalues of $\SL_2$ (or \cite{Don82} for general semisimple groups), one can find a holomorphic function $\nu_M$ on $\fh^*_\dC$ that has moderate growth on vertical strips with zeros exactly $\lambda^{-1}\(\Lambda\setminus\lambda(\chi_{\pi_\infty})\)$, where $\Lambda\subseteq\dC$ is the set of Casimir eigenvalues of cuspidal automorphic representations $\sigma$ of $M(\dA)$ that have a $(K_\infty\cap M(\dR))$-type in $\fT_M$ and nontrivial $K_M$-invariants.\footnote{However, we can not require $\nu_M$ to be of exponential type at the same time.} Using $\nu_M$, it is not hard to construct an element $\mu_\infty^M\in\cM_{\theta\cup\{1\}}(\fh^*_\dC)^\sfW$ such that for every $f\in\cS(G(\dR))$, $\mu_\infty^M\star f$ annihilates all but finitely components $L^2_{(M,\sigma)}(G(\dQ)\backslash G(\dA)/K)$ with given $M=M_{(1,\dots,1,2)}$ and maintains the action of $f$ on $\pi$. In fact, we can achieve this for every $M$.

The last step is to annihilate, without changing $\pi$, every single component $L^2_{(M,\sigma)}(G(\dQ)\backslash G(\dA)/K)$ that is not isomorphic to $\pi$. This is easy if $M=M_{(n)}=G$, since we can use spherical Hecke operators at unramified (nonarchimedean) places. When $M\neq G$, $L^2_{(M,\sigma)}(G(\dQ)\backslash G(\dA)/K)$ is a ``continuous'' space of induced representations. However, there is a secret correlation between (nonarchimedean) Hecke eigenvalues and (archimedean) infinitesimal characters for all representations that contribute to $L^2_{(M,\sigma)}(G(\dQ)\backslash G(\dA)/K)$. This motivates us to construct, for every component $L^2_{(M,\sigma)}(G(\dQ)\backslash G(\dA)/K)$ that is not isomorphic to $\pi$, a multiplier $\mu_{(M,\sigma)}$ that is ``mixed'' from $\cM_{\theta\cup\{1\}}(\fh^*_\dC)^\sfW$ and spherical Hecke operators, so that for every $f\in\cS(G(\dA))_K$, $\mu_{(M,\sigma)}\star f$ annihilates $L^2_{(M,\sigma)}(G(\dQ)\backslash G(\dA)/K)$ and maintains the action of $f$ on $\pi$. We remark that this step is inspired by the work \cite{LV07}.

To conclude the proof of Theorem \ref{th:isolation_pre}, we only need to take the product of $\mu_\infty^0$, $\{\mu_\infty^M\}_M$, and (finitely many) $\mu_{(M,\sigma)}$.

In Section \ref{ss:4}, we prove the Gan--Gross--Prasad conjecture in the stable case and other related results. Here, we will explain how to apply Theorem \ref{th:isolation_pre} to the Jacquet--Rallis relative trace formulae to attack the Gan--Gross--Prasad conjecture. The Jacquet--Rallis relative trace formulae have two sides: the group $G'\coloneqq\Res_{E/F}\GL_{n,E}\times\Res_{E/F}\GL_{n+1,E}$, and the group $G^V\coloneqq\rU(V_n)\times\rU(V_{n+1})$ where $V=(V_n,V_{n+1})$ is a pair of hermitian spaces over $E$ as in Theorem \ref{th:ggp}. The two sides share the same space of orbits, which is an affine variety $B$ over $F$; in other words, we have surjective morphisms $G'\to B\leftarrow G^V$. By Zydor's extension of the Jacquet--Rallis relative trace formulae, we have a relative trace formula on each side: For every $f'\in\cS(G'(\dA_F))$ that annihilates the entire non-cuspidal part of $L^2(G'(F)\backslash G'(\dA_F),1)$, we have the identity
\[
\sum_{\text{$\Pi$ cuspidal}}I_\Pi(f')=\sum_{\gamma\in B(F)}I_\gamma(f')
\]
where $I_\Pi$ and $I_\gamma$ are certain invariant functionals on $\cS(G'(\dA_F))$ associated to a cuspidal automorphic representation $\Pi$ and an element $\gamma\in B(F)$ defined via relative characters and (regularized) relative orbital integrals, respectively. Similarly, for every $f^V\in\cS(G^V(\dA_F))$ that annihilates the entire non-cuspidal part of $L^2(G^V(F)\backslash G^V(\dA_F),1)$, we have the identity
\[
\sum_{\text{$\pi^V$ cuspidal}}J_{\pi^V}(f^V)=\sum_{\delta^V\in B(F)}J_{\delta^V}(f^V).
\]
In practice, we have to consider all pairs $V$ up to isomorphism. The starting point for the comparison of trace formulae is to find a pair of test functions $(f',\{f^V\}_V)$ that have \emph{matching orbital integrals}, so that $I_\gamma(f')=J_{\delta^V}(f^V)$ when $\gamma=\delta^V$. The common strategy of finding $(f',\{f^V\}_V)$ that annihilate non-cuspidal spectra is to take $f'$ so that $f'_v$ annihilates all non-supercuspidal representations at some nonarchimedean place $v$ of $F$ split in $E$, and similarly for $f^V$. However, by doing this, we necessarily annihilate all cuspidal automorphic representations that are nowhere supercuspidal. The new invention, which is enabled by Theorem \ref{th:isolation_pre}, is to modify an arbitrary pair $(f',\{f^V\}_V)$ that have matching orbital integrals by multipliers, so that the resulting pair annihilate non-cuspidal spectra and maintain their actions on any prescribed representations $(\Pi,\{\pi^V\})$ in which $\Pi$ is cuspidal automorphic and is isomorphic to $\BC(\pi^V)$. Thus, it is a natural question to find multipliers $(\mu',\{\mu^V\}_V)$ from Theorem \ref{th:isolation_pre}, such that $(\mu'\star f',\{\mu^V\star f^V\}_V)$ still have matching orbital integrals. The answer turns out to be quite elegant: there is a natural ``base change'' map from those multipliers for $G^V$ (obtained in the way of Theorem \ref{th:isolation_pre}) to those for $G'$; and we show in Proposition \ref{pr:transfer3} that if $\mu'$ is the base change of $\mu^V$ for all $V$, then $(\mu'\star f',\{\mu^V\star f^V\}_V)$ have matching orbital integrals as long as $(f',\{f^V\}_V)$ do. Such multipliers $(\mu',\{\mu^V\}_V)$ are not hard to find. Therefore, we can compare the above two relative trace formulae without sacrificing any prescribed representations $(\Pi,\{\pi^V\})$ as above. The rest of the argument is a standard business in trace formulae approach.

The article has an appendix (Appendix \ref{app}) in which we extend a result of Delorme to reductive groups, which is only used in the proof of Proposition \ref{pr:multiplier}.

\subsection{Notation and conventions}
\label{ss:notation}

\begin{itemize}
  \item For a set $S$, we denote by $\mathbf{1}_S$ the characteristic function of $S$.

  \item In the main text, if we do not specify the base ring of a tensor product $\otimes$, then the base ring is $\dC$.

  \item For a real vector space $U$, we put $U_\dC\coloneqq U\otimes_\dR\dC$, and $iU\coloneqq U\otimes_\dR i\dR$, which is a subspace of the underlying real vector space of $U_\dC$. We have the $\dR$-linear map $\r{Re}\colon U_\dC\to U$ by taking the real part.

  \item For a finite dimensional complex vector space $U$, we denote by $\cO(U)$ the ring of holomorphic functions on $U$, $\cO_\expo(U)\subseteq\cO(U)$ the subring of holomorphic functions of exponential type, and $\dC[U]\subseteq\cO_\expo(U)$ the subring of polynomial functions. A lattice (resp.\ full lattice) of $U$ is a subgroup $L$ of $U$ such that the induced map $L\otimes_\dZ\dC\to U$ is injective (resp. bijective).

  \item For a complex linear algebra $\cS$, we denote by $\Mul(\cS)$ the $\dC$-algebra of multipliers of $\cS$, that is, complex linear operators $\mu\star\colon\cS\to\cS$ satisfying $\mu\star(f\ast g)=(\mu\star f)\ast g=f\ast(\mu\star g)$ for every $f,g\in\cS$, where $\ast$ denotes the multiplication in $\cS$.

  \item By a prime of a number field, we mean a nonarchimedean place. In Sections \ref{ss:3} and \ref{ss:4}, we will encounter various sets of places of a number field $F$. To summarize,
    \begin{itemize}
      \item $\tS$ will always be a finite set consisting of primes;

      \item $\tT$ will always be a (possibly infinite) set consisting of primes;

      \item $\Box$ will always be a finite set containing all archimedean places.
    \end{itemize}

  \item For an algebraic group $G$ over a number field $F$, we put $G_\infty\coloneqq G(F\otimes_\dQ\dR)$ for short.

  \item A subgroup of an algebraic group defined over a field is by default defined over the same field.

  \item Let $P$ be a parabolic subgroup of a reductive group $G$.
     \begin{itemize}
       \item We denote by $N_P\subseteq P$ the unipotent radical.

       \item When $\sigma$ is an admissible representation of $P(R)$ for an appropriate ring $R$ so that admissibility makes sense, we denote by $\rI^G_P(\sigma)$ the normalized parabolic induction as an admissible representation of $G(R)$.

       \item When $\sigma$ is an admissible representation of $M(R)$ for a Levi subgroup $M$ of $P$, we also write $\rI^G_P(\sigma)$ by regarding $\sigma$ as a representation of $P(R)$ through inflation.
     \end{itemize}
\end{itemize}

\subsubsection*{Acknowledgements}

We would like to thank the American Institute of Mathematics for supporting three of us (Y.~L., W.~Z., and X.~Z.) an AIM SQuaREs project \emph{Geometry of Shimura varieties and arithmetic application to L-functions} from 2017 to 2019, where an initial discussion toward this article happened in our last SQuaREs workshop in the spring of 2019. We thank Yichao~Tian and Liang~Xiao, the other two members of our SQuaREs project, since the project itself is a crucial motivation for the current article. We thank Jean-Loup~Waldspurger for reading an early draft of the article and providing several useful comments. We thank the anonymous referee for careful reading and many helpful suggestions.

The project leading to this publication of R.~B.-P. has received funding from Excellence Initiative of Aix-Marseille University-A*MIDEX, a French ``Investissements d'Avenir'' programme. The research of Y.~L. is partially supported by the NSF grant DMS--1702019 and a Sloan Research Fellowship. The research of W.~Z. is partially supported by the NSF grant DMS--1838118 and DMS--1901642. The research of X.~Z. is partially supported by the NSF grant DMS--1902239.

\section{Multipliers of Schwartz convolution algebra}
\label{ss:2}

In this section, we construct sufficiently many multipliers of the convolution algebra of Schwartz functions for a real connected reductive algebraic group.

In Subsection \ref{ss:holomorphic}, we record some lemmas on constructing holomorphic functions on complex vector spaces with special properties. In Subsection \ref{ss:multiplier_functions}, we define various spaces of functions related to the multipliers of Schwartz convolution algebras. In Subsection \ref{ss:multiplier_algebra}, we state our result (Theorem \ref{th:multiplier}) on the existence of sufficiently many multipliers of Schwartz convolution algebras. In Subsection \ref{ss:multiplier_proof}, we prove Theorem \ref{th:multiplier}, following the strategy described in Subsection \ref{ss:strategy}.

\subsection{Preliminaries on holomorphic functions}
\label{ss:holomorphic}

We first review some facts about entire functions on the complex plane. Recall that the order of an entire function $\Psi\colon\dC\to\dC$ is defined as
\[
\inf\{e\in[0,+\infty)\res\exists C_e>0\text{, such that }|\Psi(z)|<C_e\expo(|z|^e)\text{ for all $z\in\dC$}\}.
\]
Here, $|z|=\sqrt{z\overline{z}}$. If the above set is empty, then we say that $\Psi$ has infinite order.

Next, we review the construction of entire functions with prescribed zeroes. Let $\Lambda$ be a subset of $\dC$, and $p\geq 0$ an integer. We define the (formal) Weierstrass product
\[
\Psi_{\Lambda,p}(z)\coloneqq z^\delta\prod_{\substack{\lambda\in\Lambda \\ \lambda\neq 0}}E_p(z/\lambda),
\]
where $E_p(z)\coloneqq(1-z)\expo(z+\cdots+z^p/p)$ is the elementary function, and $\delta=1$ (resp.\ $\delta=0$) if $0$ belongs (resp.\ does not belong) to $\Lambda$.

\begin{definition}\label{de:rank}
Let $p\geq 0$ be an integer. We say that $\Lambda$ has \emph{rank $p$} if
\begin{enumerate}
  \item $\Lambda$ is countable;

  \item $p$ is the least nonnegative integer such that $\sum_{\lambda\in\Lambda,\lambda\neq 0}|\lambda|^{-(p+1)}$ converges.
\end{enumerate}
\end{definition}

\begin{lem}\label{le:standard_function}
Let $\Lambda\subseteq\dC$ be a subset of rank $p$ for some integer $p\geq 0$. Then $\Psi_{\Lambda,p}$ is a well-defined entire function of (finite) order at most $p+1$, with the set of zeroes exactly $\Lambda$.
\end{lem}

\begin{proof}
It is well-known that there exist constants $C,C'>0$ such that $|\Psi_{\Lambda,p}(z)|<C\expo(C'|z|^{p+1})$ for $z\in\dC$. Thus, $\Psi_{\Lambda,p}$ is of finite order at most $p+1$ by definition. The set of zeroes is clear from the construction.
\end{proof}

Now we consider a finite dimensional real vector space $U$.

\begin{definition}\label{de:moderate}
Let $\mu$ be a holomorphic function on $U_\dC$. We say that $\mu$ has \emph{moderate vertical growth} if for every $M>0$, there exists $r_M\in\dR$ such that
\[
\sup_{\|\RE z\|<M}|\nu(z)|\cdot(1+\|z\|)^{r_M}<\infty
\]
holds for some, hence every, Euclidean norm $\|\cdot\|$ on $U_\dC$.
\end{definition}

\begin{proposition}\label{pr:rapid}
Let $\lambda\colon U_\dC\to\dC$ be a polynomial function. For every entire function $\Psi$ of finite order, there exists a holomorphic function $\nu$ on $U_\dC$ that has moderate vertical growth, such that the set of zeroes of $\nu$ is exactly the inverse image of the set of zeroes of $\Psi$ along $\lambda$.
\end{proposition}

\begin{proof}
We choose an isomorphism $U\simeq\dR^r$ and write elements in $U_\dC$ as $z=(z_1,\dots,z_r)$. Suppose that $\Psi$ is of order $e$ and $\lambda$ has degree $d$. Take an odd integer $q>de$. Define $\nu$ by the formula
\[
\nu(z)\coloneqq\expo(z_1^{2q}+\cdots+z_r^{2q})\cdot\Psi(\lambda(z))
\]
for $z\in U_\dC$, which is holomorphic. It is straightforward to check that $\nu(z)$ has moderate vertical growth. Moreover, we have $\nu(z)=0$ if and only if $\lambda(z)$ is a zero of $\Psi$. The proposition follows.
\end{proof}

\begin{corollary}\label{co:rapid}
Let $\lambda\colon U_\dC\to\dC$ be a polynomial function. For every subset $\Lambda\subseteq\dC$ of finite rank, there exists a holomorphic function $\nu$ on $U_\dC$ that has moderate vertical growth, such that the set of zeroes of $\nu$ is exactly $\lambda^{-1}\Lambda$.
\end{corollary}

\begin{proof}
It follows from Lemma \ref{le:standard_function} and Proposition \ref{pr:rapid}.
\end{proof}

\begin{corollary}\label{co:lattice}
Let $L$ be a lattice of $U_\dC$ and $A\subseteq U_\dC$ a finite subset. Then there exists a holomorphic function $\nu$ on $U_\dC$ that has moderate vertical growth, such that $\nu$ vanishes on $L\setminus A$ and is nowhere vanishing on $A$.
\end{corollary}

\begin{proof}
Let $r$ be the dimension of $U$. We may choose linearly independent complex linear maps $\lambda_1,\dots,\lambda_r\colon U_\dC\to\dC$ such that $L$ is contained in $\bigcap_{i=1}^r\lambda_i^{-1}\dZ$. For every $1\leq i\leq r$, the subset $\dZ\setminus\lambda_i(A)\subseteq\dC$ is of finite rank. By Corollary \ref{co:rapid}, we may find a holomorphic function $\nu_i$ on $U_\dC$ that has moderate vertical growth, such that the set of zeroes of $\nu_i$ is exactly $\lambda_i^{-1}(\dZ\setminus\lambda_i(A))$. Put $\nu\coloneqq\prod_{i=1}^r\nu_i$. Then $\nu$ has moderate vertical growth, vanishes on all but finitely many elements in $L$, and is nowhere vanishing on $A$. Let $z_1,\dots,z_s$ be the finitely many elements in $L\setminus A$ at which $\nu$ is nonvanishing. For each $1\leq j\leq s$, we may choose an affine function $l_j\colon U_\dC\to\dC$ such that $l_j(z_j)=0$ and that $l_j$ is nowhere vanishing on $A$. Then the holomorphic function $\nu\cdot\prod_{j=1}^sl_j$ satisfies the requirement in the corollary.
\end{proof}

\begin{remark}
In fact, from the proof of Proposition \ref{pr:rapid}, we may even require $\mu$ in Corollary \ref{co:rapid} and Corollary \ref{co:lattice} to have exponential decay on vertical strips. However, we do not need this in what follows.
\end{remark}

\subsection{Multiplier functions}
\label{ss:multiplier_functions}

In this subsection, we introduce the spaces of multiplier functions that will give multipliers of the convolution algebra of Schwartz functions. We fix a root datum $(\sfX^*,\Phi,\sfX_*,\Phi^\vee)$; and
\begin{itemize}
  \item let $\sfW$ be the Weyl group of $(\sfX^*,\Phi,\sfX_*,\Phi^\vee)$;

  \item let $\Aut(\sfX^*,\Phi,\sfX_*,\Phi^\vee)^\heartsuit$ be the set of automorphisms of $(\sfX^*,\Phi,\sfX_*,\Phi^\vee)$ of order at most $2$, which is finite and stable under the conjugation action of $\sfW$;

  \item put $\fh^*\coloneqq\sfX^*\otimes_\dZ\dR$;

  \item for every $\vartheta\in\Aut(\sfX^*,\Phi,\sfX_*,\Phi^\vee)^\heartsuit$, let $\fh^*_\vartheta$ (resp.\ $\fh^{*-}_\vartheta$) be the $+1$-eigenspace (resp.\ $-1$-eigenspace) of the action of $\vartheta$ on $\fh^*$; and let $\sfY^*_\vartheta$ be the projection of $\rho+\sfX^*\subseteq\fh^*$ onto $\fh^{*-}_\vartheta$, where $\rho$ is the half sum of positive roots in $\Phi$ with respect to an arbitrary base of $\Phi$.\footnote{It is easy to see that $\sfY^*_\vartheta$ does not depend on the choice of the base.}
\end{itemize}
It is clear that $\sfY^*_\vartheta$ is a translation of a discrete subgroup of $\fh^{*-}_\vartheta$. Though $\sfY^*_\vartheta$ is not necessarily a subgroup, in what follows, we will sometimes write $\sfH\oplus\sfY^*_\vartheta$ for a subgroup $\sfH$ of $\fh^*_{\vartheta,\dC}$ as the subset of $\fh^*_\dC$ consisting of elements of the form $\alpha+\varpi$ for $\alpha\in\sfH$ and $\varpi\in\sfY^*_\vartheta$.

We take a subset $\theta\subseteq\Aut(\sfX^*,\Phi,\sfX_*,\Phi^\vee)^\heartsuit$ that is stable under the conjugation action of $\sfW$. In the following definition, we introduce several important spaces of holomorphic functions on $\fh^*_\dC$ that will be related to multipliers of Schwartz convolution algebras.

\begin{definition}\label{de:multiplier}
We define several spaces of holomorphic functions on $\fh^*_\dC$.
\begin{enumerate}
  \item Define $\cM_\theta(\fh_\dC^*)$ to be the space of holomorphic functions $\mu$ on $\fh^*_\dC$ such that for every $\vartheta\in\theta$ and every $M>0$, there exists $r_{\vartheta,M}\in\dR$ such that
      \begin{align}\label{eq:growth1}
      \sup_{\substack{\alpha\in\fh^*_{\vartheta,\dC}, \|\RE\alpha\|<M \\ \varpi\in\sfY^*_\vartheta}}
      |\mu(\alpha+\varpi)|\cdot(1+\|\alpha+\varpi\|)^{r_{\vartheta,M}}<\infty
      \end{align}
      holds for some, hence every, Euclidean norm $\|\cdot\|$ on $\fh_\dC^*$.

  \item Define $\cN_\theta(\fh_\dC^*)$ to be the space of holomorphic functions $\mu$ on $\fh^*_\dC$ such that for every $\vartheta\in\theta$, every $M>0$, and every $r\in\dR$,
      \begin{align*}
      \sup_{\substack{\alpha\in\fh^*_{\vartheta,\dC}, \|\RE\alpha\|<M \\ \varpi\in\sfY^*_\vartheta}}
      |\mu(\alpha+\varpi)|\cdot(1+\|\alpha+\varpi\|)^r<\infty
      \end{align*}
      holds for some, hence every, Euclidean norm $\|\cdot\|$ on $\fh_\dC^*$.

  \item Define $\cM^\sharp_\theta(\fh_\dC^*)$ to be the subspace of $\cM_{\theta\cup\{1\}}(\fh_\dC^*)$ consisting of $\mu$ satisfying that for every $\vartheta\in\theta$, $\mu\res_{\fh^*_{\vartheta,\dC}+\varpi}=0$ for all but finitely many elements $\varpi\in\sfY^*_\vartheta$.
\end{enumerate}
When $\theta=\{1\}$, we suppress it in the subscripts in above notation.
\end{definition}

\if false

\begin{definition}\label{de:multiplier}
We define several spaces of functions on subsets of $\fh^*_\dC$.
\begin{enumerate}
  \item Define $\cM_\theta(\fh_\dC^*)$ to be the space of functions $\mu$ on $\bigcup_{\vartheta\in\theta}\fh^*_{\vartheta,\dC}\times\sfY^*_\vartheta$ such that for every $\vartheta\in\theta$,
      \begin{itemize}
        \item the restriction of $\mu$ on $\fh^*_{\vartheta,\dC}\times\sfY^*_\vartheta$ is holomorphic in the direction $\fh^*_{\vartheta,\dC}$;\footnote{When $\vartheta=1$, this simply means that $\mu$ is a holomorphic function on $\fh^*_\dC$.}

        \item for every $M>0$, there exists $r_M\in\dR$ such that
            \[
            \sup_{\substack{(\alpha,\varpi)\in\fh^*_{\vartheta,\dC}\times\sfY^*_\vartheta\\ \|\RE\alpha\|<M}}
            |\mu(\alpha,\varpi)|\cdot(1+\|(\alpha,\varpi)\|)^{r_M}<\infty
            \]
            holds for some, hence every, Euclidean norm $\|\cdot\|$ on $\fh_\dC^*$.\footnote{When $\vartheta=1$, this simply means that the holomorphic function $\mu$ on $\fh^*_\dC$ has moderate vertical growth (Definition \ref{de:moderate}).}
      \end{itemize}

  \item Define $\cN_\theta(\fh_\dC^*)$ to be the space of functions $\mu$ on $\bigcup_{\vartheta\in\theta}\fh^*_{\vartheta,\dC}\times\sfY^*_\vartheta$ such that for every $\vartheta\in\theta$,
      \begin{itemize}
        \item the restriction of $\mu$ on $\fh^*_{\vartheta,\dC}\times\sfY^*_\vartheta$ is holomorphic in the direction $\fh^*_{\vartheta,\dC}$;

        \item for every $M>0$ and every $r\in\dR$,
            \[
            \sup_{\substack{(\alpha,\varpi)\in\fh^*_{\vartheta,\dC}\times\sfY^*_\vartheta\\ \|\RE\alpha\|<M}}
            |\mu(\alpha,\varpi)|\cdot(1+\|(\alpha,\varpi)\|)^r<\infty
            \]
            holds for some, hence every, Euclidean norm $\|\cdot\|$ on $\fh_\dC^*$.\footnote{When $\vartheta=1$, this is usually referred to that the holomorphic function $\mu$ on $\fh^*_\dC$ has rapid decay on vertical strips.}
      \end{itemize}

  \item Define $\cM^\sharp_\theta(\fh_\dC^*)$ to be the space of holomorphic functions $\mu$ on $\fh^*_\dC$ having moderate vertical growth (Definition \ref{de:moderate}), and satisfying that for every $\vartheta\in\theta$, $\mu\res_{\fh^*_{\vartheta,\dC}\times\{\varpi\}}=0$ for all but finitely many elements $\varpi\in\sfY^*_\vartheta$.
\end{enumerate}
When $\theta=\{1\}$, we suppress it in the subscripts in above notation.
\end{definition}

\fi

\begin{remark}\label{re:multiplier}
We have the following concerning Definition \ref{de:multiplier}.
\begin{enumerate}
  \item In (1) (resp.\ (2)), for every $\vartheta\in\theta$ and every $\varpi\in\sfY^*_\vartheta$, the function $\alpha\mapsto\mu(\alpha+\varpi)$ has moderate vertical growth (Definition \ref{de:moderate}) (resp.\ has rapid decay on vertical strips) on $\fh^*_{\vartheta,\dC}$. On the other hand, for every $\vartheta\in\theta$, if we restrict the growth condition \eqref{eq:growth1} to $i\fh^*_\vartheta\oplus\sfY^*_\vartheta$, then it means that $\mu$ has polynomial growth.

  \item The spaces $\cM_\theta(\fh_\dC^*)$, $\cN_\theta(\fh_\dC^*)$, and $\cM^\sharp_\theta(\fh_\dC^*)$ are all closed under multiplication and the action of $\sfW$.

  \item If $\mu\in\cO(\fh_\dC^*)$ has moderate vertical growth, then $\mu\cdot\cM^\sharp_\theta(\fh_\dC^*)\subseteq \cM^\sharp_\theta(\fh_\dC^*)$. In particular, we have $\cM(\fh_\dC^*)\cdot\cM^\sharp_\theta(\fh_\dC^*)\subseteq \cM^\sharp_\theta(\fh_\dC^*)$. Note that $\cM(\fh_\dC^*)=\cM_{\{1\}}(\fh_\dC^*)$ according to Definition \ref{de:multiplier}.
\end{enumerate}
\end{remark}

The following lemma demonstrates the existence of elements in $\cM^\sharp_\theta(\fh_\dC^*)^\sfW$.

\begin{lem}\label{le:multiplier1}
For every element $\alpha_0\in\fh^*_\dC$, there exists an element $\mu\in\cM^\sharp_\theta(\fh_\dC^*)^\sfW$ such that $\mu(\alpha_0)\neq 0$.
\end{lem}

\begin{proof}
For each $\vartheta\in\theta$, we can find, by Corollary \ref{co:lattice}, a holomorphic function $\nu_\vartheta$ on $\fh^*_\dC$ that has moderate vertical growth and vanishes on $\fh^*_{\vartheta,\dC}+\varpi$ for all but finitely many elements $\varpi\in\sfY^*_\vartheta$, such that $\nu(w\alpha_0)\neq 0$ for every $w\in\sfW$. Now since $\theta$ is a finite set, we can take the product $\nu\coloneqq\prod_{\vartheta\in\theta}\nu_\vartheta\in\cM^\sharp_\theta(\fh_\dC^*)$, satisfying $\nu(w\alpha_0)\neq 0$ for every $w\in\sfW$. Put $\mu\coloneqq\prod_{w\in\sfW}\nu\circ w$. Then $\mu$ belongs to $\cM^\sharp_\theta(\fh_\dC^*)^\sfW$ and satisfies $\mu(\alpha_0)\neq 0$. The lemma follows.
\end{proof}

\subsection{Multipliers of Schwartz algebra}
\label{ss:multiplier_algebra}

We now consider a connected reductive algebraic group $\ul{G}$ over $\dR$. Let $(\sfX^*,\Phi,\sfX_*,\Phi^\vee)$ be the root datum associated to $\ul{G}_\dC$, namely, $\sfX^*$ and $\sfX_*$ are the weight and coweight lattices of the abstract Cartan group of $\ul{G}_\dC$, with $\Phi$ and $\Phi^\vee$ the subsets of roots and coroots of $\ul{G}_\dC$, respectively. We keep the notation in the previous subsection.

We let $\theta$ be the subset of $\Aut(\sfX^*,\Phi,\sfX_*,\Phi^\vee)^\heartsuit$ consisting of elements of the form $w\vartheta$ in which $w\in\sfW$ and $\vartheta\in\Aut(\sfX^*,\Phi,\sfX_*,\Phi^\vee)^\heartsuit$ is an element that induces the real form $\ul{G}$ of $\ul{G}_\dC$. Then $\theta$ is stable under the conjugation action of $\sfW$.

Denote by $\fg$ the complex Lie algebra of $\ul{G}_\dC$, and $\cU(\fg)$ the universal enveloping algebra of $\fg$ with the center $\cZ(\fg)$. By the Harish-Chandra isomorphism $\cZ(\fg)\simeq\dC[\fh^*_\dC]^\sfW$ \cite{HC51}, we obtain a character $\chi_\alpha$ of $\cZ(\fg)$ for every element $\alpha$ of $\fh^*_\dC$. Conversely, every character $\chi$ of $\cZ(\fg)$ gives rise to a $\sfW$-orbit in $\fh^*_\dC$, hence $\mu(\chi)$ is well-defined for an element $\mu\in\cO(\fh_\dC^*)^\sfW$.

Put $G\coloneqq\ul{G}(\dR)$. We fix a maximal compact subgroup $K$ of $G$, and a Haar measure $\r{d}g$ on $G$. Denote by $\dC[G]$ and $\cD[G]$ the complex algebras of algebraic functions and algebraic differential operators on $G$, respectively.

By an \emph{admissible representation} of $G$, we mean a smooth admissible Fr\'{e}chet representation of moderate growth of $G$ in the sense of Casselman--Wallach (\cite{Cas89}, \cite{Wal92}*{Section~11}). The category of admissible representations of $G$ is equivalent to the category of Harish-Chandra $(\fg,K)$-modules by the functor of taking $K$-finite vectors. For an irreducible admissible representation $\pi$ of $G$, we denote by $\chi_\pi\colon\cZ(\fg)\to\dC$ its infinitesimal character, which is identified with a $\sfW$-orbit in $\fh^*_\dC$.

We recall some definitions and facts from \cite{Wal83}*{Sections~2.5~\&~2.6}. A \emph{Schwartz function} on $G$ is a smooth function $f$ such that $Df$ is bounded for every $D\in\cD(G)$. Let $\cS(G)$ be the convolution algebra of Schwartz functions on $G$, equipped with its natural Fr\'{e}chet topology, under which the convolution product $\ast$, defined by the formula
\[
(f_1\ast f_2)(g)\coloneqq\int_{G}f_1(gh^{-1})f_2(h)\rd h,
\]
is continuous. For every admissible representation (or more generally Fr\'{e}chet representation of moderate growth) $(\pi,V_\pi)$ of $G$, the expression
\[
\pi(f)v\coloneqq\int_Gf(g)\pi(g)v\rd g
\]
is absolutely convergent for every $f\in\cS(G)$ and $v\in V_\pi$, hence defines a continuous operator $\pi(f)\in\End(V_\pi)$.

\begin{remark}
We have $\cS(G)\subseteq\bigcap_{p>0}\cC^p(G)$ where $\cC^p(G)$ denotes the convolution algebra of Harish-Chandra $L^p$-Schwartz functions on $G$; the inclusion is an equality if and only if the center of $G$ is compact.
\end{remark}

\begin{remark}\label{re:sobolev}
Let $L^2(G)$ be the $L^2$-space of $G$. Using the Sobolev lemma, it is easy to see that $\cS(G)$ is also the space of $f\in L^2(G)$ such that $Df\in L^2(G)$ for every $D\in\cD(G)$, where $Df$ is understood in the sense of distributions. Moreover, the Fr\'echet topology of $\cS(G)$ is also induced from the semi-norms $f\mapsto\|Df\|_{L^2}$ for all $D\in\cD(G)$.
\end{remark}

The following theorem provides many multipliers of the algebra $\cS(G)$.

\begin{theorem}\label{th:multiplier}
For every element $\mu\in\cM_{\theta\cup\{1\}}(\fh_\dC^*)^\sfW$, there is a unique linear operator
\[
\mu\star\colon\cS(G)\to\cS(G),
\]
such that
\[
\pi(\mu\star f)=\mu(\chi_\pi)\cdot\pi(f)
\]
holds for every $f\in\cS(G)$ and every irreducible admissible representation $\pi$ of $G$. In particular, $\mu\star\in\Mul(\cS(G))$ is a multiplier of $\cS(G)$.
\end{theorem}

The proof of this theorem will be given in the next subsection.

\begin{remark}\label{re:theta}
The subset $\theta$ defined above only determines the inner form class of $\ul{G}$. We may define a refined invariant $\theta_{\ul{G}}$ associated to $\ul{G}$ to be the subset of $\Aut(\sfX^*,\Phi,\sfX_*,\Phi^\vee)^\heartsuit$ of elements induced from all maximal tori $\ul{T}$ of $\ul{G}$ in the way described in Lemma \ref{le:multi_0} below, which is contained in $\theta$ and stable under the conjugation action of $\sfW$. Note that
\begin{itemize}
  \item $\theta_{\ul{G}}=\theta$ if and only if $\ul{G}$ is quasi-split;

  \item $\theta_{\ul{G}}$ contains $1$ if and only if $\ul{G}$ is split;

  \item $\theta_{\ul{G}}$ contains $-1$ if and only if $\ul{G}$ admits an anisotropic maximal torus;

  \item $\theta_{\ul{G}}=\{-1\}$ if and only if $\ul{G}$ is anisotropic.
\end{itemize}
\end{remark}

\begin{example}\label{ex:multiplier}
Suppose that $\ul{G}_\dC=\SL_{2,\dC}$. In this case, we may identity $\sfX^*$ with $\dZ$ under which $\Phi=\{2,-2\}$, hence $\sfW=\Aut(\sfX^*,\Phi,\sfX_*,\Phi^\vee)=\{\pm1\}$. There are two cases of $\ul{G}$ up to isomorphism.
\begin{enumerate}
  \item When $\ul{G}$ is split, we have $\theta=\theta_{\ul{G}}=\{\pm 1\}$. In this case, $\cM_{\theta\cup\{1\}}(\fh_\dC^*)^\sfW=\cM_{\theta_{\ul{G}}}(\fh_\dC^*)^\sfW=\cM_{\{\pm1\}}(\fh_\dC^*)^\sfW$, which consists of even holomorphic functions on $\dC$ that have moderate vertical growth, and have polynomial growth on the subset $\dZ\subseteq\dC$.

  \item When $\ul{G}$ is anisotropic, we have $\theta_{\ul{G}}=\{-1\}\subseteq\theta=\{\pm 1\}$. In this case, $\cM_{\theta_{\ul{G}}}(\fh_\dC^*)^\sfW=\cM_{\{-1\}}(\fh_\dC^*)^\sfW$, which consists of even holomorphic functions on $\dC$ that have polynomial growth on the subset $\dZ\subseteq\dC$.
\end{enumerate}
\end{example}

\begin{remark}\label{re:conjecture}
It is natural to ask when a holomorphic function $\mu$ on $\fh^*_\dC$ gives a multiplier as in Theorem \ref{th:multiplier}. We conjecture that $\mu$ does this if and only if it belongs to $\cM_{\theta_{\ul{G}}}(\fh_\dC^*)^\sfW$; in particular, the conjecture gives rise to a homomorphism $\cM_{\theta_{\ul{G}}}(\fh_\dC^*)^\sfW\to\Mul(\cS(G))$ of complex algebras. By a result of Harish-Chandra, the subset of $\fh_\dC^*$ of all infinitesimal characters of irreducible admissible representations of $G$ is the union $\bigcup_{\vartheta\in\theta_{\ul{G}}}\fh^*_{\vartheta,\dC}\oplus\sfY^*_\vartheta$. Thus, the kernel of the previous homomorphism consists exactly of those $\mu$ that vanish on $\bigcup_{\vartheta\in\theta_{\ul{G}}}\fh^*_{\vartheta,\dC}\oplus\sfY^*_\vartheta$.

This conjecture can be easily checked when $\ul{G}$ is anisotropic, that is, when $G$ is compact. On the other hand, when $\ul{G}$ is split, Theorem \ref{th:multiplier} implies the existence of the homomorphism $\cM_{\theta_{\ul{G}}}(\fh_\dC^*)^\sfW\to\Mul(\cS(G))$, as in this case $\theta\cup\{1\}=\theta=\theta_{\ul{G}}$ (Remark \ref{re:theta}). We also remark that when $\ul{G}=\SL_{2,\dR}$, this conjecture was known as a consequence of the work \cite{Bar88}.
\end{remark}

\begin{remark}\label{re:multiplier_map}
Theorem \ref{th:multiplier} together with Definition \ref{de:multiplier}(3) provide us with a homomorphism
\[
\cM^\sharp_\theta(\fh_\dC^*)^\sfW\to\Mul(\cS(G))
\]
of complex algebras.
\end{remark}

To end this subsection, we record a property for elements in $\cM^\sharp_\theta(\fh_\dC^*)^\sfW$ that will be used later. As usual, by a \emph{$K$-type}, we mean an isomorphism class of irreducible smooth representations of $K$.

\begin{lem}\label{le:multiplier0}
Let $\mu$ be an element in $\cM^\sharp_\theta(\fh_\dC^*)^\sfW$. Then there is a finite set $\fT(\mu)$ of $K$-types such that for every irreducible admissible representation $\pi$ of $G$ satisfying $\mu(\chi_\pi)\neq 0$, we have that $\pi\res_K$ contains some member from $\fT(\mu)$.
\end{lem}

\begin{proof}
By Definition \ref{de:multiplier}(3) together with Harish-Chandra's description of the infinitesimal characters of discrete series of Levi subgroups of $G$ \cite{HC66}, there are only finitely many pairs $(M,\sigma)$ where $M$ is a Levi subgroup of $G$ and $\sigma$ is an irreducible discrete series representation of $M$, up to conjugation and unramified twists, such that $\mu$ is nonvanishing on the infinitesimal character of the parabolic induction of $(M,\sigma)$. By the Langlands classification and the fact that for a standard module, the Langlands quotient inherits its minimal $K$-types \cite{Kna01}*{Theorem~15.10}, we can take $\fT(\mu)$ to be the minimal $K$-types of the parabolic induction of those finitely many pairs $(M,\sigma)$.
\end{proof}

%This is a consequence of \cite{DS04}*{Th\'{e}or\`{e}me~3(ii)}. In fact, by Definition \ref{de:moderate}(3) together with Harish-Chandra's description of the infinitesimal characters of the discrete series for Levi subgroups of $G$, there are only finitely many pairs $(Q_0,\delta_0)$ (in the notation of \cite{DS04}*{Th\'{e}or\`{e}me~3}) such that $\mu$ is nonzero on (the infinitesimal character of) $I^{Q_0}_{\delta_0,\lambda_0}$ for some $\lambda_0$. Let $l$ be a positive integer that is greater than the length of the minimal $K$-types of $I^{Q_0}_{\delta_0,\lambda_0}$ (which depend only on $(Q_0,\delta_0)$) for those pairs $(Q_0,\delta_0)$ as above. Then we can take $\fT(\mu)$ to be the set of $K$-types of length at most $l$, which is a finite set. The lemma follows.

\subsection{Proof of Theorem \ref{th:multiplier}}
\label{ss:multiplier_proof}

In this subsection, we prove Theorem \ref{th:multiplier}.

We fix a $\sfW$-invariant Euclidean norm $\|\cdot\|$ on $\fh_\dC^*$. For $\vartheta\in\Aut(\sfX^*,\Phi,\sfX_*,\Phi^\vee)^\heartsuit$, $M>0$, and $r\in\dR$, we put
\[
p_{\vartheta,M,r}(\mu)\coloneqq\sup_{\substack{\alpha\in\fh^*_{\vartheta,\dC}, \|\RE\alpha\|<M \\ \varpi\in\sfY^*_\vartheta}}
|\mu(\alpha+\varpi)|\cdot(1+\|\alpha+\varpi\|)^r
\]
for $\mu\in\cO(\fh_\dC^*)$. Clearly, we have $p_{\vartheta,M,r}(\mu)\leq p_{\vartheta,M',r}(\mu)$ if $M\leq M'$ and $p_{\vartheta,M,r}(\mu)\leq p_{\vartheta,M,r'}(\mu)$ if $r\leq r'$.

We start from some discussion on maximal tori of $\ul{G}$. Let $H$ be the abstract Cartan group of $\ul{G}_\dC$, whose weight lattice is $\sfX^*$. Recall that the abstract Cartan group $H$ of $\ul{G}_\dC$ represents the presheaf on the opposite category of complex tori whose value on (a complex torus) $S$ is the set of collections $\{q_B\colon B\to S\}$ of homomorphisms for every Borel subgroup $B$ of $\ul{G}_\dC$ satisfying that $q_B=q_{B'}\circ\alpha_{B,B'}$ for every inner automorphism $\alpha_{B,B'}$ of $\ul{G}_\dC$ mapping $B$ into $B'$. In particular, for every maximal torus $T$ of $\ul{G}$, there is a canonical $\sfW$-conjugacy class of isomorphisms $\ul{T}_\dC\xrightarrow{\sim}H$ from the universal object, which we refer as universal isomorphisms.

Let $\ul{T}$ be a maximal torus of $\ul{G}$. Put $T\coloneqq\ul{T}(\dR)$. Then $T$ admits a decomposition $T=A_TT_c$ in which $A_T$ and $T_c$ are the maximal split and the anisotropic (analytic) sub-tori, respectively; it induces a decomposition $\ft=\fa_T\oplus\ft_c$ for the corresponding Lie algebras.

\begin{lem}\label{le:multi_0}
For every universal isomorphism $\ul{T}_\dC\xrightarrow{\sim}H$, which induces a decomposition $\fh^*=\fa^*_T\oplus i\ft_c^*$, there is a unique element $\vartheta\in\theta$ such that $\fa_T^*=\fh^*_\vartheta$ and $i\ft_c^*=\fh^{*-}_\vartheta$.
\end{lem}

\begin{proof}
We define an automorphism $\vartheta_{\ul{T}}$ of $\ul{T}$ as follows: Let $d$ be the rank of $\ul{T}$ and fix an isomorphism $\iota_{\ul{T}}\colon\ul{T}_\dC\xrightarrow{\sim}\bG_{m,\dC}^d$. Then we set $\vartheta_{\ul{T}}\coloneqq\ol{\iota_{\ul{T}}}^{-1}\circ\iota_{\ul{T}}$, where $\ol{\iota_{\ul{T}}}$ denotes the complex conjugate of $\iota_{\ul{T}}$.

It is easy to see that $\vartheta_{\ul{T}}$ does not depend on the choice of $\iota_{\ul{T}}$, and moreover that on the Lie algebra $\ft$ of $\ul{T}$, $\vartheta_{\ul{T}}$ acts by $+1$ on $\fa_T$ and by $-1$ on $\ft_c$. Take a universal isomorphism $\ul{T}_\dC\xrightarrow{\sim}H$, which gives rise to a decomposition $\fh^*=\fa^*_T\oplus i\ft_c^*$. Then $\vartheta_{\ul{T}}$ induces an element $\vartheta\in\Aut(\sfX^*,\Phi,\sfX_*,\Phi^\vee)$ that belongs to $\theta$, which satisfies $\fa_T^*=\fh^*_\vartheta$ and $i\ft_c^*=\fh^{*-}_\vartheta$. The lemma is proved as the uniqueness is clear.
\end{proof}

\if false

\begin{lem}\label{le:multi_1}
There exists a $\dC$-linear involution $\vartheta$ of $\fh_\dC^*$ that normalizes the action of $\sfW$, such that the following hold:
\begin{enumerate}
  \item For every maximal torus $\ul{T}$ of $\ul{G}_\theta$ and every $\theta$-admissible decomposition $\fh_\dC^*=\fa_{T,\dC}^*\oplus\ft_{c,\dC}^*$, there exists an element $w\in\sfW$ with $(w\vartheta)^2=1$ such that $\fa_{T,\dC}^*=\Ker(w\vartheta-1)$ and $\ft_{c,\dC}^*=\Ker(w\vartheta+1)$.

  \item For every element $w\in\sfW$ with $(w\vartheta)^2=1$, we can find a maximal torus $\ul{T}$ of $\ul{G}_\theta$ and an $\theta$-admissible decomposition $\fh_\dC^*=\fa_{T,\dC}^*\oplus\ft_{c,\dC}^*$ under which $\fa_{T,\dC}^*=\Ker(w\vartheta-1)$ and $\ft_{c,\dC}^*=\Ker(w\vartheta+1)$.
\end{enumerate}
\end{lem}

\begin{proof}
Let $\ul{T}$ be an algebraic torus defined over $\dR$. We define an automorphism $\vartheta_{\ul{T}}$ of $\ul{T}$ as follows: let $d$ be the rank of $\ul{T}$ and fix an isomorphism $\iota_{\ul{T}}\colon\ul{T}_\dC\xrightarrow{\sim}\bG_{m,\dC}^d$. Then we set $\vartheta_{\ul{T}}\coloneqq\ol{\iota_{\ul{T}}}^{-1}\circ\iota_{\ul{T}}$, where $\ol{\iota_{\ul{T}}}$ denotes the complex conjugate of $\iota_{\ul{T}}$. It is easy to see that $\vartheta_{\ul{T}}$ does not depend on the choice of $\iota_{\ul{T}}$, and moreover that on the Lie algebra $\ft_\dC$, $\vartheta_{\ul{T}}$ acts trivially on $\fa_{T,\dC}$ and by $-1$ on $\ft_{c,\dC}$.

To prove the lemma, we may fix a maximal torus $\ul{S}$ of $\ul{G}_\theta$ and identify $\sfX^*$ with the weight lattice of $\ul{S}_\dC$. We let $\vartheta$ be the induced action of $\vartheta_{\ul{S}}$ on $\fh_\dC^*$. As $\vartheta\circ w=\overline{w}\circ\vartheta$ for every $w\in\sfW$, where $w\mapsto\ol{w}$ is the complex conjugation on $\sfW$, we see that $\vartheta$ normalizes the action of $\sfW$.

For (1), let $\ul{T}$ be a maximal torus of $\ul{G}_\theta$ and choose an element $g\in\ul{G}_\theta(\dC)$ such that $\Ad(g)(\ul{T}_\dC)=\ul{S}_\dC$. Put $\iota\coloneqq\Ad(g)\colon\ul{T}_\dC\xrightarrow{\sim}\ul{S}_\dC$. Then there exists an element $w\in\sfW$ such that the actions of $w$ and $\iota\circ\overline{\iota}^{-1}$ on $\ul{S}_\dC$ coincide. Choose an isomorphism $\iota_{\ul{T}}\coloneqq\ul{T}_\dC\xrightarrow{\sim}\bG_{m,\dC}^d$ and put $j_{\ul{T}}\coloneqq\iota_{\ul{T}}\circ\iota^{-1}$. We have
\[
\iota\circ\vartheta_{\ul{T}}\circ\iota^{-1}=\iota\circ\overline{\iota_{\ul{T}}}^{-1}\circ\iota_{\ul{T}}\circ\iota^{-1}
=\iota\circ\overline{\iota}^{-1}\circ\overline{j_{\ul{T}}}^{-1}\circ j_{\ul{T}}=w\circ\vartheta_{\ul{S}},
\]
which implies (1).

For (2), the condition $(w\vartheta)^2=1$ is equivalent to $w\overline{w}=1$, hence $w$ composed with the complex conjugation induces a real structure on the complex vector space $\fh_\dC$. In particular, there exists a regular element $X\in\fh_\dC$ such that $\overline{X}=w^{-1}X$. This implies that the conjugacy class of $X$ is defined over $\dR$, hence by \cite{Ste65}*{Theorem~9.8}, there exists an element $g\in\ul{G}_\theta(\dC)$ such that $Y\coloneqq\ad(g)^{-1}X$ belongs to $\Lie\ul{G}_\theta$. Let $\ul{T}$ be the centralizer of $Y$ in $\ul{G}_\theta$, which is a maximal torus of $\ul{G}_\theta$. Then $\Ad(g)$ sends $\ul{T}_\dC$ to $\ul{S}_\dC$; and if we put $\iota\coloneqq\Ad(g)\colon\ul{T}_\dC\xrightarrow{\sim}\ul{S}_\dC$, then we have $\iota_*Y=X$ and $\overline{\iota}_*Y=w^{-1}X$. As $X$ is regular, this further implies that the actions of $w$ and $\iota\circ\overline{\iota}^{-1}$ on $\ul{S}_\dC$ coincide. By the above discussion for (1), we obtain (2).

The lemma is proved.
\end{proof}

\fi

As we have explained in Subsection \ref{ss:strategy}, the key step is to bound $\|D(\mu\star f)\|_{L^2}$ when $D\in\cD(G)$ is a polynomial, which boils down to the study of certain linear operators on $\cN_{\theta\cup\{1\}}(\fh_\dC^*)^\sfW$. Those operators turn out to be from the algebra $\sS$ below, which we study now.

Denote by $\wt\sfW\coloneqq\sfX^*\rtimes\sfW$ the extended affine Weyl group, and $\wt\sfW^\heartsuit\subseteq\wt\sfW$ the subset of reflections. Recall that an element $s\in\wt\sfW$ is a \emph{reflection} if the locus $H_s$ of fixed points of $s$ on $\fh_\dC^*$ is an affine hyperplane. For every $s\in\wt\sfW^\heartsuit$, we fix an affine function $\ell_s$ on $\fh_\dC^*$ with zero locus $H_s$.

For every $w\in\wt\sfW$, we denote by $t_w$ the operator on $\dC[\fh_\dC^*]$ given by $(t_wP)(\alpha)=P(w^{-1}\alpha)$. Let $\sR$ be the algebra of endomorphisms of $\dC[\fh_\dC^*]$ generated by $t_w$ for all $w\in\wt\sfW$ and multiplications by elements of $\dC[\fh_\dC^*]$. We consider $\sR$ as a left $\dC[\fh_\dC^*]$-module in the obvious way. Put
\begin{align*}
\sS&\coloneqq\dC[\fh_\dC^*][\ell_s^{-1}\res s\in\wt\sfW^\heartsuit]\otimes_{\dC[\fh_\dC^*]}\sR, \\
\sT&\coloneqq\dC(\fh_\dC^*)\otimes_{\dC[\fh_\dC^*]}\sR,
\end{align*}
where $\dC(\fh_\dC^*)$ denotes the function field of $\dC[\fh_\dC^*]$. We have natural maps $\sR\to\sS\to\sT$ of algebras. The algebra $\sS$ is independent of the choices of $\{\ell_s\res s\in\wt\sfW^\heartsuit\}$. The algebra $\sT$ can be considered as the algebra of endomorphisms of $\dC(\fh_\dC^*)$ generated by $t_w$ for all $w\in\wt\sfW$ and multiplications by elements of $\dC(\fh_\dC^*)$.

The following lemma tells us when an element $S\in\sS$ extends to a linear operator on $\cN_{\theta\cup\{1\}}(\fh_\dC^*)$ and how it behaves with respect to the family of semi-norms $\{p_{\vartheta,M,r}\}$.

\begin{lem}\label{le:multi_3}
Let $S\in\sS$ be an element satisfying $S\dC[\fh_\dC^*]\subset\dC[\fh_\dC^*]$. Then $S$ extends uniquely to a continuous endomorphism of $\cO(\fh_\dC^*)$ with respect to the topology of uniform convergence on compact subsets, which we still denote by $S$. Moreover, there exist $M_S>0$ and $r_S>0$ such that for every $M>0$ and every $r\in\dR$, there exists $C_{S,M,r}>0$ such that
\begin{align*}
p_{1,M,r}(S\mu)&\leq C_{S,M,r}\cdot p_{1,M+M_S,r+r_S}(\mu),\\
\max_{\vartheta\in\theta}p_{\vartheta,M,r}(S\mu)&\leq C_{S,M,r}\cdot\max_{\vartheta\in\theta}p_{\vartheta,M+M_S,r+r_S}(\mu)
\end{align*}
hold for every $\mu\in\cO(\fh_\dC^*)$. In particular,
\begin{enumerate}
  \item $S$ preserves the subspace $\cN_{\theta\cup\{1\}}(\fh_\dC^*)$;

  \item if $S\dC[\fh_\dC^*]\subset\dC[\fh_\dC^*]^\sfW$, then $S$ preserves the subspace $\cN_{\theta\cup\{1\}}(\fh_\dC^*)^\sfW$.
\end{enumerate}
\end{lem}

\begin{proof}
Note that for the topology of uniform convergence on compact subsets, the subspace $\dC[\fh_\dC^*]$ is dense in $\cO(\fh_\dC^*)$, hence the uniqueness of the extension is clear.

Now we show the existence of the extension, with the estimate on $p_{\vartheta,M,r}(S\mu)$ together. By definition, $S$ is of the form $\Theta^{-1}S'$, where $S'\in\sR$ and $\Theta$ is a finite product (possibly with multiplicities) of $\ell_s$ for $s\in\wt\sfW^\heartsuit$. It is immediate that the action of $S'$ on $\dC[\fh_\dC^*]$ extends by continuity to an action on $\cO(\fh_\dC^*)$ with the following property: there exist $M_{S'}>0$ and $r_{S'}>0$ such that for every $M>0$ and every $r\in\dR$, there exists $C_{S',M,r}>0$ such that
\begin{align*}
p_{1,M,r}(S'\mu)&\leq C_{S',M,r}\cdot p_{1,M+M_{S'},r+r_{S'}}(\mu),\\
\max_{\vartheta\in\theta}p_{\vartheta,M,r}(S'\mu)&\leq C_{S',M,r}\cdot\max_{\vartheta\in\theta}p_{\vartheta,M+M_{S'},r+r_{S'}}(\mu)
\end{align*}
hold for every $\mu\in\cO(\fh_\dC^*)$. In particular, $S'$ preserves the subspace $\cN_{\theta\cup\{1\}}(\fh_\dC^*)$. Now the similar properties for $S$ follow from Lemma \ref{le:reflection} below.

For the last two claims, (1) is an immediate consequence of the estimate on $p_{\vartheta,M,r}(S\mu)$; and (2) follows from the uniqueness of the extension.
\end{proof}

\begin{lem}\label{le:reflection}
For every $s\in\wt\sfW^\heartsuit$, there exists $M_s>0$ such that for every $M>0$ and every $r\in\dR$, there exists $C_{s,M,r}>0$ such that
\begin{align*}
p_{1,M,r}(\ell_s^{-1}\mu)&\leq C_{s,M,r}\cdot p_{1,M+M_s,r}(\mu),\\
\max_{\vartheta\in\theta}p_{\vartheta,M,r}(\ell_s^{-1}\mu)&\leq C_{s,M,r}\cdot\max_{\vartheta\in\theta}p_{\vartheta,M+M_s,r}(\mu)
\end{align*}
hold for every $\mu\in\cO(\fh_\dC^*)$ that vanishes on $H_s$. In particular, for a (holomorphic) function $\mu\in\cN_{\theta\cup\{1\}}(\fh_\dC^*)$ that vanishes on $H_s$, the function $\ell_s^{-1}\mu$ belongs to $\cN_{\theta\cup\{1\}}(\fh_\dC^*)$ as well.
\end{lem}

\begin{proof}
For every fixed $\vartheta\in\theta\cup\{1\}$, we claim that there exists $M_s>0$ such that for every $M>0$ and $r\in\dR$, there exists $C_{s,M,r}>0$ such that
\begin{align*}
p_{\vartheta,M,r}(\ell_s^{-1}\mu)\leq
\begin{dcases}
C_{s,M,r} \cdot p_{\vartheta,M+M_s,r}(\mu), &\text{if $\vartheta=1$}\\
C_{s,M,r}\cdot\max_{\vartheta'\in\theta}p_{\vartheta',M+M_s,r}(\mu), &\text{if $\vartheta\neq 1$}
\end{dcases}
\end{align*}
holds for every $\mu\in\cO(\fh_\dC^*)$ that vanishes on $H_s$. The lemma then follows from this claim.

To prove the claim, let $\ell^\circ_s$ be the linear part of $\ell_s$. Take a function $\mu\in\cN_{\theta\cup\{1\}}(\fh_\dC^*)$ that vanishes on $H_s$.

First suppose that $\ell^\circ_s\res_{\fh^*_{\vartheta,\dC}}\neq 0$. Then the claim is an easy consequence of Cauchy's integral formula. Indeed, take an element $\alpha_s\in\fh^*_{\vartheta,\dC}$ such that $\ell^\circ_s(\alpha_s)=1$. When $|\ell_s(\alpha+\varpi)|>1/2$, we have
\[
\left|\frac{\mu(\alpha+\varpi)}{\ell_s(\alpha+\varpi)}\right|\leq 2|\mu(\alpha+\varpi)|;
\]
whereas when $|\ell_s(\alpha+\varpi)|\leq 1/2$, we have
\[
\left|\frac{\mu(\alpha+\varpi)}{\ell_s(\alpha+\varpi)}\right|
=\left|\int_{u\in\dC^\times,|u|=1}
\frac{\mu(\alpha+u\alpha_s+\varpi)}{\ell_s(\alpha+u\alpha_s+\varpi)}\rd u\right|
\leq 2\sup_{u\in\dC^\times,|u|=1}|\mu(\alpha+u\alpha_s+\varpi)|.
\]
Together, we obtain
\begin{align}\label{eq:multi_0}
\left|\frac{\mu(\alpha+\varpi)}{\ell_s(\alpha+\varpi)}\right|\leq
2\sup_{|u|\leq 1}\|\mu(\alpha+u\alpha_s+\varpi)\|.
\end{align}
Choose $C_{s,M,r}>0$ so that
\[
2(1+\|\alpha+\varpi\|)^r\leq C_{s,M,r}\cdot\min_{u\in\dC^\times,|u|=1}\(1+\|\alpha+u\alpha_s+\varpi\|\)^r
\]
holds for every $\alpha\in\fh^*_{\vartheta,\dC}$ with $\|\RE\alpha\|<M$ and every $\varpi\in\sfY^*_\vartheta$. Then we have
\[
p_{\vartheta,M,r}(\ell_s^{-1}\mu)\leq C_{s,M,r} \cdot p_{\vartheta,M+M_s,r}(\mu)
\]
with $M_s\coloneqq\|\alpha_s\|$ depending only on $s$. Thus, the claim holds.

Now suppose that $\ell^\circ_s\res_{\fh^*_{\vartheta,\dC}}=0$, hence $\vartheta\neq 1$. Since the image of $\fh^*_{\vartheta,\dC}\oplus\sfY^*_\vartheta$ under $\ell_s$ is a discrete subset of $\dC$, there exists $C_s>0$ such that $|\ell_s|>C_s$ on $(\fh^*_{\vartheta,\dC}\oplus\sfY^*_\vartheta)\setminus H_s$. It follows that
\begin{align}\label{eq:multi_6}
\sup_{\substack{\alpha\in\fh^*_{\vartheta,\dC},\|\RE\alpha\|<M \\ \varpi\in\sfY^*_\vartheta, \alpha+\varpi\not\in H_s }}
\left|\frac{\mu(\alpha+\varpi)}{\ell_s(\alpha+\varpi)}\right|\cdot(1+\|\alpha+\varpi\|)^r
\leq C_s^{-1}\cdot p_{\vartheta,M,r}(\mu).
\end{align}
It remains to bound $\ell_s^{-1}\mu$ on vertical strips of $\fh^*_{\vartheta,\dC}\oplus(\sfY^*_\vartheta\cap H_s)=(\fh^*_{\vartheta,\dC}\oplus\sfY^*_\vartheta)\cap H_s$. Let $w_s\in\wt\sfW^\heartsuit\cap\sfW$ be the reflection that is the linear part of $s$. As $\ell^\circ_s\res_{\fh^*_{\vartheta,\dC}}=0$, the actions of $w_s$ and $\vartheta$ on $\fh_\dC^*$ commute. Thus, the element $\vartheta'\coloneqq w_s\vartheta$ satisfies $\vartheta'^2=1$, hence belongs to $\theta$. However for $\vartheta'$, we have $\fh^{*-}_{\vartheta',\dC}=\fh^{*-}_{\vartheta,\dC}\cap\Ker\ell^\circ_s$ and $\fh^*_{\vartheta,\dC}\subseteq\fh^*_{\vartheta',\dC}$.

Since the map $\ell_s^\circ$ induces an isomorphism $\fh^*_{\vartheta',\dC}\cap\fh^{*-}_{\vartheta,\dC}\xrightarrow{\sim}\dC$, there exists an element $\alpha'_s\in\fh^*_{\vartheta',\dC}$ such that $\varpi'\coloneqq\varpi-\alpha'_s\in\fh^{*-}_{\vartheta',\dC}$ for every $\varpi\in\sfY^*_\vartheta\cap H_s$. Since $\ell^\circ_s\res_{\fh^*_{\vartheta',\dC}}\neq 0$, we obtain from \eqref{eq:multi_0} that there exists $\alpha_s\in\fh^*_{\vartheta',\dC}$ such that
\[
\left|\frac{\mu(\alpha+\varpi)}{\ell_s(\alpha+\varpi)}\right|\leq 2\sup_{|u|\leq 1}|\mu(\alpha+u\alpha_s+\alpha'_s+\varpi')|
\]
holds for $\alpha+\varpi\in(\fh^*_{\vartheta,\dC}\oplus\sfY^*_\vartheta)\cap\Ker\ell_s$. Note that the projection of $\sfY^*_\vartheta$ onto $\fh^{*-}_{\vartheta',\dC}$ coincides with $\sfY^*_{\vartheta'}$. Thus, as in the previous case, we may find a constant $C_{s,M,r}\geq C_s^{-1}$ such that
\begin{align*}
\sup_{\substack{\alpha\in\fh^*_{\vartheta,\dC},\|\RE\alpha\|<M \\ \varpi\in\sfY^*_\vartheta, \alpha+\varpi\in H_s }}
\left|\frac{\mu(\alpha+\varpi)}{\ell_s(\alpha+\varpi)}\right|\cdot(1+\|\alpha+\varpi\|)^r
\leq C_{s,M,r}\cdot p_{\vartheta',M+M_s,r}(\mu)
\end{align*}
with $M_s\coloneqq \|\alpha_s\|+\|\alpha'_s\|$. Thus, the claim holds after combining with \eqref{eq:multi_6}.
\end{proof}

Now we relate the algebra $\sS$ to finite dimensional algebraic representations of $G$. Let $(\tau,W_\tau)$ be a finite dimensional algebraic representation of $G$, and
\[
\delta\colon\cU(\fg)\to\cU(\fg)\otimes\End(W_\tau)
\]
the homomorphism that sends $X\in\fg$ to $X\otimes 1+1\otimes \tau(X)$. Let $\cR_\tau$ be the centralizer of the image of $\delta$. By \cite{Kos75}*{Theorem~4.8}, $\cR_\tau$ is a finitely generated free $\cZ(\fg)$-module, where $\cZ(\fg)$ acts on $\cU(\fg)\otimes\End(W_\tau)$ by $z.(u\otimes A)=(zu)\otimes A$.  Recall that we have identified $\cZ(\fg)$ with $\dC[\fh_\dC^*]^\sfW$ via the Harish-Chandra isomorphism.

\begin{lem}\label{le:multi_4}
For every basis $v_1,\dots,v_r$ of the finite free $\cZ(\fg)$-module $\cR_\tau$, there exist elements $S_1,\dots,S_r\in\sS$ sending $\dC[\fh_\dC^*]$ to $\dC[\fh_\dC^*]^\sfW$ such that
\[
\delta(z)=\sum_{i=1}^r S_i(z)v_i
\]
for every $z\in\cZ(\fg)$.
\end{lem}

\begin{proof}
The basis $v_1,\dots,v_r$ induces an isomorphism $\cR_\tau\simeq \cZ(\fg)^{\oplus r}$, through which the right multiplication by $\delta(z)$ for $z\in \cZ(\fg)$ is represented by a matrix $S(z)=(S_{i,j}(z))_{1\leq i,j\leq r}\in\Mat_r(\c\cZ(\fg))$. By Lemma \ref{le:multi_2} below, it suffices to show that for every $1\leq i,j\leq r$, the map $z\in\cZ(\fg)\mapsto S_{i,j}(z)$ is induced from an element of $\sT$ sending $\dC[\fh_\dC^*]$ to $\dC[\fh_\dC^*]^\sfW$.

Let $\Lambda(\tau)\subset\sfX^*$ be the set of weights of $\tau$. For every $z\in\cZ(\fg)$, put
\[
P_z(X)\coloneqq\prod_{\lambda\in \Lambda(\tau)}(X-t_\lambda z),
\]
which belongs to $\cZ(\fg)[X]=\dC[\fh_\dC^*]^\sfW[X]$, as $\Lambda(\tau)$ is $\sfW$-invariant. According to \cite{Kos75}*{Theorem~4.9}, we have $P_z(\delta(z))=0$, or equivalently $P_z(S(z))=0$ for every $z\in \cZ(\fg)$. As the characters $z\in \cZ(\fg)\mapsto t_\lambda z\in\dC[\fh_\dC^*]$ for $\lambda\in \Lambda(\tau)$ are all distinct, there exists an element $P\in\GL_r(\dC(\fh_\dC^*))$ and a family of weights $\lambda_1,\ldots,\lambda_r\in\Lambda(\tau)$ such that
\[
S(z)=P\begin{pmatrix} t_{\lambda_1}z \\ & \ddots \\ & & t_{\lambda_r}z\end{pmatrix} P^{-1}
\]
holds in $\Mat_r(\dC(\fh_\dC^*))$ for every $z\in\cZ(\fg)$. This shows that for every $1\leqslant i,j\leqslant r$, the map $S_{i,j}$ is induced from an element $S'_{i,j}\in\sT$ sending $\dC[\fh_\dC^*]^\sfW$ to $\dC[\fh_\dC^*]^\sfW$. However, up to replacing $S'_{i,j}$ by $|\sfW|^{-1}\sum_{w\in\sfW}S'_{i,j}t_w$, we see that $S'_{i,j}$ also sends $\dC[\fh_\dC^*]$ to $\dC[\fh_\dC^*]^\sfW$.

The lemma is proved.
\end{proof}

\begin{lem}\label{le:multi_2}
If an element $T\in\sT$ satisfies $T\dC[\fh_\dC^*]\subset \dC[\fh_\dC^*]$, then $T$ is contained in the image of $\sS\to\sT$.
\end{lem}

This lemma is essentially \cite{KK86}*{Lemma~4.9}. For readers' convenience, we give a detailed proof.

\begin{proof}
The lemma amounts to the following: For an irreducible polynomial $P\in\dC[\fh_\dC^*]$, and a family $(P_w)_{w\in\wt\sfW}$ of elements of $\dC[\fh_\dC^*]$ with $P_w=0$ for all but finitely many $w$, if
\[
\left(\sum_{w\in\wt\sfW}P_w t_w\right)\dC[\fh_\dC^*]\subset P\dC[\fh_\dC^*],
\]
then either $P$ divides $P_w$ for every $w\in\wt\sfW$, or $P\in\ell_s\dC[\fh_\dC^*]$ for some $s\in\wt\sfW^\heartsuit$.

Suppose that $P\not\in\ell_s\dC[\fh_\dC^*]$ for every $s\in\wt\sfW^\heartsuit$. Then, the zero set of $P$ is not contained in the union $\bigcup_{w\in\wt\sfW\setminus\{1\}}\Ker(w-1)$. This implies that the stabilizer of a generic element of the zero set $Z(P)\coloneqq\{\alpha\in\fh_\dC^*\res P(\alpha)=0\}$ in $\wt\sfW$ is trivial. Let $\alpha\in Z(P)$ be such an element. Note that for every $Q\in\dC[\fh_\dC^*]$, we have
\[
\sum_{w\in\wt\sfW}P_w(\alpha)Q(w^{-1}\alpha)=0.
\]
As the linear forms $Q\mapsto Q(w^{-1}\alpha)$ on $\dC[\fh_\dC^*]$ are linearly independent, it follows that $P_w(\alpha)=0$ for every generic  element $\alpha\in Z(P)$. Thus, $P$ divides $P_w$ for every $w\in\wt\sfW$. The lemma is proved.
\end{proof}

As we have explained in Subsection \ref{ss:strategy}, we will first prove Theorem \ref{th:multiplier} for the smaller space $\cN_{\theta\cup\{1\}}(\fh_\dC^*)^\sfW$, together with bounds on $\|D(\mu\star f)\|_{L^2}$, which is the content of the next proposition.

\begin{proposition}\label{pr:multiplier}
Theorem \ref{th:multiplier} holds for $\mu\in\cN_{\theta\cup\{1\}}(\fh_\dC^*)^\sfW$. Moreover, for every $D\in\cD(G)$, there exists a real number $M_D>0$ such that for every $f\in\cS(G)$ and every $r\in\dR$, there exists $C_{D,f,r}>0$ such that
\[
\|D(\mu\star f)\|_{L^2}\leq C_{D,f,r}\cdot\max_{\vartheta\in\theta}p_{\vartheta,M_D,r}(\mu)
\]
holds for every $\mu\in\cN_{\theta\cup\{1\}}(\fh_\dC^*)^\sfW$.
\end{proposition}

\begin{proof}
Let $C^\infty_c(G)_{(K)}$ be the subalgebra of $C^\infty_c(G)$ of bi-$K$-finite functions. As $\cN_{\theta\cup\{1\}}(\fh_\dC^*)^\sfW\subseteq\cN(\fh_\dC^*)^\sfW$, the existence of the function $\mu\star f$ for $f\in C^\infty_c(G)_{(K)}$ was essentially proved by Delorme (see Proposition \ref{pr:delorme}). As $C^\infty_c(G)_{(K)}$ is dense in $\cS(G)$, it suggests us to show that the map $C^\infty_c(G)_{(K)}\to\cS(G)$, $f\mapsto \mu\star f$, extends continuously to an endomorphism of $\cS(G)$.

We claim that for every $D\in\cD(G)$, there exists $M_D>0$ such that for every $r\in\dR$, there exists a continuous semi-norm $\nu_{D,r}$ on $\cS(G)$ such that
\begin{align}\label{eq:multi_1}
\| D(\mu\star f)\|_{L^2}\leq\nu_{D,r}(f)\cdot\max_{\vartheta\in\theta}p_{\vartheta,M_D,r}(\mu)
\end{align}
holds for every $\mu\in\cN_{\theta\cup\{1\}}(\fh_\dC^*)^\sfW$ and $f\in C^\infty_c(G)_{(K)}$.

We deduce the proposition assuming the claim. Fix $\mu\in\cN_{\theta\cup\{1\}}(\fh_\dC^*)^\sfW$, and for every $f\in\cS(G)$, choose a sequence $\{f_n\}$ in $C^\infty_c(G)_{(K)}$ that converges to $f$. Then by Remark \ref{re:sobolev}, \eqref{eq:multi_1} implies that $\mu\star f_n$ converges to an element in $\cS(G)$, which we denote by $\mu\star f$. It is easy to see that the map $\mu\star\colon\cS(G)\to\cS(G)$ satisfies the requirement in Theorem \ref{th:multiplier}. The second part of the proposition follows from \eqref{eq:multi_1} by taking $C_{D,f,r}=\nu_{D,r}(f)$.

Now we show the claim. Note that we have a canonical isomorphism $\dC[G]\otimes\cU(\fg)\simeq\cD(G)$ given by $P\otimes X\mapsto P\cdot\rL(X)$ where $\rL(X)$ stands for the left invariant differential operator on $G$ associated to $X$. For $X\in\cU(\fg)$, we have $\rL(X)(\mu\star f)=\mu\star\rL(X)f$. Since the action of $\rL(X)$ preserves the set of continuous semi-norms on $\cS(G)$, it suffices to show \eqref{eq:multi_1} for $D=P\in\dC[G]$.

First, we consider the case $P=1$. By the Plancherel formula of Harish-Chandra \cite{HC76}, there exists a Borel measure $\r{d}\pi$ on the tempered dual $\Temp(G)$ of $G$ such that
\[
\|f\|_{L^2}^2=\int_{\Temp(G)}\|\pi(f)\|_{\r{HS}}^2\rd\pi
\]
for every $f\in\cS(G)$, where $\|\cdot\|_{\r{HS}}$ stands for the Hilbert-Schmidt norm. Thus, we have
\[
\|\mu\star f\|_{L^2}^2=\int_{\Temp(G)}|\mu(\chi_\pi)|^2\|\pi(f)\|_{\r{HS}}^2\rd\pi
\]
for every $f\in C^\infty_c(G)_{(K)}$. Let $\theta_{\ul{G}}$ be the subset of $\theta$ obtained from maximal tori of $\ul{G}$ in the way of Lemma \ref{le:multi_0} (see also Remark \ref{re:theta}). By Harish-Chandra's description of the infinitesimal characters of tempered representations \cite{HC75}, the union of $\chi_\pi$ for $\pi\in\Temp(G)$ is exactly the $\sfW$-stable subset
\[
\bigcup_{\vartheta\in\theta_{\ul{G}}}i\fh^*_\vartheta\oplus\sfY^*_\vartheta
\]
of $\fh^*_\dC$. In particular, we have
\[
|\mu(\chi_\pi)|\leq\(\max_{\vartheta\in\theta_{\ul{G}}}p_{\vartheta,1,r}(\mu)\)\cdot(1+\|\chi_\pi\|)^{-r}
\]
for every $\pi\in\Temp(G)$. Choose elements $z_1,\ldots,z_N\in \dC[\fh_\dC^*]^\sfW=\cZ(\fg)$ (depending on $r$) such that
\[
(1+\|\chi_\pi\|)^{-2r}\leq |z_1(\chi_\pi)|^2+\cdots+|z_N(\chi_\pi)|^2
\]
holds for every $\pi\in \Temp(G)$. It follows that
\[
\|\mu\star f\|_{L^2}^2 \leq
\(\max_{\vartheta\in\theta_{\ul{G}}}p_{\vartheta,1,r}(\mu)\)^2\sum_{i=1}^N\int_{\Temp(G)} |z_i(\chi_\pi)|^2 \cdot \|\pi(f)\|_{\r{HS}}^2\rd\pi
=\(\max_{\vartheta\in\theta_{\ul{G}}}p_{\vartheta,1,r}(\mu)\)^2\sum_{i=1}^N \|z_i\star f\|_{L^2}^2
\]
for every $f\in C^\infty_c(G)_{(K)}$. Thus, we obtain \eqref{eq:multi_1} with the semi-norm $\nu_{1,r}$ given by
\[
\nu_{1,r}(f)\coloneqq\left(\sum_{i=1}^N \|z_i\star f\|_{L^2}^2 \right)^{\frac{1}{2}}.
\]

Now we treat the case for general $P\in\dC[G]$. As $P$ is a finite sum of matrix coefficients of finite dimensional algebraic representations of $G$, by linearity, we may as well assume that there exists a finite dimensional algebraic representation $(\tau,W_\tau)$ of $G$, $w\in W_\tau$ and $w^*\in W_\tau^*$ such that
\[
P(g)=\langle\tau(g)w,w^*\rangle,\qquad g\in G.
\]
Let $v_1,\ldots,v_r$ be a basis of the finite free $\cZ(\fg)$-module $\cR_\tau$; and let $S_1,\ldots,S_r$ be elements of $\sS$ as in Lemma \ref{le:multi_4}. By Lemma \ref{le:multi_3}, $S_1,\dots,S_r$ extend continuously to endomorphisms of $\cO(\fh_\dC^*)^\sfW$ for the topology of uniform convergence on compact subsets, which preserve $\cN_{\theta\cup\{1\}}(\fh_\dC^*)^\sfW$, for which we use the same notation. Since the map $\cU(\fg)\otimes\End(W_\tau)\to\cU(\fg)\otimes\End(W_\tau)$ given by $u\otimes A\mapsto(1\otimes A)\delta(u)$ is an isomorphism, each element $v_i$ in the basis can be written as a finite sum
\begin{align}\label{eq:multi_2}
v_i=\sum_j (1\otimes A_{ij})\delta(u_{ij})
\end{align}
for some $A_{ij}\in\End(W_\tau)$ and $u_{ij}\in\cU(\fg)$. Let $P_{ij}\in\dC[G]$ be the element defined by $P_{ij}(g)\coloneqq\langle A_{ij}\tau(g)w,w^*\rangle$ for $g\in G$. Note that $S_i$, $u_{ij}$, and $P_{ij}$ depend on the data $(\tau,W_\tau,w,w^*)$ only. We make the following claim.
\begin{itemize}
  \item[($*$)] For every $f\in C^\infty_c(G)_{(K)}$,
      \begin{align*}
      P(\mu\star f)=\sum_{i,j}S_i(\mu)\star P_{ij}\rL(u_{ij})f
      \end{align*}
      holds.
\end{itemize}

Assuming ($*$), by the $P=1$ case, we have
\[
\|P(\mu\star f)\|_{L^2} \leq \sum_{i,j} \|S_i(\mu)\star P_{ij}\rL(u_{ij})f\|_{L^2}
\leq \sum_{i,j}\nu_{1,r'}(P_{ij}\rL(u_{ij})f)\cdot\max_{\vartheta\in\theta}p_{\vartheta,1,r'}(S_i(\mu))
\]
for every $f\in C^\infty_c(G)_{(K)}$ and every $r'\in\dR$. By Lemma \ref{le:multi_3}, there exist $M_P>1$, $r_P>0$, and $C_P>0$ such that
\[
\max_{\vartheta\in\theta}p_{\vartheta,1,r-r_P}(S_i(\mu)) \leq C_P\cdot\max_{\vartheta\in\theta}p_{\vartheta,M_P,r}(\mu)
\]
holds for every $i$ and every $\mu\in\cN_{\theta\cup\{1\}}(\fh_\dC^*)^\sfW$. Thus, we obtain \eqref{eq:multi_1} with the semi-norm $\nu_{P,r}$ given by
\[
\nu_{P,r}(f)\coloneqq C_P\sum_{i,j}\nu_{1,r-r_P}(P_{ij}\rL(u_{ij})f).
\]

It remains to confirm ($*$). By the injectivity of the operator valued Fourier transform, it suffices to check that
\begin{align}\label{eq:multi_4}
\pi(P(\mu\star f))=\sum_{i,j}\pi(S_i(\mu)\star P_{ij}\rL(u_{ij})f)=\sum_{i,j}S_i(\mu)(\chi_\pi)\cdot\pi(P_{ij}\rL(u_{ij})f)
\end{align}
for every irreducible admissible representation $\pi$ of $G$. By the subquotient theorem, it is enough to check \eqref{eq:multi_4} when $\pi$ is a principal series whose underlying space we denote by $V_\pi$. We have the equality
\begin{align}\label{eq:multi_5}
\pi(P(\mu\star f))=c_{w^*}\circ(\pi\otimes\tau)(\mu\star f)\circ b_w
\end{align}
in $\End(V_\pi)$, where $b_w\colon V_\pi\to V_\pi\otimes W_\tau$ is the map defined by $b_w(v)=v\otimes w$, and $c_{w^*}\colon V_\pi\otimes W_\tau\to V_\pi$ is the map defined by $c_{w^*}(v\otimes w')=\langle w',w^*\rangle v$. By a density argument, it is even sufficient to establish \eqref{eq:multi_4} for a principal series in general position, for which, by Lemma \ref{le:multi_5} below, we can assume that $\pi\otimes\tau$ is semisimple.

Assume now that $\pi\otimes\tau$ is semisimple. Let $(z_n)_n$ be a sequence of elements of $\cZ(\fg)$ converging to $\mu$ for the topology of uniform convergence on compact subsets. As $\pi\otimes\tau$ decomposes as a direct sum of irreducible admissible representations of $G$, $(\pi\otimes\tau)(z_n\star f)$ converges to $(\pi\otimes\tau)(\mu\star f)$ for the topology of pointwise convergence. The action of $G$ induces an action of $\cU(\fg)$ on (the smooth vectors in) $V_\pi$, which further induces an action of $\cR_\tau\subseteq\cU(\fg)\otimes\End(W_\tau)$ on $V_\pi\otimes W_\tau$. Note that under this action, $\delta(u)$ acts by $(\pi\otimes\tau)(u)$ for every $u\in\cU(\fg)$, and $z\otimes 1$ acts by $\pi(z)\otimes 1_{W_\tau}$ for every $z\in\cZ(\fg)$. Thus, by \eqref{eq:multi_2} and Lemma \ref{le:multi_4}, we have
\begin{align*}
(\pi\otimes\tau)(z_n)&=\delta(z_n)=\sum_{i=1}^r S_i(z_n)v_i
=\sum_{i,j}S_i(z_n)(1\otimes A_{ij})\delta(u_{ij}) \\
&=\sum_{i,j}(\pi(S_i(z_n))\otimes 1_{W_\tau})\circ (1_{V_\pi}\otimes A_{ij})\circ(\pi\otimes\tau)(u_{ij}) \\
&=\sum_{i,j}S_i(z_n)(\chi_\pi)\cdot(1_{V_\pi}\otimes A_{ij})\circ(\pi\otimes\tau)(u_{ij})
\end{align*}
in $\End(V_\pi\otimes W_\tau)$. Pre-composing $(\pi\otimes\tau)(f)$ and passing to the limit, we get
\[
(\pi\otimes\tau)(\mu\star f)=\sum_{i,j}S_i(\mu)(\chi_\pi)\cdot (1_{V_\pi}\otimes A_{ij})\circ(\pi\otimes\tau)(\rL(u_{ij})f),
\]
which, together with \eqref{eq:multi_5}, implies \eqref{eq:multi_4}. The claim ($*$) is proved.

The proof of the proposition is now complete.
\end{proof}

\begin{lem}\label{le:multi_5}
Let $P_0=M_0N_0$ be a minimal parabolic subgroup of $G$, $\sigma$ an irreducible (finite dimensional) representation of $M_0$, and put $\sigma_\xi\coloneqq\sigma\otimes\xi$ for $\xi$ an unramified character of $M_0$. Then for $\xi$ in general position, the admissible representation $\rI_{P_0}^G(\sigma_\xi)\otimes\tau$ is semisimple, that is, a direct sum of irreducible admissible representations.
\end{lem}

\begin{proof}
By the Frobenius reciprocity, $\rI_{P_0}^G(\sigma_\xi)\otimes\tau$ is isomorphic to $\rI_{P_0}^G(\sigma_\xi\otimes\tau\res_{P_0})$, where $\tau\res_{P_0}$ denotes the restriction of $\tau$ to $P_0$ (so that the inducing representation is now non-trivial on $N_0$). Let $A_0$ be the split center of $M_0$ and $\fa_0$ its Lie algebra. The representation $\sigma\otimes\tau\res_{P_0}$ of $P_0$ admits a filtration indexed by the characters of $\fa_0$ for the partial order defined by the cone of positive roots with respect to $P_0$, such that in the associated graded vector space $\bigoplus_{\lambda\in\fa^*_{0,\dC}}V_\lambda$, $\fa_0$ acts on $V_\lambda$ by the character $\lambda$. This implies that $N_0$ acts trivially on each of the graded pieces. Thus, $\{V_\lambda\res\lambda\in\fa^*_{0,\dC}\}$ are semisimple representations of $M_0$ with distinct central characters. Therefore, $\rI_{P_0}^G(\sigma_\xi)\otimes\tau$ admits a filtration with associated graded representations $\bigoplus_{\lambda\in\fa^*_{0,\dC}} \rI_{P_0}^G(V_\lambda\otimes\xi)$; and for generic $\xi$, the representations $\{\rI_{P_0}^G(V_\lambda\otimes\xi)\res\lambda\in\fa^*_{0,\dC}\}$ are all semisimple with distinct infinitesimal characters, hence $\rI_{P_0}^G(\sigma_\xi)\otimes\tau$ is itself semisimple. The lemma follows.
\end{proof}

Now we deduce Theorem \ref{th:multiplier} from Proposition \ref{pr:multiplier} by a limit process.

\begin{proof}[Proof of Theorem \ref{th:multiplier}]
First, we claim that there exists a polynomial $P\in\dC[\fh^*_\dC]^\sfW$ that is homogeneous, and takes positive real values on $\(\bigcup_{\vartheta\in\theta\cup\{1\}}i\fh^*_\vartheta\oplus\fh^{*-}_\vartheta\)\setminus\{0\}$. Indeed, such polynomial can be obtained
as follows. Let $\sfW_\theta$ be the group of linear automorphisms of $\fh^*$ generated by $\sfW$ and $\theta$. Let $P_1,\dots,P_N$ be homogeneous generators of the kernel of $\dR[\fh^*]^{\sfW_\theta}\to\dR$ sending $P$ to $P(0)$ (as an ideal of $\dR[\fh^*]^{\sfW_\theta}$) with $d_j\coloneqq\deg P_j>0$ for $1\leq j\leq N$. Put
\[
P\coloneqq P_1^{4d'_1}+\cdots+P_N^{4d'_N}
\]
where $d'_j\coloneqq d_1\cdots \widehat{d_j}\cdots d_N$. Since $P$ is homogeneous of degree $4d_1\cdots d_N$, we have $P(i\alpha)=P(\alpha)$ for every $\alpha\in\fh^*_\dC$. As the only common zero of $P_1,\dots,P_N$ is zero, it suffices to show that $P_j$ takes real values on $\fh^*_\vartheta\oplus i\fh^{*-}_\vartheta$ for every $1\leq j\leq N$ and every $\vartheta\in\theta\cup\{1\}$. Since $\vartheta(\ol\alpha)=\alpha$ for $\alpha\in\fh^*_\vartheta\oplus i\fh^{*-}_\vartheta$ and that $P_j$ is $\vartheta$-invariant, we have $P_j(\alpha)=P_j(\vartheta(\ol\alpha))=P_j(\ol\alpha)=\ol{P_j(\alpha)}$, hence $P_j(\alpha)\in\dR$ for $\alpha\in\fh^*_\vartheta\oplus i\fh^{*-}_\vartheta$.

We now choose such a polynomial $P$ as above, of degree $d>0$. It has the following property: there exist $\delta>0$ and $\epsilon>0$ such that for every $\vartheta\in\theta\cup\{1\}$ and every $\alpha\in\fh^*_{\vartheta,\dC}$ and every $\varpi\in\fh^{*-}_\vartheta$ satisfying $\|\alpha+\varpi\|=1$,
\[
\RE P(\alpha+\varpi)\geq \epsilon
\]
holds as long as $\|\RE\alpha\|\leq\delta$. In particular, for every $\vartheta\in\theta\cup\{1\}$ and $M>0$,
\[
\RE P(\alpha+\varpi)=\|\alpha+\varpi\|^d\cdot\RE P\(\frac{\alpha+\varpi}{\|\alpha+\varpi\|}\)
\geq \epsilon\|\alpha+\varpi\|^d
\]
holds if $\|\RE\alpha\|<M$ and $\|\alpha+\varpi\|\geq \delta^{-1}M$. This implies the following statements:
\begin{enumerate}
  \item For every $M>0$ and every $r<0$, there exists $C_{M,r}>0$ such that
      \[
      \left|\exp\(-P(\alpha+\varpi)\)\right|<C_{M,r}(1+\|\alpha+\varpi\|)^r
      \]
      holds for every $\vartheta\in\theta\cup\{1\}$, every $\alpha\in\fh^*_{\vartheta,\dC}$ with $\|\RE\alpha\|<M$, and every $\varpi\in\sfY^*_\vartheta$.

  \item For every $\epsilon>0$ that is sufficiently small and every $M>0$, there exists $C_{\epsilon,M}>0$ such that
      \[
      \left|\exp\(-P(\alpha+\varpi)\)-1\right|\leq C_{\epsilon,M}\|\alpha+\varpi\|^{\epsilon}
      \]
      holds for every $\vartheta\in\theta\cup\{1\}$, every $\alpha\in\fh^*_{\vartheta,\dC}$ with $\|\RE\alpha\|<M$, and every $\varpi\in\sfY^*_\vartheta$.
\end{enumerate}

Now we take an element $\mu\in\cM_{\theta\cup\{1\}}(\fh^*_\dC)^\sfW$, and define functions $\mu_n$ for every integer $n\geq 1$ by the formula
\[
\mu_n(\alpha)=\exp\(-P\(\frac{\alpha}{n}\)\)\cdot\mu(\alpha),\qquad\alpha\in\fh^*_\dC.
\]
By Definition \ref{de:multiplier}(2), for every $M>0$, we may choose $r_M\in\dR$ such that $p_{\vartheta,M,r_M}(\mu)<\infty$ holds for every $\vartheta\in\theta\cup\{1\}$.

From (1), we obtain $p_{\vartheta,M,r}(\mu_n)<\infty$ for every $\vartheta\in\theta\cup\{1\}$, every $M>0$, every $r>r_M$, and every $n\geq 1$. Thus, $\mu_n\in\cN_{\theta\cup\{1\}}(\fh^*_\dC)^\sfW$ for every $n\geq 1$. From (2), we obtain
\[
\lim_{n\to\infty}p_{\vartheta,M,r}(\mu_n-\mu)=0
\]
for every $\vartheta\in\theta\cup\{1\}$, every $M>0$, and every $r<r_M$. Then, by Proposition \ref{pr:multiplier} and Remark \ref{re:sobolev}, for every $f\in\cS(G)$, $\{\mu_n\star f\}_{n\geq 1}$ is a Cauchy sequence in $\cS(G)$; and we denote its limit (in $\cS(G)$) by $\mu\star f$. It follows immediately that $\pi(\mu\star f)=\mu(\chi_\pi)\cdot\pi(f)$ holds for every irreducible admissible representation $\pi$ of $G$. The theorem is proved.
\end{proof}

\begin{remark}\label{re:continuous}
In fact, from the proof of Theorem \ref{th:multiplier}, we see that for every $\mu\in\cM_{\theta\cup\{1\}}(\fh^*_\dC)$, $\mu\star$ is continuous, and preserves the subalgebra $\cS(G)_{(K)}$ of bi-$K$-finite functions in $\cS(G)$. Moreover, our argument of deducing Theorem \ref{th:multiplier} from Proposition \ref{pr:multiplier} can be applied to deduce from Proposition \ref{pr:delorme} that every element of $\cM(\fh^*_\dC)$ gives a multiplier of $\cS(G)_{(K)}$.
\end{remark}

\section{Isolation of spectrum}
\label{ss:3}

In this section, we state and prove various results on the isolation of cuspidal components of the $L^2$-spectrum. In Subsection \ref{ss:statement}, we introduce several important definitions and state the results. In Subsection \ref{ss:discrete}, we recall the coarse Langlands decomposition, cuspidal data, and cuspidal components. In Subsection \ref{ss:annihilation}, we show how to annihilate all but finitely many cuspidal components. We finish the proof in Subsection \ref{ss:isolation_proof}.

Let $F$ be a number field.

\subsection{Statement of results}
\label{ss:statement}

We consider a connected reductive algebraic group $G$ over $F$. Let $\fg$ be the complex Lie algebra of $G\otimes_\dQ\dC$, and $Z$ the maximal $F$-split torus in the center of $G$. Let $\tS_G$ be the set of primes of $F$ such that $G_v$ is ramified. We fix a maximal compact subgroup $K_0$ of $G(\dA_F)$, and a Haar measure $\r{d}g=\prod_v\r{d}g_v$ on $G(\dA_F)$, so that $K_{0,v}$ is hyperspecial maximal with volume $1$ under $\r{d}g_v$ for every prime $v$ not in $\tS_G$.

We first recall the definition of the $L^2$-spectrum for $G$. Take a unitary automorphic character
\[
\omega\colon Z(\dA_F)\to\dC^\times.
\]
We define $L^2(G(F)\backslash G(\dA_F),\omega)$ to be the $L^2$ completion of the subspace of smooth functions $\varphi$ on $G(\dA_F)$ satisfying
\begin{itemize}
  \item $\varphi(z\gamma g)=\omega(z)\varphi(g)$ for every $z\in Z(\dA_F)$, $\gamma\in G(F)$, and $g\in G(\dA_F)$;

  \item $|\varphi|^2$, regarded as a function on $Z(\dA_F)G(F)\backslash G(\dA_F)$, is integrable.
\end{itemize}
The group $G(\dA_F)$ acts on $L^2(G(F)\backslash G(\dA),\omega)$ via the right regular representation $\rR$, which preserves the subspace $L^2_\cusp(G(F)\backslash G(\dA),\omega)$ of cuspidal functions.

\begin{definition}
We define the space of \emph{Schwartz test functions} to be
\[
\cS(G(\dA_F))\coloneqq\cS(G_\infty)\otimes\cS(G(\dA_F^\infty)),
\]
which is endowed with the convolution product with respect to the fixed Haar measure $\r{d}g$. Here, $\cS(G(\dA_F^\infty))$ denotes the
restricted tensor product $\bigotimes'_{v\nmid\infty}\cS(G(F_v))$, in which $\cS(G(F_v))$ is nothing but $C^\infty_c(G(F_v))$.
\end{definition}

The algebra $\cS(G(\dA_F))$ acts on $L^2(G(F)\backslash G(\dA_F),\omega)$ continuously via the right regular representation $\rR$, with respect to the Haar measure $\r{d}g$.

\begin{definition}\label{de:quasi_cuspidal}
We say that a Schwartz test function $f\in\cS(G(\dA_F))$ is \emph{$\omega$-quasi-cuspidal} (or simply \emph{quasi-cuspidal} when $\omega=1$ is the trivial character) if the image of the endomorphism $\rR(f)$ on $L^2(G(F)\backslash G(\dA_F),\omega)$ is contained in $L^2_\cusp(G(F)\backslash G(\dA_F),\omega)$.
\end{definition}

For every prime $v\not\in\tS_G$, we let
\[
\cH_{G,v}\coloneqq\dC[K_{0,v}\backslash G(F_v)/K_{0,v}]
\]
be the spherical Hecke algebra of $G_v$ (with respect to $K_{0,v}$) with the unit $\mathbf{1}_{K_{0,v}}$. For a (possibly infinite) set $\tT$ of primes of $F$ containing $\tS_G$, we let $\cH_G^\tT$ be the restricted tensor product of $\cH_{G,v}$ for primes $v\not\in\tT$.

\begin{definition}\label{de:compatible}
A \emph{$\tT$-character} (for $G$) is a pair $\chi=(\chi_\infty,\chi^{\infty,\tT})$ in which $\chi_\infty$ is a character of $\cZ(\fg)$ and $\chi^{\infty,\tT}$ is a character of $\cH_G^\tT$.\footnote{Warning: $\tT$-characters are actually characters away from $\tT$; same for $\tT$-multipliers later.}
\end{definition}

For a $\tT$-character $\chi$ and a prime $v\not\in\tT$, its $v$-component $\chi_v$, which is a character of $\cH_{G,v}$, gives rise to a ($K_{0,v}$-)spherical (irreducible admissible) representation of $G(F_v)$, unique up to isomorphism. For an irreducible admissible representation $\pi=\otimes_v\pi_v$ of $G(\dA_F)$ such that $\pi_v^{K_{0,v}}\neq\{0\}$ for every prime $v\not\in\tT$, we have the induced $\tT$-character $\chi_{\pi^\tT}=(\chi_{\pi_\infty},\chi_{\pi^{\infty,\tT}})$.

The following definition mimics the original notion of CAP representations by Piatetski-Shapiro \cite{PS82}. In this section, we will only use it for $G'=G$; while the more general case will be used in Section \ref{ss:4}.

\begin{definition}\label{de:cap}
Let $G'$ be an inner form of $G$ (over $F$). We say that a $\tT$-character $\chi$ is \emph{$(G',\tT)$-CAP} if there exist a proper parabolic subgroup $P'$ of $G'$ and a cuspidal automorphic representation $\sigma$ of $M_{P'}(\dA_F)$, where $M_{P'}\coloneqq P'/N_{P'}$, such that for all but finitely many primes $v$ of $F$ not in $\tT$ for which $G_v\simeq G'_v$, the spherical representation corresponding to $\chi_v$ is a constituent of $\rI^{G'}_{P'}(\sigma_v)$. For an irreducible admissible representation $\pi$ of $G(\dA_F)$, we say that $\pi$ is $(G',\tT)$-CAP if $\chi_{\pi^\tT}$ is.
\end{definition}

Now we fix
\begin{itemize}
  \item a subset $\tT$ of primes of $F$ containing $\tS_G$ and a $\tT$-character $\chi=(\chi_\infty,\chi^{\infty,\tT})$,

  \item a finite set $\tS$ of primes of $F$ satisfying $\tS_G\subseteq\tS\subseteq\tT$, and

  \item a subgroup $K\subseteq K_0^\infty$ of finite index of the form $K=K_\tS\times\prod_{v\not\in\tS}K_{0,v}$.
\end{itemize}
We denote by $\cS(G(\dA_F^\infty))_K$ the subalgebra of $\cS(G(\dA_F^\infty))$ of bi-$K$-invariant functions, and put
\[
\cS(G(\dA_F))_K\coloneqq\cS(G_\infty)\otimes\cS(G(\dA_F^\infty))_K.
\]
For $?=\emptyset,\cusp$, we denote by $L^2_?(G(F)\backslash G(\dA_F)/K,\omega)$ the subspace of $L^2_?(G(F)\backslash G(\dA_F),\omega)$ consisting of functions that are invariant under $\rR(K)$, on which $\cS(G(\dA_F))_K$ acts continuously via the right regular representation $\rR$. We denote by
\[
L^2_\cusp(G(F)\backslash G(\dA_F)/K,\omega)[\chi]\subseteq L^2_\cusp(G(F)\backslash G(\dA_F)/K,\omega)
\]
the maximal closed subspace of which $\cZ(\fg)\otimes\cH_G^\tT$ acts on its smooth vectors by the character $\chi_\infty\otimes\chi^{\infty,\tT}$.

The connected reductive group $G\otimes_\dQ\dR$ over $\dR$ determines a root datum $(\sfX^*,\Phi,\sfX_*,\Phi^\vee)$ and a subset $\theta\subseteq\Aut(\sfX^*,\Phi,\sfX_*,\Phi^\vee)^\heartsuit$ as at the beginning of Subsection \ref{ss:multiplier_algebra}. We adopt notation in Subsection \ref{ss:multiplier_functions}. For an element $\mu\in\cO(\fh^*_\dC)^\sfW\otimes\cH_G^\tT$, we may evaluate $\mu$ on $\chi$ to obtain a complex number $\mu(\chi)$. Theorem \ref{th:multiplier} provides us with a linear map $\cM^\sharp_\theta(\fh_\dC^*)^\sfW\to\Mul(\cS(G_\infty))$ (see Remark \ref{re:multiplier_map}).

\begin{definition}\label{de:T_multiplier}
A \emph{$\tT$-multiplier} of $\cS(G(\dA_F))_K$ is an element in $\cM^\sharp_\theta(\fh_\dC^*)^\sfW\otimes\cH_G^\tT$.
\end{definition}

As $\cH_G^\tT$ is contained in the center of $\cS(G(\dA_F^\infty))_K$, a $\tT$-multiplier $\mu$ induces a multiplier
\[
\mu\star\in\Mul(\cS(G(\dA_F))_K)
\]
of $\cS(G(\dA_F))_K$. Now we can state our main results on the isolation of the spectrum.

\begin{theorem}\label{th:isolation}
Suppose that $\chi$ is \emph{not} $(G,\tT)$-CAP (Definition \ref{de:cap}). Then there exists a $\tT$-multiplier $\mu$ of $\cS(G(\dA_F))_K$ such that for every $f\in\cS(G(\dA_F))_K$,
\begin{enumerate}
  \item $\rR(\mu\star f)$ maps $L^2(G(F)\backslash G(\dA_F)/K,\omega)$ into $L^2_\cusp(G(F)\backslash G(\dA_F)/K,\omega)[\chi]$;

  \item $\mu(\chi)=1$.
\end{enumerate}
In particular, $\mu\star f$ is an $\omega$-quasi-cuspidal Schwartz test function (Definition \ref{de:quasi_cuspidal}).
\end{theorem}

In particular, Theorem \ref{th:isolation_pre} follows by taking $\tT=\tS$ and $\chi=\chi_{\pi^\tT}$.

For general $\chi$, we have the following theorem. In fact, we will prove a stronger result in Theorem \ref{th:isolation_tri} below.

\begin{theorem}\label{th:isolation_bis}
There exists a $\tT$-multiplier $\mu$ of $\cS(G(\dA_F))_K$ such that for every $f\in\cS(G(\dA_F))_K$,
\begin{enumerate}
  \item $\rR(\mu\star f)$ maps $L^2_\cusp(G(F)\backslash G(\dA_F)/K,\omega)$ into $L^2_\cusp(G(F)\backslash G(\dA_F)/K,\omega)[\chi]$;

  \item $\mu(\chi)=1$.
\end{enumerate}
\end{theorem}

\begin{remark}
In the above two theorems, we do not require that $\chi=\chi_{\pi^\tT}$ for an irreducible admissible representation $\pi$ of $G(\dA_F)$.
\end{remark}

The remaining part of this section will be devoted to the proof of the above two theorems.

\subsection{Cuspidal data and cuspidal components}
\label{ss:discrete}

We recall the notion of cuspidal data and cuspidal components for the group $G$. We fix
\begin{itemize}
  \item a minimal Levi subgroup $M_0$ of (a parabolic subgroup of) $G$,

  \item a minimal parabolic subgroup $P_0$ of $G$ containing $M_0$, and

  \item a maximal torus $T$ of $G\otimes_\dQ\dR$ over $\dR$ that is contained in $M_0\otimes_\dQ\dR$.
\end{itemize}
Without lost of generality, we assume that the fixed maximal compact subgroup $K_0$ of $G(\dA_F)$ is admissible relative to $M_0(\dA_F)$ (in the sense of \cite{Art81}*{Section~1}). We identify $\sfX^*$ with the weight lattice of $T_\dC$. Let $\sfW(G,M_0)$ be the Weyl group of the pair $(G,M_0)$. We say that a subgroup $M$ of $G$ is a \emph{standard Levi subgroup} if there exists a parabolic subgroup of $G$ containing $P_0$, of which $M$ is the unique Levi subgroup containing $M_0$.

For a standard Levi subgroup $M$ of $G$, we denote by
\begin{itemize}
  \item $Z_M$ the center of $M$ (in particular, $Z$ is the maximal $F$-split torus in $Z_G$),

  \item $P_M$ the unique parabolic subgroup of $G$ containing $P_0$, of which $M$ is a Levi subgroup,

  \item $\Omega_M(\omega)$ the set of unitary automorphic characters $\omega_M\colon Z_M(\dA_F)\to\dC^\times$ satisfying $\omega_M\res_{Z(\dA_F)}=\omega$,

  \item $M(\dA_F)^1$ the intersection of the kernels of all automorphic characters $M(\dA_F)\to\dR^\times_+$, and

  \item $\fa_M$ the real vector space $Z(\dA_F)M(\dA_F)^1\backslash M(\dA_F)$.\footnote{In fact, $\fa_M$ is usually denoted by $\fa_M^G$ for example in \cites{Art05,Art13}; it is the real Lie algebra of the maximal $\dQ$-split torus contained in $\Res_{F/\dQ}Z\backslash M$.}
\end{itemize}
For $\bbs\in\fa_{M,\dC}^*$, we denote by $\xi_{\bbs}\colon Z(\dA_F)M(\dA_F)^1\backslash M(\dA_F)\to\dC^\times$ the corresponding (automorphic) character obtained by composing $s$ with the exponential map $\exp\colon\dC\to\dC^\times$, which is unitary if and only if $\bbs\in i\fa_M^*$. For an admissible representation $\sigma$ of $M(\dA_F)$, we put $\sigma_{\bbs}\coloneqq\sigma\otimes\xi_{\bbs}$ for $\bbs\in\fa_{M,\dC}^*$.

\begin{definition}\label{de:component}
For a cuspidal automorphic representation $\sigma$ of $M(\dA_F)$ with central character $\omega_M\in\Omega_M(\omega)$, we denote by
\[
L^2(M,\sigma)\subseteq L^2_\cusp(M(F)\backslash M(\dA_F),\omega_M)
\]
the maximal closed $\sigma$-isotypic subspace.

We define $\fC(M,\omega)$ to be the set of isomorphism classes of cuspidal automorphic representations of $M(\dA_F)$ whose central character belongs to $\Omega_M(\omega)$, and an equivalence relation $\sim$ on $\fC(M,\omega)$ by the rule: $\sigma\sim\sigma'$ if there exists $\bbs\in i\fa_M^*$ such that $\sigma'=\sigma_{\bbs}$. Let $\fC(M,\omega)^\heartsuit$ be the quotient of $\fC(M,\omega)$ by $\sim$.

We define $\fD(G,\omega)$ to be the set of pairs $(M,\sigma)$ where $M$ is a standard Levi subgroup of $G$ and $\sigma\in\fC(M,\omega)$, and an equivalence relation $\approx$ on $\fD(G,\omega)$ by the rule: $(M,\sigma)\approx(M',\sigma')$ if there exists $w\in\sfW(G,M_0)$ such that $M'=M^w$ and $\sigma'\sim\sigma^w$ in the sense above. We denote by $\fD(G,\omega)^\heartsuit$ the quotient of $\fD(G,\omega)$ by $\approx$.
\end{definition}

By \cite{Lan76}*{Lemma~4.6(i)} or \cite{MW95}*{II.2.4}, we have the following decomposition
\begin{align}\label{eq:bernstein}
L^2(G(F)\backslash G(\dA_F),\omega)=\underset{(M,\sigma)\in\fD(G,\omega)^\heartsuit}{\widehat\bigoplus}
L^2_{(M,\sigma)}(G(F)\backslash G(\dA_F),\omega)
\end{align}
of Hilbert spaces, known as the coarse Langlands decomposition. We call an element $(M,\sigma)$ in $\fD(G,\omega)^\heartsuit$ a \emph{cuspidal datum}, and $L^2_{(M,\sigma)}(G(F)\backslash G(\dA_F),\omega)$ the \emph{(automorphic) cuspidal component} associated to $(M,\sigma)$.

For readers' convenience, we recall the construction of $L^2_{(M,\sigma)}(G(F)\backslash G(\dA_F),\omega)$. For $(M,\sigma)\in\fD(G,\omega)$, we denote by $\cA^G_{(M,\sigma)}$ the space of $\rR(K_0)$-finite smooth functions $\phi$ on $N_{P_M}(\dA_F)\backslash G(\dA_F)$ satisfying that for every $x\in K_0$, the function $\phi_x\colon m\mapsto \delta_{P_M}^{-1/2}(m)\phi(mx)$ is an automorphic form on $M(\dA_F)$ contained in $L^2(M,\sigma)$, where $\delta_{P_M}$ is the modulus character of $P_M$. For every Paley--Wiener function
\[
\Phi\colon\fa_{M,\dC}^*\to\cA^G_{(M,\sigma)}
\]
valued in a finite dimensional subspace of $\cA^G_{(M,\sigma)}$, we put
\[
\widetilde\Phi(g)\coloneqq\int_{i\fa_M^*}\Phi_{\bbs}(g)\cdot\xi_{\bbs}(m(g))\rd\bbs
\]
for $g\in P(F)\backslash G(\dA_F)$, where $m(g)\in Z(\dA_F)M(\dA_F)^1\backslash M(\dA_F)$ denotes the image of the $M(\dA_F)$-component of $g$ under the Iwasawa decomposition $G(\dA_F)=N_{P_M}(\dA_F)M(\dA_F)K_0$. We have the \emph{pseudo-Eisenstein series}
\[
E(g,\Phi)\coloneqq\sum_{\gamma\in P(F)\backslash G(F)}\widetilde\Phi(\gamma g),
\]
which belongs to $L^2(G(F)\backslash G(\dA_F),\omega)$. Then $L^2_{(M,\sigma)}(G(F)\backslash G(\dA_F),\omega)$ is defined to be the closure of the subspace spanned by $E(-,\Phi)$ for all Paley--Wiener functions $\Phi$ as above. We have that
\[
L^2_{(M,\sigma)}(G(F)\backslash G(\dA_F),\omega)=L^2_{(M',\sigma')}(G(F)\backslash G(\dA_F),\omega)
\]
if and only if $(M,\sigma)\approx(M',\sigma')$.

We denote by $L^2_{(M,\sigma)}(G(F)\backslash G(\dA_F)/K,\omega)$ the subspace of $L^2_{(M,\sigma)}(G(F)\backslash G(\dA_F),\omega)$ consisting of functions that are invariant under $\rR(K)$. Taking $\rR(K)$-invariants, \eqref{eq:bernstein} induces the following decomposition
\begin{align}\label{eq:decomposition}
L^2(G(F)\backslash G(\dA_F)/K,\omega)=\underset{(M,\sigma)\in\fD(G,\omega)^\heartsuit}{\widehat\bigoplus}L^2_{(M,\sigma)}(G(F)\backslash G(\dA_F)/K,\omega)
\end{align}
of Hilbert spaces.

\subsection{Annihilation of all but finitely many components}
\label{ss:annihilation}

We start to prove the results in Subsection \ref{ss:statement}. In particular, we recall the $\tT$-character $\chi=(\chi_\infty,\chi^{\infty,\tT})$, the subsets $\tS_G\subseteq\tS\subseteq\tT$ of primes of $F$, and the subgroup $K\subseteq K_0^\infty$. In this subsection, we construct an element $\mu_\infty^\chi\in\cM^\sharp_\theta(\fh_\dC^*)^\sfW$ satisfying $\mu_\infty^\chi(\chi_\infty)\neq 0$ and such that $\rR(\mu_\infty^\chi\star f)$ annihilates all but finitely many cuspidal components in \eqref{eq:decomposition}.

Take a standard Levi subgroup $M$ of $G$. Let $M_\der$ be the derived subgroup of $M$. Let $\sfX^*_M$ be the weight lattice of $(T\cap M_\der\otimes_\dQ\dR)_\dC$, and put $\fb^*_M\coloneqq\sfX^*_M\otimes_\dZ\dR$. Then we have a canonical map $\sfX^*\to\sfX^*_M$, which induces a surjective map $\fh^*\to\fb^*_M$. Denote by $\fc^*_M$ the kernel of $\fh^*\to\fb^*_M$, which canonically contains $\fa^*_M$.

\begin{remark}\label{re:adjacent}
In fact, $\fc_M^*/\fa_M^*$ is canonically the real cotangent space of $Z_M(\dA_F)^1Z(\dA_F)$ at the identity. Thus, $\fc_M^*=\fa_M^*$ holds if and only if $Z=1$ and $Z_M(\dA_F)^1$ is discrete; and $\fc_M^*=\fa_M^*$ holds for every standard Levi subgroup $M$ if and only if $G$ is semisimple and split over $\dQ$.
\end{remark}

Denote by $\gamma_M\colon\fh^*\to\fh^*/\fa^*_M$ the quotient map. We have a canonical decomposition $\fh^*=\fc_M^*\oplus\fb^*_M$, which gives rise to two linear maps
\[
\gamma_M^+\colon\fh^*\to\fc^*_M/\fa_M^*,\qquad
\gamma_M^-\colon\fh^*\to\fb^*_M
\]
so that $\gamma_M=\gamma_M^+\oplus\gamma_M^-$. Let $\sfW_M$ be the Weyl group of the pair $(M\otimes_\dQ\dC,T_\dC)$, which is canonically a subgroup of $\sfW$ and acts on $\fc^*_M$ trivially. For every element $\sigma\in\fC(M,\omega)$ and every $\bbs\in\fa_{M,\dC}^*$, the infinitesimal character $\chi_{\sigma_{\bbs,\infty}}$ is a $\sfW_M$-orbit in $\fh^*_\dC$, satisfying that $\gamma_M(\chi_{\sigma_{\bbs,\infty}})\subseteq\fh_\dC^*/\fa_{M,\dC}^*$ does not depend on $\bbs$, hence only on the class of $\sigma$ in $\fC(M,\omega)^\heartsuit$. We also denote by $\chi_{\sigma_{\bbs,\infty}}^G$ the infinitesimal character of $\Ind^G_{P_M}(\chi_{\sigma_{\bbs,\infty}})$, which is simply the $\sfW$-orbit of $\chi_{\sigma_{\bbs,\infty}}$ in $\fh^*_\dC$.

The Casimir operator for $M_\der\otimes_\dQ\dR$ defines a map
\[
\lambda_M\colon\fb_{M,\dC}^*\to\dC,
\]
which is a $\sfW_M$-invariant polynomial function.

\begin{definition}\label{de:adjacent}
We say that an element $\sigma\in\fC(M,\omega)^\heartsuit$ is \emph{$\chi_\infty$-typical} if both
\begin{itemize}
  \item $\gamma_M^+(\chi_{\sigma_\infty})\in\gamma_M^+(\chi_\infty)$; and

  \item $\lambda_M(\gamma_M^-(\chi_{\sigma_\infty}))\in\lambda_M(\gamma_M^-(\chi_\infty))$
\end{itemize}
hold. Denote by $\fC(M,\omega)^\heartsuit_{\chi_\infty}$ the subset of $\fC(M,\omega)^\heartsuit$ consisting of $\chi_\infty$-typical elements.
\end{definition}

Informally speaking, $\chi_\infty$-typical elements are those whose infinitesimal characters can not be distinguished from $\chi_\infty$ via the two maps $\gamma_M^+$ and $\lambda_M\circ\gamma_M^-$, hence their associated cuspidal components will not be annihilated by the method described below in this subsection.

Now we start the construction of $\mu_\infty$, following the strategy (the first two steps) described in Subsection \ref{ss:strategy}. We first choose an element $\mu_\infty^0\in\cM^\sharp_\theta(\fh_\dC^*)^\sfW$ such that $\mu_\infty^0(\chi_\infty)\neq 0$, which is possible by Lemma \ref{le:multiplier1}. We denote by $\fT\coloneqq\fT(\mu_\infty^0)$ the finite set of $K_{0,\infty}$-types from Lemma \ref{le:multiplier0}. For every standard Levi subgroup $M$ of $G$, we
\begin{itemize}
  \item put $K^M_{0,\infty}\coloneqq M_\infty\cap K_{0,\infty}$, which is a maximal compact subgroup of $M_\infty$;

  \item fix a finite set $\fT_M$ of $K^M_{0,\infty}$-types satisfying the following property: if $\sigma$ is an irreducible admissible representation of $M_\infty$ such that $\rI^G_{P_M}(\sigma)\res_{K_\infty}$ contains a member from $\fT$, then $\sigma\res_{K^M_{0,\infty}}$ contains a member from $\fT_M$ (when $M=G$, we take $\fT_M=\fT$);

  \item fix an open compact subgroup $K_{M,\tS}$ of $M(F_\tS)$ satisfying the following property: if $\sigma$ is an irreducible admissible representation of $M(F_\tS)$ such that $\rI^G_{P_M}(\sigma)^{K_\tS}\neq\{0\}$, then $\sigma^{K_{M,\tS}}\neq\{0\}$ (when $M=G$, we take $K_{M,\tS}=K_\tS$);

  \item put $K_{M,v}\coloneqq M(F_v)\cap K_v$ for every prime $v$ of $F$ not in $\tS$, which is a hyperspecial maximal subgroup of $M(F_v)$;

  \item put $K_M\coloneqq K_{M,\tS}\times\prod_{v\not\in\tS}K_{M,v}$.
\end{itemize}
Let $\fC(M,\omega;K_M,\fT_M)$ be the subset of $\fC(M,\omega)$ consisting of $\sigma$ satisfying $\sigma^{K_M}\neq\{0\}$ and that $\sigma\res_{K^M_{0,\infty}}$ contains a member from $\fT_M$. It is clear that $\fC(M,\omega;K_M,\fT_M)$ is closed under the equivalence relation $\sim$ in Definition \ref{de:component}. Let $\fC(M,\omega;K_M,\fT_M)^\heartsuit$ be the quotient of $\fC(M,\omega;K_M,\fT_M)$ by $\sim$.

Put $K_{M_\der}\coloneqq K_M\cap M_\der(\dA_F^\infty)$ and $K^{M_\der}_{0,\infty}\coloneqq K^M_{0,\infty}\cap M_{\der,\infty}$. Let $\fT_{M_\der}$ be the (finite) set of $K^{M_\der}_{0,\infty}$-types that appear in the restriction of members in $\fT_M$ to $K^{M_\der}_{0,\infty}$. Denote by $\fC(M_\der;K_{M_\der},\fT_{M_\der})$ the set of isomorphism classes of cuspidal automorphic representations $\varsigma$ of $M_\der(\dA_F)$ satisfying $\varsigma^{K_{M_\der}}\neq\{0\}$ and that $\varsigma\res_{K^{M_\der}_{0,\infty}}$ contains a member from $\fT_{M_\der}$.

The following lemma is crucial for us to construct functions in $\cM(\fh^*_\dC)^\sfW$ that vanish on $\chi_{\sigma_\infty}$ for $\sigma$ not $\chi_\infty$-typical.

\begin{lem}\label{le:casimir}
The subset
\[
\Lambda(M,\omega;K_M,\fT_M)\coloneqq\{\lambda_M(\gamma^-_M(\chi_{\sigma_\infty}))\res\sigma\in\fC(M,\omega;K_M,\fT_M)\}\subseteq\dC
\]
is of finite rank (Definition \ref{de:rank}).
\end{lem}

\begin{proof}
It suffices to show that the subset of Casimir eigenvalues of elements in $\fC(M_\der;K_{M_\der},\fT_{M_\der})$ is of finite rank. As $\fT_{M_\der}$ is a finite set, this is a direct consequence of \cite{Don82}*{Theorem~9.1}.
\end{proof}

The following lemma shows that there are only finitely many $\chi_\infty$-typical elements in $\fC(M,\omega;K_M,\fT_M)^\heartsuit$.

\begin{lem}\label{le:finite2}
The set $\fC(M,\omega;K_M,\fT_M)^\heartsuit\bigcap\fC(M,\omega)^\heartsuit_{\chi_\infty}$ is finite.
\end{lem}

\begin{proof}
There are only finitely many elements in $\fC(M_\der;K_{M_\der},\fT_{M_\der})$ whose Casimir eigenvalue belongs to the set $\lambda_M(\gamma_M^-(\chi_\infty))$ by \cite{Don82}*{Theorem~9.1}. It follows that there are only finitely many elements in $\fC(M,\omega;K_M,\fT_M)^\heartsuit$ that are $\chi_\infty$-typical. In other words, the set $\fC(M,\omega;K_M,\fT_M)^\heartsuit\bigcap\fC(M,\omega)^\heartsuit_{\chi_\infty}$ is finite.
\end{proof}

The following lemma achieves the goal in the second step of the strategy described in Subsection \ref{ss:strategy}.

\begin{lem}\label{le:finite3}
For every standard Levi subgroup $M$ of $G$, there exists an element $\mu_\infty^M\in\cM^\sharp_\theta(\fh_\dC^*)^\sfW$ satisfying $\mu_\infty^M(\chi_\infty)\neq 0$, and such that for every $\sigma\in\fC(M,\omega;K_M,\fT_M)^\heartsuit\setminus\fC(M,\omega)^\heartsuit_{\chi_\infty}$, $\mu_\infty^M$ vanishes on $\chi_{\sigma_{\bbs,\infty}}^G$, the infinitesimal character of $\Ind^G_{P_M}(\sigma_{\bbs,\infty})$, for every $\bbs\in\fa_{M,\dC}^*$.
\end{lem}

\begin{proof}
We first consider the special case where $\fc_M=\fa_M$. In this case, we have $\gamma_M=\gamma_M^-\colon\fh^*\to\fb^*_M$. By Lemma \ref{le:casimir}, we may apply Corollary \ref{co:rapid} to the subset $\Lambda=\Lambda(M,\omega;K_M,\fT_M)\setminus\lambda_M(\gamma_M(\chi_\infty))$, hence obtain a holomorphic function $\nu$ on $\fb_{M,\dC}^*$ that has moderate vertical growth, and vanishes exactly on $\lambda_M^{-1}(\Lambda(M,\omega;K_M,\fT_M)\setminus\lambda_M(\gamma_M(\chi_\infty)))$. We regard $\nu$ as a function on $\fh_\dC^*$ via the quotient map $\gamma_M\colon\fh_\dC^*\to\fb_{M,\dC}^*$. By construction, $\nu$ is nowhere vanishing on $\chi_\infty$, and vanishes on $\chi_{\sigma_{\bbs,\infty}}$ for every $\sigma\in\fC(M,\omega;K_M,\fT_M)^\heartsuit\setminus\fC(M,\omega)^\heartsuit_{\chi_\infty}$ and $\bbs\in\fa_{M,\dC}^*$. Finally, put
\[
\mu_\infty^M\coloneqq\mu_\infty^0\cdot\prod_{w\in\sfW}\nu\circ w,
\]
which belongs to $\cM^\sharp_\theta(\fh_\dC^*)^\sfW$ by Remark \ref{re:multiplier}(3); it satisfies the requirements in the lemma.

Now we treat the general case where $\fc_M/\fa_M$ might be nontrivial. The extra work is to deal with the factor $\gamma_M^+$. For every element $\sigma\in\fC(M,\omega;K_M,\fT_M)^\heartsuit$, the restriction of the central character of $\sigma_{\bbs}$ to $Z_M(\dA_F)^1Z(\dA_F)$, say $\omega_\sigma$, is independent of $s$. Moreover, the infinitesimal character of $\omega_{\sigma,\infty}$, which is canonically an element in $\fc_{M,\dC}^*/\fa_{M,\dC}^*$ by Remark \ref{re:adjacent}, is simply $\gamma_M^+(\chi_{\sigma_{\bbs,\infty}})$. Since $\omega_\sigma\res_{Z(\dA_F)}=\omega$, $\omega_\sigma$ is invariant under $Z_M(F)\cdot(K_M\cap Z_M(\dA_F^\infty))$, and $Z_M(F)\backslash Z_M(\dA_F)^1/K_M\cap Z_M(\dA_F^\infty)$ is a compact abelian real Lie group, there exists a lattice $L_M$ of $\fc_{M,\dC}^*/\fa_{M,\dC}^*$ such that $\gamma_M^+(\chi_{\sigma_{\bbs,\infty}})$ is contained in $L_M$ for every $\sigma\in\fC(M,\omega;K_M,\fT_M)^\heartsuit$. Applying Corollary \ref{co:lattice} to $U=\fc_M^*/\fa_M^*$, $L=L_M$, and $A=\gamma_M^+(\chi_\infty)$, we obtain a holomorphic function $\nu^+$ on $\fc_{M,\dC}^*/\fa_{M,\dC}^*$ that has moderate vertical growth, vanishes on $L_M\setminus\gamma_M^+(\chi_\infty)$, and is nowhere vanishing on $\gamma_M^+(\chi_\infty)$. On the other hand, similar to the special case, we have a holomorphic function $\nu^-$ on $\fb_{M,\dC}^*$ that has moderate vertical growth, and vanishes exactly on $\lambda_M^{-1}(\Lambda(M,\omega;K_M,\fT_M)\setminus\lambda_M(\gamma_M^-(\chi_\infty)))$. Put
\[
\nu\coloneqq\nu^+\times\nu^-\colon\fc_{M,\dC}^*/\fa_{M,\dC}^*\oplus\fb_{M,\dC}^*\to\dC,
\]
which we regard as a function on $\fh_\dC^*$ via the quotient map $\fh_\dC^*\to\fh_\dC^*/\fa_{M,\dC}^*$. As in the special case, we put
\[
\mu_\infty^M\coloneqq\mu_\infty^0\cdot\prod_{w\in\sfW}\nu\circ w,
\]
which belongs to $\cM^\sharp_\theta(\fh_\dC^*)^\sfW$ and satisfies the requirements in the lemma.
\end{proof}

By definition, we have an identification $\coprod_{M}\fC(M,\omega)=\fD(G,\omega)$, which induces a surjective map
\begin{align}\label{eq:component}
\coprod_{M}\fC(M,\omega)^\heartsuit\to\fD(G,\omega)^\heartsuit
\end{align}
by passing to the equivalence relations. We denote by $\fD(G,\omega;K,\fT)^\heartsuit$ the image of
\[
\coprod_{M}\fC(M,\omega;K_M,\fT_M)^\heartsuit
\]
under the map \eqref{eq:component}, and $\fD(G,\omega;K,\fT)^\heartsuit_{\chi_\infty}$ the image of
\[
\coprod_{M}\fC(M,\omega;K_M,\fT_M)^\heartsuit\cap\fC(M,\omega)^\heartsuit_{\chi_\infty}
\]
under the map \eqref{eq:component}. Now we can construct the desired function $\mu_\infty$ in the following proposition.

\begin{proposition}\label{pr:isolation1}
The function
\[
\mu_\infty\coloneqq\mu_\infty^0\cdot\prod_{M}\mu^M_\infty,
\]
in which the product is taken over standard Levi subgroups of $G$ and $\mu^M_\infty$ is a function obtained from Lemma \ref{le:finite3},\footnote{In fact, it suffices to take the product over a set of representatives of standard Levi subgroups of $G$ with respect to the conjugation action of $\sfW(G,M_0)$.} belongs to $\cM^\sharp_\theta(\fh_\dC^*)^\sfW$ and satisfies $\mu_\infty(\chi_\infty)\neq 0$. Moreover, for every $f\in\cS(G(\dA_F))_K$, the endomorphism $\rR(\mu_\infty\star f)$ of $L^2(G(F)\backslash G(\dA_F)/K,\omega)$ annihilates the subspace $L^2_{(M,\sigma)}(G(F)\backslash G(\dA_F)/K,\omega)$ if $(M,\sigma)$ does not belong to $\fD(G,\omega;K,\fT)^\heartsuit_{\chi_\infty}$.
\end{proposition}

Note that by Lemma \ref{le:finite2}, $\fD(G,\omega;K,\fT)^\heartsuit_{\chi_\infty}$ is a finite set.

\begin{proof}
The first statement is clear from the construction.

By the description of $L^2_{(M,\sigma)}(G(F)\backslash G(\dA_F)/K,\omega)$ recalled in the last subsection, it suffices to show that $\rR(\mu_\infty\star f)E(-,\Phi)=0$ for every Paley--Wiener functions $\Phi$. By Remark \ref{re:continuous}, it suffices to consider the case where $f$, hence $\mu_\infty\star f$, are bi-$K_{0,\infty}$-finite. Then since $\rR(\mu_\infty\star f)E(-,\Phi)=E(-,\rR(\mu_\infty\star f)\Phi)$, it suffices to show that $\Ind^G_{P_M}(\sigma_{\bbs})(\mu_\infty\star f)=0$ for every $\bbs\in\fa_{M,\dC}^*$. Note that we have
\[
\Ind^G_{P_M}(\sigma_{\bbs})(\mu_\infty\star f)=\mu_\infty(\chi_{\sigma_{\bbs,\infty}}^G)\cdot\Ind^G_{P_M}(\sigma_{\bbs})(f).
\]
Take an element $(M,\sigma)\in\fD(G,\omega)$ whose image in $\fD(G,\omega)^\heartsuit$ does not belong to $\fD(G,\omega;K,\fT)^\heartsuit_{\chi_\infty}$. If $\sigma^{K_M}=\{0\}$, then we have $\Ind^G_{P_M}(\sigma_{\bbs})(f)=0$ by our choice of $K_M$. If $\sigma\res_{K^M_{0,\infty}}$ does not contain any member from $\fT_M$, then we have $\mu_\infty(\chi_{\sigma_{\bbs,\infty}}^G)=\mu_\infty^0(\chi_{\sigma_{\bbs,\infty}}^G)=0$ for every $\bbs\in\fa_{M,\dC}^*$ by our choice of $\fT_M$. Otherwise, we must have $\mu_\infty(\chi_{\sigma_{\bbs,\infty}}^G)=0$ for every $\bbs\in\fa_{M,\dC}^*$ by the property of $\mu^M_\infty$ from Lemma \ref{le:finite3}. The proposition follows.
\end{proof}

\subsection{Proof of results}
\label{ss:isolation_proof}

To annihilate components $L^2_{(M,\sigma)}(G(F)\backslash G(\dA_F)/K,\omega)$ in which $\sigma$ is possibly $\chi_\infty$-typical, we need mixed multipliers from both archimedean and nonarchimedean places.

\begin{definition}\label{de:occur}
For an element $(M,\sigma)\in\fD(G,\omega)$, we say that the $\tT$-character \emph{$\chi$ occurs in $(M,\sigma)$} if there exists $\bbs\in\fa_{M,\dC}^*$ such that
\begin{itemize}
  \item $\chi^G_{\sigma_{\bbs},\infty}=\chi_\infty$;

  \item for every prime $v$ of $F$ not in $\tT$, the spherical representation corresponding to $\chi_v$ is a constituent of $\Ind^G_{P_M}(\sigma_{\bbs,v})$.
\end{itemize}
It is clear that such property depends only on the equivalence class of $(M,\sigma)$ in $\fD(G,\omega)^\heartsuit$. We denote by $\fD(G,\omega)^\heartsuit_\chi$ the subset of $\fD(G,\omega)^\heartsuit$ in which $\chi$ occurs.
\end{definition}

We need to annihilate every element $(M,\sigma)\in\fD(G,\omega)$ in which $\chi$ does not occur, which is possible by the following proposition.

\begin{proposition}\label{pr:isolation2}
Let $(M,\sigma)$ be an element in $\fD(G,\omega)$. Assume that $\chi$ does not occur in $(M,\sigma)$. Then there exists a $\tT$-multiplier $\mu_{(M,\sigma)}$ of $\cS(G(\dA_F))_K$ satisfying
\begin{enumerate}
  \item for every $f\in\cS(G(\dA_F))_K$, the endomorphism $\rR(\mu_{(M,\sigma)}\star f)$ of $L^2(G(F)\backslash G(\dA_F)/K,\omega)$ annihilates $L^2_{(M,\sigma)}(G(F)\backslash G(\dA_F)/K,\omega)$;

  \item $\mu_{(M,\sigma)}(\chi)\neq 0$.
\end{enumerate}
\end{proposition}

\begin{proof}
We may assume that $\sigma^{K_M}\neq\{0\}$; otherwise $L^2_{(M,\sigma)}(G(F)\backslash G(\dA_F)/K,\omega)=\{0\}$, hence the proposition is trivial. Then for every $\bbs\in\fa_{M,\dC}^*$ and every prime $v$ of $F$ not in $\tT$, we have $\rI^G_{P_M}(\sigma_{\bbs,v})^{K_v}\neq\{0\}$, which gives rise to a character $\chi^G_{\sigma_{\bbs,v}}\colon\cH_{G,v}\to\dC$. Then we have a character
\[
\chi^G_{\sigma^{\infty,\tT}_{\bbs}}\colon\cH_G^\tT\to\dC.
\]
We also recall the infinitesimal character $\chi^G_{\sigma_{\bbs,\infty}}$ of $\rI^G_{P_M}(\sigma_{\bbs,\infty})$, which is simply the $\sfW$-orbit of $\chi_{\sigma_{\bbs,\infty}}$ in $\fh_\dC^*$. Thus, we obtain a $\tT$-character $\chi_{\sigma_{\bbs}^\tT}^G=(\chi^G_{\sigma_{\bbs,\infty}},\chi^G_{\sigma^{\infty,\tT}_{\bbs}})$ for $G$. We suppress the superscript $G$ when $M=G$.

We first consider the easy case where $M=G$. Since $\chi$ does not occur in $(G,\sigma)$, we have either $\chi_{\sigma_\infty}\neq\chi_\infty$ or $\chi_{\sigma^{\infty,\tT}}\neq\chi^{\infty,\tT}$. In the first case, by Lemma \ref{le:multiplier1}, we can find an element $\mu_{(M,\sigma)}\in\cM^\sharp_\theta(\fh_\dC^*)^\sfW$ satisfying $\mu_{(M,\sigma)}(\chi_\infty)\neq 0$. After multiplying by a suitable element in $\dC[\fh_\dC^*]^\sfW$, we may further require that $\mu_{(M,\sigma)}(\chi_{\sigma_\infty})=0$. In the second case, we can certainly find an element $\mu_{(M,\sigma)}\in\cH_G^\tT$ satisfying $\mu_{(M,\sigma)}(\chi_{\sigma^{\infty,\tT}})=0$ but $\mu_{(M,\sigma)}(\chi^{\infty,\tT})\neq 0$. The proposition follows.

We now consider the hard case where $M$ is a proper standard Levi subgroup. Recall that $\fa_M^*$ is canonically a subspace of $\fh^*$. We fix
\begin{itemize}
  \item an element $\alpha_\sigma$ in $\chi_{\sigma_\infty}\subseteq\fh_\dC^*$,

  \item an element $\alpha_\chi$ in $\chi_\infty\subseteq\fh_\dC^*$,

  \item a linear splitting map $\ell\colon\fh^*\to\fa_M^*$ of the subspace $\fa_M^*\subseteq\fh^*$.
\end{itemize}
For $w\in \sfW$, put $\bbs_w\coloneqq\ell(w\alpha_\chi)-\ell(\alpha_\sigma)\in\fa_{M,\dC}^*$. Since $\chi$ does not occur in $(M,\sigma)$, for every $w\in\sfW$, we can take $v[w]$ to be either $\infty$ or a prime of $F$ not in $\tT$ such that $\chi^G_{\sigma_{\bbs_w,v[w]}}\neq\chi_{v[w]}$. It allows us to choose an element $\nu_w\in\cH_{G,v[w]}$ (we regard $\cH_{G,\infty}$ as $\dC[\fh_\dC^*]^\sfW$) such that
\begin{align}\label{eq:isolation3}
\nu_w(\chi_{v[w]})\neq\nu_w(\chi^G_{\sigma_{\bbs_w,v[w]}}).
\end{align}

Note that for every $w'\in \sfW$, the function $\alpha\mapsto\nu_w(\chi^G_{\sigma_{\ell(w'\alpha)-\ell(\alpha_\sigma),v[w]}})$ is of exponential type, hence belongs to $\cO_\expo(\fh_\dC^*)$. We put
\[
\nu^\dag_{w,w'}\coloneqq\nu_w-\nu_w(\chi^G_{\sigma_{\ell\circ w'(-)-\ell(\alpha_\sigma),v[w]}}),
\]
regarded as an element in $\cO_\expo(\fh_\dC^*)\otimes\cH_G^\tT$. Then \eqref{eq:isolation3} simply says $\nu^\dag_{w,w}(\alpha_\chi,\chi^{\infty,\tT})\neq 0$. From this, it is elementary to see that there exist complex constants $\{C_w\}_{w\in \sfW}$ such that
\[
\sum_{w\in \sfW}C_w\nu^\dag_{w,w'}(\alpha_\chi,\chi^{\infty,\tT})\neq 0
\]
for every $w'\in \sfW$. Put
\[
\nu^\dag_{w'}\coloneqq\sum_{w\in \sfW}C_w\nu^\dag_{w,w'}.
\]
Then put
\[
\nu^\dag\coloneqq\prod_{w'\in \sfW}\nu^\dag_{w'},
\]
which is an element in $\cO_\expo(\fh_\dC^*)^{\sfW}\otimes\cH_G^\tT$, satisfying $\nu^\dag(\chi)\neq 0$.

Now we claim that $\nu^\dag(\chi^G_{\sigma^\tT_{\bbs}})=0$ for every $\bbs\in\fa_{M,\dC}^*$. Note that the element $\alpha_\sigma+\bbs\in\fh_\dC^*$ belongs to $\chi^G_{\sigma_{\bbs,\infty}}$. Since we have
\begin{align*}
\nu^\dag_1(\alpha_\sigma+\bbs,\chi^G_{\sigma^{\infty,\tT}_{\bbs}})
&=\sum_{w\in \sfW}C_w\nu^\dag_{w,1}(\alpha_\sigma+\bbs,\chi^G_{\sigma^{\infty,\tT}_{\bbs}}) \\
&=\sum_{w\in\sfW}
C_w\(\nu_w(\chi^G_{\sigma_{\bbs,v[w]}})-\nu_w(\chi^G_{\sigma_{\ell(\alpha_\sigma+\bbs)-\ell(\alpha_\sigma),v[w]}})\)=0,
\end{align*}
the claim follows. It is easy to find an element $\mu^\dag\in\cM(\fh_\dC^*)^\sfW$ satisfying $\mu^\dag(\chi_\infty)\neq 0$, and such that
$\mu^\dag\cdot\nu^\dag$ belongs to $\cM(\fh_\dC^*)^{\sfW}\otimes\cH_G^\tT$. Then by Remark \ref{re:multiplier}(3),
\[
\mu_{(M,\sigma)}\coloneqq\mu_\infty^0\cdot\mu^\dag\cdot\nu^\dag
\]
is a $\tT$-multiplier of $\cS(G(\dA_F))_K$, satisfying $\mu_{(M,\sigma)}(\chi)\neq 0$.

It remains to show that for every $f\in\cS(G(\dA_F))_K$, the endomorphism $\rR(\mu_{(M,\sigma)}\star f)$ annihilates $L^2_{(M,\sigma)}(G(F)\backslash G(\dA_F)/K,\omega)$. However, this follows from the same argument as in the proof of Proposition \ref{pr:isolation1} as we have $\mu_{(M,\sigma)}(\chi^G_{\sigma^\tT_{\bbs}})=0$ for every $\bbs\in\fa_{M,\dC}^*$. Thus, the proposition follows.
\end{proof}

\begin{remark}
The idea of constructing the element $\nu^\dag$ in the proof of Proposition \ref{pr:isolation2} is inspired by \cite{LV07}. The similar construction also appeared in \cite{YZ} in the case where $G=\PGL_2$ and $F$ is a function field.
\end{remark}

The following theorem is the most general form on isolating cuspidal components, from which we will deduce Theorem \ref{th:isolation} and Theorem \ref{th:isolation_bis}.

\begin{theorem}\label{th:isolation_tri}
Let $\chi$ be a $\tT$-character. There exists a $\tT$-multiplier $\mu$ of $\cS(G(\dA_F))_K$ such that
\begin{enumerate}
  \item for every $f\in\cS(G(\dA_F))_K$, the image of the endomorphism
      \[
      \rR(\mu\star f)\colon L^2(G(F)\backslash G(\dA_F)/K,\omega)\to L^2(G(F)\backslash G(\dA_F)/K,\omega)
      \]
      is contained in
      \[
      \bigoplus_{(M,\sigma)\in\fD(G,\omega)^\heartsuit_\chi}L^2_{(M,\sigma)}(G(F)\backslash G(\dA_F)/K,\omega)
      \]
      under the decomposition \eqref{de:component}, where $\fD(G,\omega)^\heartsuit_\chi$ is introduced in Definition \ref{de:occur};

  \item $\mu(\chi)=1$.
\end{enumerate}
\end{theorem}

In particular, Theorem \ref{th:isolation_tri_pre} follows by taking $\tT=\tS$ and $\chi=\chi_{\pi^\tT}$, and the strong multiplicity one property for $G=\Res_{F'/F}\GL_n$ \cite{PS79}.

\begin{proof}
By Lemma \ref{le:finite2}, we know that $\fD(G,\omega;K,\fT)^\heartsuit_{\chi_\infty}$ is a finite set. Thus, we may choose a finite subset $\fD\subseteq\fD(G,\omega)$ that maps surjectively to the (finite) set $\fD(G,\omega;K,\fT)^\heartsuit_{\chi_\infty}\setminus\fD(G,\omega)^\heartsuit_\chi$. Now we put
\[
\mu\coloneqq\mu_\infty\cdot\prod_{(M,\sigma)\in\fD}\mu_{(M,\sigma)}
\]
from Proposition \ref{pr:isolation1} and Proposition \ref{pr:isolation2}, which is a $\tT$-multiplier of $\cS(G(\dA_F))_K$. Then (1) is satisfied. For (2), we only need to replace $\mu$ by $\mu(\chi)^{-1}\mu$ as $\mu(\chi)\neq 0$. The theorem follows.
\end{proof}

\begin{proof}[Proof of Theorem \ref{th:isolation} and Theorem \ref{th:isolation_bis}]
Note that the intersection
\[
L^2_\cusp(G(F)\backslash G(\dA_F)/K,\omega)\bigcap
\underset{(M,\sigma)\in\fD(G,\omega)^\heartsuit_\chi}{\widehat\bigoplus}L^2_{(M,\sigma)}(G(F)\backslash G(\dA_F)/K,\omega)
\]
is exactly $L^2_\cusp(G(F)\backslash G(\dA_F)/K,\omega)[\chi]$. Thus, Theorem \ref{th:isolation_bis} immediately follows from Theorem \ref{th:isolation_tri}. Moreover, Theorem \ref{th:isolation} follows from Theorem \ref{th:isolation_tri} and the relation
\[
\underset{(M,\sigma)\in\fD(G,\omega)^\heartsuit_\chi}{\widehat\bigoplus}L^2_{(M,\sigma)}(G(F)\backslash G(\dA_F)/K,\omega)
\subseteq L^2_\cusp(G(F)\backslash G(\dA_F)/K,\omega)
\]
since $\chi$ is not $G$-CAP.
\end{proof}

\begin{remark}\label{re:langlands}
In fact, we may use the Langlands decomposition of the $L^2$-spectrum \cite{Lan76} (as formulated in \cite{Art05}*{Theorem~7.2}) instead of the decomposition \eqref{eq:bernstein}, and \cite{Mul89}*{Theorem~0.1} instead of \cite{Don82}*{Theorem~9.1} in the proof of Lemma \ref{le:casimir}, to obtain a version of Theorem \ref{th:isolation_tri} in terms of the Langlands components, which is slightly stronger than the current version of Theorem \ref{th:isolation_tri}.

More precisely, let $\fD'(G,\omega)$ be the set of pairs $(M,\sigma)$ in which $\sigma$ is now a discrete automorphic representation of $M(\dA_F)$ whose central character belongs to $\Omega_M(\omega)$, with a similarly defined equivalence relation $\approx$. We have the following Langlands decomposition
\[
L^2(G(F)\backslash G(\dA_F),\omega)=\underset{(M,\sigma)\in\fD'(G,\omega)^\heartsuit}{\widehat\bigoplus}
L^2_{(M,\sigma)}(G(F)\backslash G(\dA_F),\omega)',
\]
which refines \eqref{eq:bernstein}, in which $L^2_{(M,\sigma)}(G(F)\backslash G(\dA_F),\omega)'$ denotes the Langlands component associated to $(M,\sigma)$. The conclusion of Theorem \ref{th:isolation_tri} can be strengthened to that there exists a $\tT$-multiplier $\mu$ of $\cS(G(\dA_F))_K$ such that for every $f\in\cS(G(\dA_F))_K$,
\begin{enumerate}
  \item $\rR(\mu\star f)$ maps $L^2(G(F)\backslash G(\dA_F)/K,\omega)$ into
      \[
      \underset{(M,\sigma)\in\fD'(G,\omega)^\heartsuit_\chi}{\widehat\bigoplus}L^2_{(M,\sigma)}(G(F)\backslash G(\dA_F)/K,\omega)'
      \]
      in which $\fD'(G,\omega)^\heartsuit_\chi$ is similarly defined as $\fD(G,\omega)^\heartsuit_\chi$ in Definition \ref{de:occur};

  \item $\mu(\chi)=1$.
\end{enumerate}
Note that we have the inclusion
\[
\underset{(M,\sigma)\in\fD'(G,\omega)^\heartsuit_\chi}{\widehat\bigoplus}L^2_{(M,\sigma)}(G(F)\backslash G(\dA_F)/K,\omega)'
\subseteq
\underset{(M,\sigma)\in\fD(G,\omega)^\heartsuit_\chi}{\widehat\bigoplus}L^2_{(M,\sigma)}(G(F)\backslash G(\dA_F)/K,\omega).
\]
However, it could be strict. For example, when $G=\PGL_2$, $M$ is the standard diagonal Levi subgroup, and $\sigma$ is the trivial character, we have
\[
L^2_{(M,\sigma)}(G(F)\backslash G(\dA_F)/K,1)=L^2_{(M,\sigma)}(G(F)\backslash G(\dA_F)/K,1)'\oplus
L^2_{(G,\tilde\sigma)}(G(F)\backslash G(\dA_F)/K,1)'
\]
where $\tilde\sigma$ is the trivial character of $G(\dA_F)$. It is possible to find $\chi$ such that $(M,\sigma)\in\fD'(G,\omega)^\heartsuit_\chi$, hence $(M,\sigma)\in\fD(G,\omega)^\heartsuit_\chi$, but $(G,\tilde\sigma)\not\in\fD'(G,\omega)^\heartsuit_\chi$. However, one cannot find $\chi$ such that $(G,\tilde\sigma)\in\fD'(G,\omega)^\heartsuit_\chi$ but $(M,\sigma)\not\in\fD'(G,\omega)^\heartsuit_\chi$.
\end{remark}

\section{Application to the Gan--Gross--Prasad conjecture}
\label{ss:4}

In this section, we discuss the application of the results in previous sections to the global Gan--Gross--Prasad conjecture and the Ichino--Ikeda conjecture for $\rU(n)\times\rU(n+1)$ in the stable case. In Subsection \ref{ss:zydor}, we recall the Jacquet--Rallis relative trace formulae and their extension by Zydor. In Subsection \ref{ss:smooth}, we introduce the notion of smooth transfer and study its relation with multipliers. In Subsection \ref{ss:base_change}, we deduce some results concerning weak base change using Jacquet--Rallis relative trace formulae, which are necessary for the proof of the Gan--Gross--Prasad conjecture. In Subsection \ref{ss:ggp}, we prove the Gan--Gross--Prasad conjecture in the stable case and other related results.

Let $E/F$ be a quadratic extension of number fields, and $\eta_{E/F}\colon\dA_F^\times\to\dC^\times$ the associated quadratic automorphic character. Let $n\geq 1$ be an integer.

\subsection{Jacquet--Rallis relative trace formulae, d'apr\`{e}s Zydor}
\label{ss:zydor}

In this subsection, we collect some results from Zydor's extension of the Jacquet--Rallis relative trace formulae. We start from the general linear groups. Put
\[
G'_n\coloneqq\Res_{E/F}\GL_{n,E},\qquad
G'_{n+1}\coloneqq\Res_{E/F}\GL_{n+1,E},\qquad
G'\coloneqq G'_n\times G'_{n+1}.
\]
We have two reductive subgroups $H'_1$ and $H'_2$ of $G'$ as at the beginning of \cite{BP1}*{Section~3}, in which $H'_1$ is the graph of the natural embedding $G'_n\to G'_{n+1}$ via the first $n$ coordinates; and $H'_2$ is the subgroup $\GL_{n,F}\times\GL_{n+1,F}$. Let $G'_\rss\subseteq G'$ be the Zariski open subset of regular semisimple elements. Recall that an element $\gamma$ of $G'$ is regular semisimple if $H'_1\gamma H'_2$ is Zariski closed in $G'$ and the natural map $H'_1\times H'_2\to H'_1\gamma H'_2$ is an isomorphism. Put $B'\coloneqq H'_1\backslash G'/H'_2$, which is an affine variety over $F$; and let $B'_\rss\subseteq B'$ be the image of $G'_\rss$. Let $Z'$ be the maximal $F$-split center of $G'$, which is also the center of $H'_2$. We denote by
\[
\eta\colon H'_2(\dA_F)=\GL_n(\dA_F)\times\GL_{n+1}(\dA_F)\to\dC^\times
\]
the character $(\eta_{E/F}^{n-1}\circ\det)\boxtimes(\eta_{E/F}^n\circ\det)$, which is trivial on $Z'(\dA_F)H'_2(F)$.

For every element $\Pi\in\fC(G',1)$ (Definition \ref{de:component}), there is a distribution $I_\Pi$ on $\cS(G'(\dA_F))$ \cite{BP1}*{Section~3.1} such that for $f'\in\cS(G'(\dA_F))$,
\begin{align}\label{eq:character1}
I_\Pi(f')=\sum_{\phi\in\cB(\Pi)}\(\int_{H'_1(F)\backslash H'_1(\dA_F)}(\Pi(f')\phi)(h_1)\rd h_1\)
\ol{\(\int_{Z'(\dA_F)H'_2(F)\backslash H'_2(\dA_F)}\phi(h_2)\eta(h_2)\rd h_2\)}
\end{align}
for an arbitrary orthonormal basis $\cB(\Pi)$ of $\Pi$ with respect to the Peterson inner product over $G'/Z'$. Here, $\r{d}h_1$ and $\r{d}h_2$ are the Tamagawa measures on $H'_1(\dA_F)$ and $Z'(\dA_F)\backslash H'_2(\dA_F)$, respectively.

For every element $\gamma\in B'(F)$, there is a distribution $I_\gamma$ on $\cS(G'(\dA_F))$ \cite{Zyd}*{Th\'{e}or\`{e}me~5.9}.\footnote{Although \cite{Zyd}*{Th\'{e}or\`{e}me~5.9} only states for test functions in $C^\infty_c(G'(\dA_F))$, its proof works for $\cS(G'(\dA_F))$ as well. We have a similar situation below concerning \cite{Zyd}*{Th\'{e}or\`{e}me~6.6}.} When $\gamma\in B'_\rss(F)$, $I_\gamma$ is defined by relative orbital integrals; more precisely, we have for $f'\in\cS(G'(\dA_F))$,
\begin{align}\label{eq:orbital1}
I_\gamma(f')=O_\eta(\tilde\gamma,f')\coloneqq\int_{H'_1(\dA_F)\times H'_2(\dA_F)}f'(h_1^{-1}\tilde\gamma h_2)\eta(h_2)\rd h_1\rd h_2,
\end{align}
where $\tilde\gamma\in G'_\rss(F)$ is an arbitrary lift of $\gamma$. In general, the distribution $I_\gamma$ is a certain regularization of relative orbital integrals, whose precise definition is complicated and will not be used in our later discussion, hence we will not recall. The following proposition is Zydor's extension of the Jacquet--Rallis trace formula for $G'$.

\begin{proposition}\label{pr:rtf1}
Let $f'\in\cS(G'(\dA_F))$ be a quasi-cuspidal Schwartz test function (Definition \ref{de:quasi_cuspidal}). Then we have the following identity
\[
\sum_{\Pi\in\fC(G',1)}I_\Pi(f')=\sum_{\gamma\in B'(F)}I_\gamma(f')
\]
of absolutely convergent sums.
\end{proposition}

\begin{proof}
This follows from the same proof of \cite{Zyd}*{Th\'{e}or\`{e}me~5.10}.
\end{proof}

Now we consider unitary groups. We denote by $\fV$ the set of isomorphism classes of pairs $V=(V_n,V_{n+1})$ of (non-degenerate) hermitian spaces over $E$, where $V_n$ has rank $n$ and $V_{n+1}=V_n\oplus E.e$ in which $e$ has norm $1$. For every place $v$ of $F$, we have a local analogue $\fV_v$ and a localization map $\fV\to\fV_v$ sending $V$ to $V_v$. For every $V\in\fV$, put
\[
G^V_n\coloneqq\rU(V_n),\qquad
G^V_{n+1}\coloneqq\rU(V_{n+1}),\qquad
G^V\coloneqq G^V_n\times G^V_{n+1}.
\]
We have a reductive subgroup $H^V$ of $G^V$ as at the beginning of \cite{BP1}*{Section~3}, which is the graph of the natural embedding $G^V_n\to G^V_{n+1}$. Let $G^V_\rss\subseteq G^V$ be the Zariski open subset of regular semisimple elements. Recall that an element $\delta^V$ of $G^V$ is regular semisimple if $H^V\delta^V H^V$ is Zariski closed in $G^V$ and the natural map $H^V\times H^V\to H^V\delta^V H^V$ is an isomorphism. Put $B^V\coloneqq H^V\backslash G^V/H^V$, which is an affine variety over $F$; and let $B^V_\rss\subseteq B^V$ be the image of $G^V_\rss$.

For every element $\pi^V\in\fC(G^V,1)$ (Definition \ref{de:component}), there is a distribution $J_{\pi^V}$ on $\cS(G^V(\dA_F))$ \cite{BP1}*{Section~3.1} such that for $f^V\in\cS(G^V(\dA_F))$,
\begin{align}\label{eq:character2}
J_{\pi^V}(f^V)=\sum_{\varphi\in\cB(\pi^V)}\(\int_{H^V(F)\backslash H^V(\dA_F)}(\pi^V(f^V)\varphi)(h)\rd h\)
\ol{\(\int_{H^V(F)\backslash H^V(\dA_F)}\varphi(h)\rd h\)}
\end{align}
for an arbitrary orthonormal basis $\cB(\pi^V)$ of $\pi^V$ with respect to the Peterson inner product over $G^V$. Here, $\r{d}h$ is the Tamagawa measure on $H^V(\dA_F)$.

For every element $\delta^V\in B^V(F)$, there is a distribution $J_{\delta^V}$ on $\cS(G^V(\dA_F))$ \cite{Zyd}*{Th\'{e}or\`{e}me~6.6}. When $\delta^V\in B^V_\rss(F)$, $J_{\delta^V}$ is defined by relative orbital integrals; more precisely, we have for $f^V\in\cS(G^V(\dA_F))$,
\begin{align}\label{eq:orbital2}
J_{\delta^V}(f^V)=O(\tilde\delta^V,f^V)\coloneqq\int_{H^V(\dA_F)\times H^V(\dA_F)}f^V(h_1^{-1}\tilde\delta^V h_2)\rd h_1\rd h_2,
\end{align}
where $\tilde\delta^V\in G^V_\rss(F)$ is an arbitrary lift of $\delta^V$.\footnote{In this case, the map $G^V_\rss(F)\to B^V_\rss(F)$ is not surjective. If $\delta^V$ does not lift, then we simply set $J_{\delta^V}(f^V)=0$.} In general, the distribution $I_\gamma$ is a certain regularization of relative orbital integrals, whose precise definition is complicated and will not be used in our later discussion, hence we will not recall. The following proposition is Zydor's extension of the Jacquet--Rallis trace formula for $G^V$.

\begin{proposition}\label{pr:rtf2}
Let $f^V\in\cS(G^V(\dA_F))$ be a quasi-cuspidal Schwartz test function (Definition \ref{de:quasi_cuspidal}). Then we have the following identity
\[
\sum_{\pi^V\in\fC(G^V,1)}J_{\pi^V}(f^V)=\sum_{\delta^V\in B^V(F)}J_{\delta^V}(f^V)
\]
of absolutely convergent sums.
\end{proposition}

\begin{proof}
This follows from the same proof of \cite{Zyd}*{Th\'{e}or\`{e}me~6.7}.
\end{proof}

The affine varieties $B'$ and $B^V$ are canonically isomorphic (see \cite{Zha1}*{Section~3.1}). For an $F$-algebra $R$, we say that $\gamma\in G'_\rss(R)$ and $\delta^V\in G^V_\rss(R)$ are a \emph{matching pair} if their images in $B'(R)=B^V(R)$ coincide.

\subsection{Smooth transfer of Schwartz test functions}
\label{ss:smooth}

We now discuss the smooth transfer of Schwartz test functions. We keep the notation from the previous subsection.

\begin{definition}
Let $v$ be a nonarchimedean place of $F$ that is unramified in $E$. We say that
\begin{itemize}
  \item $K'_v$ is a \emph{relative hyperspecial maximal subgroup} of $G'(F_v)$ if it is $\GL_n(O_{E_v})\times\GL_{n+1}(O_{E_v})$;

  \item $K^V_v$ is a \emph{relative hyperspecial maximal subgroup} of $G^V(F_v)$ if it is of the form $\rU(L_n)\times\rU(L_{n+1})$ in which $L_n$ is a self-dual lattice of $V_{n,v}$ and $L_{n+1}=L_n\oplus O_{E_v}.e$.
\end{itemize}
\end{definition}

For $G\in\{G',G^V\}$ and every finite set $\Box$ of places of $F$ containing all archimedean ones, we define the algebra
\begin{align*}
\cS(G(F_\Box))\coloneqq\cS(G_\infty)\otimes\bigotimes_{v\in\Box,v\nmid\infty}\cS(G(F_v)).
\end{align*}
The following definition slightly extends the notion of smooth transfer for Schwartz test functions that are not necessarily pure tensors.

\begin{definition}[Smooth transfer]\label{de:transfer}
Take a Schwartz test function $f'\in\cS(G'(\dA_F))$, which is not necessarily a pure tensor.
\begin{enumerate}
  \item Let $\Box$ be a finite set of places of $F$ containing all archimedean ones and take $V\in\fV$. We say that $f'_\Box\in\cS(G'(F_\Box))$ and $f^V_\Box\in\cS(G^V(F_\Box))$ have \emph{matching orbital integrals} if for every matching pair $\gamma\in G'_\rss(F_\Box)$ and $\delta^V\in G^V_\rss(F_\Box)$, we have
      \[
      O(\delta^V,f^V_\Box)=\Omega_\Box(\gamma)\cdot O_\eta(\gamma,f'_\Box).
      \]
      Here, $O(\delta^V,f^V_\Box)$ and $O_\eta(\gamma,f'_\Box)$ are relative orbital integrals defined in the same way as \eqref{eq:orbital2} and \eqref{eq:orbital1} after replacing $\dA_F$ by $F_\Box$, respectively; and $\Omega_\Box(\gamma)\coloneqq\prod_{v\in\Box}\Omega_v(\gamma)\in\{\pm1\}$ is a certain transfer factor (see \cite{BP1}*{Sections~3.3~\&~3.4} for more details).

  \item For $V\in\fV$ and a Schwartz test function $f^V\in\cS(G^V(\dA_F))$, we say that $f'$ and $f^V$ are \emph{smooth transfer} if there exists a sufficiently large finite set $\Box$ of places of $F$ containing all archimedean ones and those ramified in $E$ such that
      \begin{enumerate}
        \item $f'=f'_\Box\otimes\bigotimes_{v\not\in\Box}\mathbf{1}_{K'_v}$ with $f'_\Box\in\cS(G'(F_\Box))$ and $K^\prime_v$ a relative hyperspecial maximal subgroup of $G'(F_v)$;

        \item $f^V=f^V_\Box\otimes\bigotimes_{v\not\in\Box}\mathbf{1}_{K^V_v}$ with $f^V_\Box\in\cS(G^V(F_\Box))$ and $K^V_v$ a relative hyperspecial maximal subgroup of $G^V(F_v)$;

        \item $f'_\Box\in\cS(G'(F_\Box))$ and $f^V_\Box\in\cS(G^V(F_\Box))$ have matching orbital integrals in the sense of (1).
      \end{enumerate}

  \item Given a collection $(f^V)_{V\in\fV}$ with $f^V\in\cS(G^V(\dA_F))$ among which all but finitely many are zero, we say that $f'$ and $(f^V)_{V\in\fV}$ are \emph{complete smooth transfer} if $f'$ and $f^V$ are smooth transfer for every $V\in\fV$.
\end{enumerate}
\end{definition}

\begin{remark}\label{re:rfl}
By the relative fundamental lemma \cites{Yun11,BP3}, Definition \ref{de:transfer}(2) is insensitive to the choice of $\Box$.
\end{remark}

The existence of smooth transfer is ensured by the following proposition.

\begin{proposition}\label{pr:transfer1}
Let $v$ be a place of $F$.
\begin{enumerate}
  \item There is a dense subspaces $\cS_\tr(G'(F_v))$ of $\cS(G'(F_v))$, equal to $\cS(G'(F_v))$ if $v$ is nonarchimedean or split in $E$, such that for every $f'_v\in\cS_\tr(G'(F_v))$, one can find $(f^{V_v}_v)_{V_v\in\fV_v}$ with $f^{V_v}_v\in\cS(G^{V_v}(F_v))$ such that $f'_v$ and $f^{V_v}_v$ have matching orbital integrals for every $V_v\in\fV_v$.

  \item For every $V_v\in\fV_v$, there is a dense subspace $\cS_\tr(G^{V_v}(F_v))$ of $\cS(G^{V_v}(F_v))$, equal to $\cS(G^{V_v}(F_v))$ if $v$ is nonarchimedean or split in $E$, such that for every $f^{V_v}_v\in\cS_\tr(G^{V_v}(F_v))$, one can find $f'_v\in\cS(G'(F_v))$ such that $f'_v$ and $f^{W_v}_v$ have matching orbital integrals for every $W_v\in\fV_v$ where $f^{W_v}_v=0$ for $W_v\neq V_v$.
\end{enumerate}
\end{proposition}

\begin{proof}
This follows from the combination of \cite{Zha1}*{Theorem~2.6} (for $v$ nonarchimedean) and \cite{Xue}*{Theorem~2.7} (for $v$ archimedean).
\end{proof}

The following proposition is a deep theorem of Chaudouard and Zydor.

\begin{proposition}\label{pr:transfer2}
Suppose that $f'$ and $(f^V)_{V\in\fV}$ are complete smooth transfer in the sense of Definition \ref{de:transfer}(3). Then we have
\[
\sum_{\gamma\in B'(F)}I_\gamma(f')=\sum_{V\in\fV}\sum_{\delta^V\in B^V(F)}J_{\delta^V}(f^V),
\]
in which the summation over $V$ is in fact finite.
\end{proposition}

\begin{proof}
This is \cite{CZ}*{Th\'{e}or\`{e}me~16.2.4.1}. Note that the authors in \cite{CZ} only prove this for pure tensor Schwartz test functions, but their proof works for general Schwartz test functions as well, using the fact that $\bigotimes_{v\mid\infty}\cS(G'(F_v))$ and $\bigotimes_{v\mid\infty}\cS(G^V(F_v))$ are dense in $\cS(G'_\infty)$ and $\cS(G^V_\infty)$, respectively, by \cite{AG10}*{Corollary~2.6.3}.
\end{proof}

Now we study smooth transfer of Schwartz test functions under the action of $\tT$-multipliers. Let $\tT_0$ be the set of primes of $F$ that are \emph{nonsplit} in $E$. We take
\begin{itemize}
  \item a finite set $\tS$ of primes of $F$ containing those ramified in $E$,

  \item an open compact subgroup $K'\subseteq G'(\dA_F^\infty)$ of the form $K'=K'_\tS\times\prod_{v\not\in\tS}K'_v$ in which $K'_v$ is a relative hyperspecial maximal subgroup of $G'(F_v)$ for every prime $v\not\in\tS$,

  \item an element $V\in\fV$, and

  \item an open compact subgroup $K^V\subseteq G^V(\dA_F^\infty)$ of the form $K^V=K^V_\tS\times\prod_{v\not\in\tS}K^V_v$ in which $K^V_v$ is a relative hyperspecial maximal subgroup of $G^V(F_v)$ for every prime $v\not\in\tS$.
\end{itemize}
Put $\tT\coloneqq\tS\cup\tT_0$.

The connected real reductive group $G'\otimes_\dQ\dR$ determines a root datum $(\sfX^{\prime*},\Phi',\sfX_*^\prime,\Phi^{\prime\vee})$ and a subset $\theta'\subseteq\Aut(\sfX^{\prime*},\Phi',\sfX_*^\prime,\Phi^{\prime\vee})^\heartsuit$ as at the beginning of Subsection \ref{ss:multiplier_algebra}, together with $\fh^{\prime*}$ and $\sfW'$ from Subsection \ref{ss:multiplier_functions}. Similarly, for $G^V\otimes_\dQ\dR$, we have corresponding objects $(\sfX^{V*},\Phi^V,\sfX_*^V,\Phi^{V\vee})$, $\theta^V$, $\fh^{V*}$, and $\sfW^V$.\footnote{In fact, the objects $(\sfX^{V*},\Phi^V,\sfX_*^V,\Phi^{V\vee})$, $\theta^V$, $\fh^{V*}$, and $\sfW^V$ do not depend on $V$. However, we still keep $V$ in superscripts in order to make notation more consistent.} By the base change homomorphism on the dual groups, we have a canonical map $\fh^{V*}\to\fh^{\prime*}$ that is injective, by which we will identify $\fh^{V*}$ as a subspace of $\fh^{\prime*}$. Moreover, taking restriction induces a ring homomorphism
\begin{align}\label{eq:bc1}
\sf{bc}^V_\infty\colon\cO(\fh^{\prime*}_\dC)^{\sfW'}\to\cO(\fh^{V*}_\dC)^{\sfW^V}.
\end{align}
For primes away from $\tT$, we have a similar homomorphism
\begin{align}\label{eq:bc2}
\sf{bc}_V^\tT\colon\cH_{G'}^\tT\to\cH_{G^V}^\tT
\end{align}
given by unramified base change (which is simply a convolution product since primes away from $\tT$ are all split in $E$). Taking tensor product, we obtain a homomorphism
\begin{align}\label{eq:bc3}
\sf{bc}^V\coloneqq\sf{bc}^V_\infty\otimes\sf{bc}_V^\tT
\colon\cO(\fh^{\prime*}_\dC)^{\sfW'}\otimes\cH_{G'}^\tT\to\cO(\fh^{V*}_\dC)^{\sfW^V}\otimes\cH_{G^V}^\tT.
\end{align}

\begin{proposition}\label{pr:transfer3}
Let $f'=\bigotimes_vf'_v\in\cS(G'(\dA_F))_{K'}$ and $f^V=\bigotimes_vf^V_v\in\cS(G^V(\dA_F))_{K^V}$ be two pure tensor Schwartz test functions such that $f'_v$ and $f^V_v$ have matching orbital integrals for every place $v$ of $F$. Let $\mu'$ and $\mu^V$ be $\tT$-multipliers of $\cS(G'(\dA_F))_{K'}$ and $\cS(G^V(\dA_F))_{K^V}$ (Definition \ref{de:T_multiplier}), respectively, such that $\mu^V=\sf{bc}^V(\mu')$. Then $\mu'\star f'$ and $\mu^V\star f^V$ are smooth transfer in the sense of Definition \ref{de:transfer}(2).
\end{proposition}

Informally speaking, the proposition asserts that multipliers that are compatible under base change preserve smooth transfer.

\begin{proof}
Let $\Box$ be a sufficiently large finite set of places of $F$ containing all archimedean ones and disjoint from $\tT$, so that $\mu'$ and $\mu^V$ are of the form $\mu'_\Box\otimes\(\otimes_{v\not\in\tT\cup\Box}\mathbf{1}_{K'_v}\)$ and $\mu^V_\Box\otimes\(\otimes_{v\not\in\tT\cup\Box}\mathbf{1}_{K^V_v}\)$ respectively. Below we only care about $\Box$-components, hence will write $\mu'$ for $\mu'_\Box$ and $\mu^V$ for $\mu^V_\Box$, respectively.

Let $\cS(G'(F_\Box))_{K'}$ be the subalgebra of $\cS(G'(F_\Box))$ consisting of functions that are bi-invariant under $\prod_{v\in\Box,v\nmid\infty}K'_v$, and similarly for $\cS(G^V(F_\Box))_{K^V}$. To prove the proposition, it suffices to show that if $f'\in\cS(G'(F_\Box))_{K'}$ and $f^V\in\cS(G^V(F_\Box))_{K^V}$ have matching orbital integrals (in the sense of Definition \ref{de:transfer}(1)), then $\mu'\star f'$ and $\mu^V\star f^V$ have matching orbital integrals as well. In fact, by Lemma \ref{le:matching} below, $f'$ and $f^V$ have matching relative characters. Since $\mu^V=\sf{bc}^V(\mu')$, by Theorem \ref{th:multiplier}, we know that $\mu'\star f'$ and $\mu^V\star f^V$ also have matching relative characters. Thus by Lemma \ref{le:matching} below again, we know that $\mu'\star f'$ and $\mu^V\star f^V$ have matching orbital integrals. The proposition follows.
\end{proof}

\begin{lem}\label{le:matching}
Let $\Box$ be a finite set of places of $F$ containing all archimedean ones, in which all primes are split in $E$. Take $f'\in\cS(G'(F_\Box))$ and $f^V\in\cS(G^V(F_\Box))$. Then $f'$ and $f^V$ have matching orbital integrals (in the sense of Definition \ref{de:transfer}(1)) if and only if they have matching relative characters, that is,
\[
\kappa_{V_\Box}J_\pi(f^V)=I_{\BC(\pi)}(f')
\]
for every $\pi\in\r{Temp}_{H^V_\Box}(G^V_\Box)$, where we follow the notations in \cite{BP2}*{Theorem~5.4.1}.
\end{lem}

The term $\kappa_{V_\Box}\coloneqq\prod_{v\in\Box}\kappa_{V_v}$ is a constant depending only on $V_\Box$; and the distributions $J_\pi(f^V)$ and $I_{\BC(\pi)}(f')$ are local analogues of \eqref{eq:character2} and \eqref{eq:character1}, which depend linearly on $\pi(f^V)$ and $\BC(\pi)(f')$, respectively.

\begin{proof}
As $\bigotimes_{v\mid\infty}\cS(G'(F_v))$ is dense in $\cS(G'_\infty)$ by \cite{AG10}*{Corollary~2.6.3}, applying \cite{BP2}*{(5.5.10)} to local fields $F_v$ for $v\in\Box$ and by continuity, we obtain a similar identity
\begin{align}\label{eq:transfer1}
\int_{B'_\rss(F_\Box)}O_\eta(\gamma,f'_1)\overline{O_\eta(\gamma,f'_2)}\rd\gamma
=|\tau|_{E_\Box}^{-n(n-1)/2}\int_{\r{Temp}(G^{\r{qs}}_\Box)/\r{stab}}I_{\BC(\pi)}(f'_1)\overline{I_{\BC(\pi)}(f'_2)}
\frac{|\gamma^*(0,\pi,\r{Ad},\psi')|}{|S_\pi|}\rd\pi
\end{align}
for every $f'_1,f'_2\in\cS(G'(F_\Box))$. Similarly, as $\bigotimes_{v\mid\infty}\cS(G^V(F_v))$ is dense in $\cS(G^V_\infty)$ by \cite{AG10}*{Corollary~2.6.3}, applying \cite{BP2}*{(5.5.3)} to local fields $F_v$ for $v\in\Box$ and by continuity, we obtain a similar identity
\begin{align}\label{eq:transfer2}
\int_{B^V_\rss(F_\Box)}O(\delta^V,f^V_1)\overline{O(\delta^V,f^V_2)}\rd\delta^V
=\int_{\r{Temp}_{H^V_\Box}(G^V_\Box)}J_\pi(f^V_1)\overline{J_\pi(f^V_2)}\mu_{G^V_\Box}^*(\pi)\rd\pi
\end{align}
for every $f^V_1,f^V_2\in\cS(G^V(F_\Box))$.\footnote{We remark that in \eqref{eq:transfer1}, the product $O_\eta(\gamma,f'_1)\overline{O_\eta(\gamma,f'_2)}$ descends to a function on $B'_\rss(F_\Box)$; while in \eqref{eq:transfer2}, both $O(\delta^V,f^V_1)$ and $\overline{O(\delta^V,f^V_2)}$ descend to functions on $B^V_\rss(F_\Box)$.}

We have the following observations.
\begin{enumerate}
  \item On the geometric side, we may identify $B^V_\rss(F_\Box)$ as a subspace, say $B'_\rss(F_\Box)_V$, of $B'_\rss(F_\Box)$ under which the measures are compatible.

  \item On the spectral side, by the local Gan--Gross--Prasad at places in $\Box$ \cite{BP0}, the local Jacquet--Langlands transfer map
      \[
      \r{Temp}_{H^V_\Box}(G^V_\Box)\to\r{Temp}(G^{\r{qs}}_\Box)/\r{stab}
      \]
      is injective, under which the restriction of the measure $\frac{|\gamma^*(0,\pi,\r{Ad},\psi')|}{|S_\pi|}\rd\pi$ coincides with $\mu_{G^V_\Box}^*(\pi)\rd\pi$ by \cite{BP2}*{Theorem~5.4.3}. We also note that $\kappa_{V_\Box}\overline{\kappa_{V_\Box}}=|\tau|_{E_\Box}^{n(n-1)/2}$.
\end{enumerate}
Now we show the lemma.

First, suppose that $f'$ and $f^V$ have matching orbital integrals. Take $g^V\in\bigotimes_{v\in\Box}\cS_\tr(G^V(F_v))$, and let $g'\in\bigotimes_{v\in\Box}\cS(G'(F_v))$ be an element that has matching orbital integrals with $(g^V,0,\dots)$ from Proposition \ref{pr:transfer1}(2). By \cite{BP2}*{Theorem~5.4.1}, $g'$ and $(g^V,0,\dots)$ have matching relative characters as well. In particular, $O_\eta(\gamma,g')=0$ unless $\gamma\in B'_\rss(F_\Box)_V$, and $I_{\BC(\pi)}(g')=0$ unless $\pi\in\r{Temp}_{H^V_\Box}(G^V_\Box)$. Then combining \eqref{eq:transfer1} and \eqref{eq:transfer2} and using the condition that $f'$ and $f^V$ have matching orbital integrals, we obtain
\[
\int_{\r{Temp}_{H^V_\Box}(G^V_\Box)}\kappa_{V_\Box}J_\pi(f^V)\overline{J_\pi(g^V)}\mu_{G^V_\Box}^*(\pi)\rd\pi
=\int_{\r{Temp}_{H^V_\Box}(G^V_\Box)}I_{\BC(\pi)}(f')\overline{J_\pi(g^V)}\mu_{G^V_\Box}^*(\pi)\rd\pi
\]
for every $g^V\in\bigotimes_{v\in\Box}\cS_\tr(G^V(F_v))$. Now we can use the technique in \cite{BP2}*{Section~5.7.3} to separate $\pi$; so that we obtain $\kappa_{V_\Box}J_\pi(f^V)=I_{\BC(\pi)}(f')$ for every $\pi\in\r{Temp}_{H^V_\Box}(G^V_\Box)$. In other words, $f'$ and $f^V$ have matching relative characters.

Second, suppose that $f'$ and $f^V$ have matching relative characters. The proof is similar and we arrive at the identity
\[
\int_{B^V_\rss(F_\Box)}O(\delta^V,f^V)\overline{O(\delta^V,g^V)}\rd\delta^V=
\int_{B^V_\rss(F_\Box)}\Omega_\Box(\gamma(\delta^V))O(\gamma(\delta^V),f')\overline{O(\delta^V,g^V)}\rd\delta^V
\]
for every $g^V\in\bigotimes_{v\in\Box}\cS_\tr(G^V(F_v))$. Here, $\gamma(\delta^V)\in G'_\rss(F_\Box)$ is arbitrary element whose image in $B'(F_\Box)$ coincides with $\delta^V$; and it is clear that $\Omega_\Box(\gamma(\delta^V))O(\gamma(\delta^V),f')$ is independent of the choice of such $\gamma(\delta^V)$. It is easy to see that locally at every given $\delta^V\in B^V_\rss(F_\Box)$, the functions $O(-,g^V)$ for $g^V\in\bigotimes_{v\in\Box}\cS_\tr(G^V(F_v))$ span a dense subset in the $L^2$-space of $B^V_\rss(F_\Box)$. Thus, we have $O(\delta^V,f^V)=\Omega_\Box(\gamma(\delta^V))O(\gamma(\delta^V),f')$ for every $\delta^V\in B^V_\rss(F_\Box)$, as both sides are continuous functions in $\delta^V$. In other words, $f'$ and $f^V$ have matching orbital integrals.

The lemma is proved.
\end{proof}

\begin{remark}\label{re:matching}
Lemma \ref{le:matching} holds by the same proof without assuming that primes in $\Box$ are split in $E$. However, in this case, the argument (more precisely, \cite{BP2}*{Theorem~5.4.1}) relies on results from \cites{Mok15,KMSW}.
\end{remark}

The following proposition will not be used in this article, but might be useful for other purposes. In particular, it shows the existence of Gaussian test functions (in the space of Schwartz functions) in the sense of \cite{RSZ20}*{Definition~7.9}. We record it here as it is essentially a corollary of Proposition \ref{pr:transfer3}.

\begin{proposition}\label{pr:transfer4}
In the situation of Proposition \ref{pr:transfer1}(2), when $v$ is an archimedean place of $F$ that is nonsplit in $E$ and $V_v$ is positive definite, the subspace $\cS_\tr(G^{V_v}(F_v))$ contains $f^{V_v}_v$ as long as $J_{\pi_v}(f^{V_v}_v)=0$ for all but finitely many (tempered) irreducible admissible representations $\pi_v$ of $G^{V_v}(F_v)$.
\end{proposition}

\begin{proof}
Since the question is local, we may just assume that $E$ is imaginary quadratic, hence $v=\infty$ is the unique archimedean place of $F=\dQ$. To ease notation, we suppress $v$ in the proof. By linearity, it suffices to consider $f^V$ such that $J_\pi(f^V)\neq 0$ for exactly one irreducible admissible representation $\pi_0$ of $G^V(F)$ and that $J_{\pi_0}(f^V)=1$.

By Proposition \ref{pr:transfer1}(2), we may find an element $f^V_1\in\cS_\tr(G^V(F))$ satisfying $J_{\pi_0}(f^V)=1$, so that we may take an element $f'_1\in\cS(G'(F))$ such that $f'_1$ and $f^W_1$ have matching orbital integrals for every $W\in\fV$ where $f^W_1=0$ for $W\neq V$. Now since $G^V(F)$ is compact, the set of infinitesimal characters of irreducible representations of $G^V(F)$ is a lattice. It follows easily from Corollary \ref{co:lattice} that we can find an element $\mu'\in\cM^\sharp_\theta(\fh_\dC^*)^\sfW$ such that if we put $\mu^V\coloneqq\sf{bc}^V(\mu')$, then $\mu^V(\chi_\pi)\neq 0$ only when $\pi=\pi_0$ and $\mu^V(\chi_{\pi_0})=1$. By Theorem \ref{th:multiplier}, we have new elements $f^V_2\coloneqq\mu^V\star f^V_1\in\cS(G^V(F))$ and $f'_2\coloneqq\mu'\star f'_1\in\cS(G'(F))$. By Lemma \ref{le:matching}
(for $\Box=\{\infty\}$), $f'_1$ and $f^W_1$ have matching relative characters, so for $f'_2$ and $f^W_2$ where $f^W_2=0$ for $W\neq V$, hence $f'_2$ and $f^W_2$ have matching orbital integrals. Now since $J_\pi(f^V)=J_\pi(f^V_2)$ for every irreducible admissible representation $\pi$ of $G^V(F)$, using Lemma \ref{le:matching} twice, we know that $f'_2$ and $f^W$ have matching orbital integrals for every $W\in\fV$ where $f^W=0$ for $W\neq V$. In other words, $f^V$ belongs to $\cS_\tr(G^V(F))$.

The proposition is proved.
\end{proof}

\subsection{Weak automorphic base change}
\label{ss:base_change}

We keep the setup from the previous subsection.

We denote by $\fV_{(\tS)}$ the subset of $\fV$ consisting of $V$ such that $G^V_v$ is unramified for every prime $v$ of $F$ not in $\tS$. Then $\fV_{(\tS)}$ is a finite set, and for $V,W\in\fV_{(\tS)}$, we have $V_v=W_v\in\fV_v$ for every prime $v$ of $F$ not in $\tS$. We consider
\begin{itemize}
  \item a nonempty subset $\fV_{(\tS)}^\circ$ of $\fV_{(\tS)}$,

  \item a $\tT$-character $\chi'$ for $G'$ (Definition \ref{de:compatible}), and

  \item for each $V\in\fV_{(\tS)}^\circ$, a $\tT$-character $\chi^V$ for $G^V$ (Definition \ref{de:compatible}) such that the base change of $\chi^V$ coincides with $\chi'$, and also fix an open compact subgroup $K^V\subseteq G^V(\dA_F^\infty)$ as in the previous subsection.
\end{itemize}

The following lemma tells us how to find multipliers that are compatible under base change.

\begin{lem}\label{le:modify}
Given a $\tT$-multiplier $\mu'$ of $\cS(G'(\dA_F))_{K'}$ satisfying $\mu'(\chi')=1$ and a $\tT$-multiplier $\mu^V$ of $\cS(G^V(\dA_F))_{K^V}$ satisfying $\mu^V(\chi^V)=1$ for each $V\in\fV_{(\tS)}^\circ$, we can find new $\tT$-multipliers $\tilde\mu'$ of $\cS(G'(\dA_F))_{K'}$ and $\tilde\mu^V$ of $\cS(G^V(\dA_F))_{K^V}$ for each $V\in\fV_{(\tS)}^\circ$ satisfying
\begin{enumerate}
  \item $\tilde\mu'(\chi')=1$ and that $\tilde\mu'$ is a multiple of $\mu'$ by a $\tT$-multiplier of $\cS(G'(\dA_F))_{K'}$;

  \item for every $V\in\fV_{(\tS)}^\circ$, $\tilde\mu^V(\chi^V)=1$, that $\tilde\mu^V$ is a multiple of $\mu^V$ by a $\tT$-multiplier of $\cS(G^V(\dA_F))_{K^V}$, and that $\tilde\mu^V=\sf{bc}^V(\tilde\mu')$.
\end{enumerate}
\end{lem}

\begin{proof}
For every $V\in\fV_{(\tS)}^\circ$, recall the maps $\sf{bc}^V_\infty$ \eqref{eq:bc1}, $\sf{bc}_V^\tT$ \eqref{eq:bc2}, and $\sf{bc}^V$ \eqref{eq:bc3} from the previous subsection, which are all surjective. We choose a section
\[
\sf{cb}_V^\tT\colon\cH_{G^V}^\tT\to\cH_{G'}^\tT
\]
of the linear map $\sf{bc}_V^\tT$ of vector spaces. We may choose a linear splitting map $\ell_V\colon\fh^{\prime*}\to\fh^{V*}$ of the subspace $\fh^{V*}\subseteq\fh^{\prime*}$ that sends every $\sfW'$-orbit into a $\sfW^V$-orbit, which induces a section
\[
\sf{cb}^V_\infty\colon\cO(\fh^{V*}_\dC)^{\sfW^V}\to\cO(\fh^{\prime*}_\dC)^{\sfW'}
\]
of the linear map $\sf{bc}^V_\infty$ of vector spaces. Taking tensor product, we obtain a linear map
\[
\sf{cb}^V\coloneqq\sf{cb}^V_\infty\otimes\sf{cb}_V^\tT\colon
\cO(\fh^{V*}_\dC)^{\sfW^V}\otimes\cH_{G^V}^\tT\to\cO(\fh^{\prime*}_\dC)^{\sfW'}\otimes\cH_{G'}^\tT,
\]
which is a section of $\sf{bc}^V$. To construct the desired $\tT$-multipliers, we define
\[
\tilde\mu'\coloneqq(\mu')^2\cdot\prod_{V\in\fV_{(\tS)}^\circ}\sf{cb}^V((\mu^V)^2),
\qquad\tilde\mu^V\coloneqq\sf{bc}^V(\tilde\mu').
\]
Note that $\sf{bc}^V_\infty$ sends $\cM(\fh^{\prime*}_\dC)^{\sfW'}$ into $\cM(\fh^{V*}_\dC)^{\sfW^V}$; and $\sf{cb}^V_\infty$ sends $\cM(\fh^{V*}_\dC)^{\sfW^V}$ into $\cM(\fh^{\prime*}_\dC)^{\sfW'}$. Then it follows easily from Remark \ref{re:multiplier}(3) that $\tilde\mu'$ and $\tilde\mu^V$ for $V\in\fV_{(\tS)}^\circ$ are $\tT$-multipliers satisfying (1) and (2). The lemma is proved.
\end{proof}

The following proposition reveals a relation between the Gan--Gross--Prasad period integral and weak automorphic base change.

\begin{proposition}\label{pr:bc}
Consider an element $V\in\fV_{(\tS)}$ and an element $\pi^V\in\fC(G^V,1)$. Suppose that we can find a cuspidal automorphic form $\varphi\in L^2(G^V,\pi^V)$ satisfying
\[
\sP(\varphi)\coloneqq\int_{H^V(F)\backslash H^V(\dA_F)}\varphi(h)\rd h\neq 0
\]
where $\r{d}h$ is the Tamagawa measure on $H^V(\dA_F)$. Assume that either one of the following two assumptions holds:
\begin{enumerate}[label=(\alph*)]
  \item $\pi^V$ is not $(G^V,\tT)$-CAP (Definition \ref{de:cap});

  \item there is a prime $v_0$ of $F$ split in $E$ such that $\pi^V_{v_0}$ is supercuspidal.
\end{enumerate}
Then weak automorphic base change of $\pi^V$ (Definition \ref{de:weak}), as an isobaric automorphic representation of $G'(\dA_F)=\GL_n(\dA_E)\times\GL_{n+1}(\dA_E)$, exists. Moreover, if we put $\Pi\coloneqq\BC(\pi^V)=\Pi_n\boxtimes\Pi_{n+1}$, then
\begin{enumerate}
  \item The base change of $\chi_{\pi^V_\infty}$ is $\chi_{\Pi_\infty}$.

  \item In the situation (b), $\Pi_m$ is cuspidal and hermitian (Definition \ref{de:hermitian}) for $m=n,n+1$.
\end{enumerate}

\end{proposition}

\begin{proof}
In the beginning of this subsection, we take $\fV^\circ_{(\tS)}=\{V\}$, $\chi^V=\chi_{(\pi^V)^\tT}$, and $\chi'$ to be the base change of $\chi^V$.

Suppose that weak automorphic base change of $\pi^V$ does not exists. Then $\chi'$ is not $(G',\tT)$-CAP and $L^2_\cusp(G'(F)\backslash G'(\dA_F)/K',1)[\chi']=\{0\}$. By our assumption on the nonvanishing of $\sP$ on $\pi^V$, for every place $v$ of $F$ either archimedean or in $\tS$, we can choose an element $f^V_v\in\cS(G^V(F_v))$ of positive type (see the definition of positive type functions above \cite{Zha1}*{Proposition~2.12}) such that $J_{\pi^V_v}(f^V_v)>0$. Moreover, we may assume that $f^V_{v_0}$ is a supercuspidal matrix coefficient in the situation of (b). For a prime $v$ not in $\tS$, we take $f^V_v\coloneqq\mathbf{1}_{K^V_v}$. Put $f^V\coloneqq\bigotimes_vf^V_v$, and shrink $K^V_\tS$ if necessary so that $f^V\in\cS(G^V(\dA_F))_{K^V}$. By our choice of $f^V$, we have $J_{\pi^V}(f^V)>0$, and $J_\pi(f^V)\geq 0$ for every $\pi\in\fC(G^V,1)$. Note that the set
\[
\fC\coloneqq\{\pi\in\fC(G^V,1)\res L^2(G^V,\pi)\cap L^2_\cusp(G^V(F)\backslash G^V(\dA_F)/K^V,1)[\chi^V]\neq\{0\}\}
\]
is finite, by a well-known result of Harish-Chandra that there are only finitely (up to isomorphism) cuspidal automorphic representations of $G^V(\dA_F)$ with a given infinitesimal character and nontrivial $K_V$-invariants \cite{HC68}. Thus, we may replace $f^V_v$ by an element in the dense subspace $\cS_\tr(G^V(F_v))$ for each archimedean place $v$ nonsplit in $E$, so that
\[
\sum_{\pi\in\fC}J_\pi(f^V)\neq 0.
\]
Now we can apply Proposition \ref{pr:transfer1}(2) to $f^V_v$ for every place $v$ of $F$ either archimedean or in $\tS$, to obtain an element $f'_v\in\cS(G'(F_v))$ as in there. We may also assume that $f'_{v_0}$ is a supercuspidal matrix coefficient in the situation (b). For a prime $v$ not in $\tS$, we take $f'_v\coloneqq\mathbf{1}_{K'_v}$. Put $f'\coloneqq\bigotimes_vf'_v$, which is an element of $\cS(G'(\dA_F))_{K'}$ after shrinking $K'_\tS$ if necessary. Then by the relative fundamental lemma \cites{Yun11,BP3}, we know that $f'$ and $(f^V,0,\dots)$ are complete smooth transfer in the sense of Definition \ref{de:transfer}(3).

We claim that there is a $\tT$-multiplier $\mu'$ of $\cS(G'(\dA_F))_{K'}$ satisfying $\mu'(\chi')=1$ and such that $\rR(\mu'\star f')$ sends $L^2(G'(F)\backslash G'(\dA_F)/K',1)$ into $L^2_\cusp(G'(F)\backslash G'(\dA_F)/K',1)[\chi']$. In the situation (a), this follows from Theorem \ref{th:isolation} since $\chi'$ is not $(G',\tT)$-CAP. In the situation (b), this follows from Theorem \ref{th:isolation_bis}, and the observation that $\rR(\mu'\star f')$ automatically annihilates the orthogonal complement of $L^2_\cusp(G'(F)\backslash G'(\dA_F)/K',1)$ in $L^2(G'(F)\backslash G'(\dA_F)/K',1)$ since $f'_{v_0}$ is a supercuspidal matrix coefficient. Similarly, there is a $\tT$-multiplier $\mu^V$ of $\cS(G^V(\dA_F))_{K^V}$ satisfying $\mu^V(\chi^V)=1$ and such that $\rR(\mu^V\star f^V)$ sends $L^2(G^V(F)\backslash G^V(\dA_F)/K^V,1)$ into $L^2_\cusp(G^V(F)\backslash G^V(\dA_F)/K^V,1)[\chi^V]$. Moreover, by Lemma \ref{le:modify}, we may further assume that $\mu^V=\sf{bc}^V(\mu')$. Thus, by Proposition \ref{pr:transfer3}, $\mu'\star f'$ and $(\mu^V\star f^V,0,\dots)$ are complete smooth transfer as well.

Now we run relative trace formulae from the unitary side to the general linear side. We have
\[
\sum_{\pi\in\fC(G^V,1)}J_\pi(\mu^V\star f^V)=\sum_{\pi\in\fC}J_\pi(\mu^V\star f^V)=\sum_{\pi\in\fC}J_\pi(f^V)\neq 0.
\]
Since $\mu^V\star f^V$ is quasi-cuspidal, we have
\[
\sum_{\delta^V\in B^V(F)}J_{\delta^V}(\mu^V\star f^V)=\sum_{\pi\in\fC(G^V,1)}J_\pi(\mu^V\star f^V)\neq 0
\]
by Proposition \ref{pr:rtf2}. By Proposition \ref{pr:transfer2}, we have
\[
\sum_{\gamma\in B'(F)}I_\gamma(\mu'\star f')=\sum_{\delta^V\in B^V(F)}J_{\delta^V}(\mu^V\star f^V)\neq 0.
\]
Since $\mu'\star f'$ is quasi-cuspidal, we have
\[
\sum_{\tilde\Pi\in\fC(G',1)}I_{\tilde\Pi}(\mu'\star f')=\sum_{\gamma\in B'(F)}I_\gamma(\mu'\star f')\neq 0
\]
by Proposition \ref{pr:rtf1}. However, this is a contradiction since the image of $\rR(\mu'\star f')$ is contained in $L^2_\cusp(G'(F)\backslash G'(\dA_F)/K',1)[\chi']$, which is zero. Therefore, $\fD(G',1)^\heartsuit_{\chi'}\neq\emptyset$, that is, weak automorphic base change of $\pi^V$ exists. It remains to show the two additional claims.

For (1), it follows from the fact that $\fD(G',1)^\heartsuit_{\chi'}\neq\emptyset$.

For (2), note that in the situation (b), we did not use the assumption that $\chi'$ is not $(G',\tT)$-CAP in the main argument above. Thus, we should have $L^2_\cusp(G'(F)\backslash G'(\dA_F)/K',1)[\chi']\neq\{0\}$; in other words, $\Pi\coloneqq\BC(\pi^V)$ is cuspidal and we have $I_\Pi(\mu'\star f')\neq 0$. By \eqref{eq:character1},
\[
\int_{Z'(\dA_F)H'_2(F)\backslash H'_2(\dA_F)}\phi(h_2)\eta(h_2)\rd h_2\neq 0
\]
for some cusp form $\phi$ in $\Pi$. Then by the main theorem of \cite{Fli88}, we know that $L(s,\Pi_m,\As^{(-1)^{m+1}})$ has a pole at $s=1$ for $m=n,n+1$, which implies that $\Pi_m$ is hermitian.

The proposition is proved.
\end{proof}

The last theorem of this subsection contains results on weak automorphic base change, proved using Gan--Gross--Prasad period integrals.

\begin{theorem}\label{th:abc}
Let $V^\circ$ be a hermitian space over $E$ with $G^\circ\coloneqq\rU(V^\circ)$, and $\pi=\otimes_v\pi_v$ an irreducible admissible representation of $G^\circ(\dA_F)$.
\begin{enumerate}
  \item If $\pi$ is cuspidal automorphic satisfying that there exist infinitely many primes $v$ of $F$ split in $E$ such that $\pi_v$ is generic, then weak automorphic base change of $\pi$ exists.

  \item If $\pi$ is cuspidal automorphic satisfying that there exists a prime $v$ of $F$ split in $E$ such that $\pi_v$ is supercuspidal, then weak automorphic base change of $\pi$ exists, and is cuspidal and hermitian (Definition \ref{de:hermitian}).

  \item If weak automorphic base change of $\pi$ exists and is hermitian, then $\pi$ is not $(G^\bullet,\tT_0)$-CAP (Definition \ref{de:cap}) for every pure inner form $G^\bullet$ of $G^\circ$.

  \item If weak automorphic base change of $\pi$ exists and is hermitian, then the base change of $\chi_{\pi_\infty}$ is $\chi_{\BC(\pi)_\infty}$.
\end{enumerate}
\end{theorem}

Part (3) of the above theorem is consistent with the so-called CAP conjecture \cite{Jia10}*{Conjecture~6.1} (for unitary groups).

\begin{proof}
We prove the four statements at once via the induction on the rank of $V^\circ$. The statements are all trivial when the rank of $V^\circ$ is $1$. Assume that (1--4) are all known for all $V^\circ$ of rank at most $n$.

Now take a hermitian space $V^\circ$ of rank $n+1$. By scaling the hermitian form, we may assume that there is an element $V=(V_n,V_{n+1})\in\fV$ such that $V^\circ=V_{n+1}$. We consider (1) and (2) first. Since $\pi$ is cuspidal automorphic, it gives an element $\pi^V_{n+1}\in\fC(G^V_{n+1},1)$. By the Burger--Sarnak trick \cite{Zha1}*{Proposition~2.14}, we can find another element $\pi^V_n\in\fC(G^V_n,1)$ such that
\begin{itemize}
  \item $\sP$ is nonzero on cuspidal automorphic forms in $L^2(G^V,\pi^V)$, where $\pi^V\coloneqq\pi^V_n\boxtimes\pi^V_{n+1}\in\fC(G^V,1)$;

  \item there is a prime $v_0$ of $F$ split in $E$ such that $\pi^V_{n,v_0}$ is supercuspidal, and $\pi^V_{n+1,v_0}$ is generic (resp.\ supercuspidal) in (1) (resp.\ (2)).
\end{itemize}

For (1), there are two cases. If $\pi$ is $(G^\circ,\tT_0)$-CAP, then by definition there exist a proper parabolic subgroup $P^\circ$ of $G^\circ$ and a cuspidal automorphic representation $\sigma$ of $M_{P^\circ}(\dA_F)$, such that $\pi_v$ is a constituent of $\rI^{G^\circ}_{P^\circ}(\sigma_v)$ for all but finitely many primes $v$ of $F$ split in $E$. Write $M_{P^\circ}=\rU(W^\circ)\times M'$ for some hermitian space $W^\circ$ of rank at most $n$ and $M'$ a product of general linear groups, under which $\sigma=\sigma^\circ\boxtimes\sigma'$. Then there exist infinitely many primes $v$ of $F$ split in $E$ such that $\sigma^\circ_v$ is generic. By the induction hypothesis on (1), weak automorphic base change of $\sigma^\circ$ exists, which implies that weak automorphic base change of $\pi$ exists as well. If $\pi$ is not $(G^\circ,\tT_0)$-CAP, then $\pi^V_{n+1}$ is not $(G^V_{n+1},\tT_0)$-CAP. By the induction hypothesis on (2,3), we know that $\pi^V_n$ is not $(G^V_n,\tT_0)$-CAP. Thus, $\pi^V$ is not $(G^V,\tT_0)$-CAP. By the situation (a) of Proposition \ref{pr:bc}, we know that weak automorphic base change of $\pi^V$, hence of $\pi$, exists.

For (2), by the situation (b) of Proposition \ref{pr:bc}, we know that weak automorphic base change $\Pi$ of $\pi^V$, hence of $\pi$, exists and is cuspidal and hermitian.

For (3), we prove by contradiction. If $\pi$ is $(G^\bullet,\tT)$-CAP for some pure inner form $G^\bullet$ of $G^\circ$, then by definition there exist a hermitian space $V^\bullet$ of rank $n+1$ such that $G^\bullet=\rU(V^\bullet)$, a proper parabolic subgroup $P^\bullet$ of $G^\bullet$, and a cuspidal automorphic representation $\sigma$ of $M_{P^\bullet}(\dA_F)$, such that $\pi_v$ is a constituent of $\rI^{G^\bullet}_{P^\bullet}(\sigma_v)$ for all but finitely many primes $v$ of $F$ split in $E$. Write $M_{P^\bullet}=\rU(W^\bullet)\times M'$ for some hermitian space $W^\bullet$ of rank at most $n$ and $M'$ a nontrivial product of general linear groups, under which $\sigma=\sigma^\bullet\boxtimes\sigma'$. Since $\BC(\pi)$ is an isobaric sum of cuspidal automorphic representations with unitary central characters, $\BC(\pi)_w$ is generic for every place $w$ of $E$, which implies that $\sigma^\bullet$ satisfies the assumption in (1). Thus, by the induction hypothesis on (1), weak automorphic base change of $\sigma^\bullet$ exists. Thus, we have $\BC(\pi)\simeq\BC(\sigma^\bullet)\boxplus\Pi\boxplus(\Pi^\vee\circ\rc)\boxplus\cdots$, where $\rc\in\Gal(E/F)$ is the involution, for at least one nontrivial $\Pi$. This contradicts the fact that $\BC(\pi)$ is an isobaric sum of mutually non-isomorphic conjugate self-dual cuspidal automorphic representations, as it is hermitian. Therefore, (3) is proved.

For (4), $\pi$ is not $(G^\circ,\tT_0)$-CAP by (3). In particular, by the induction hypothesis, $\pi^V$ is not $(G^V,\tT_0)$-CAP. Then (4) follows from Proposition \ref{pr:bc}(1).

The theorem is proved.
\end{proof}

\begin{remark}\label{re:base_change}
Theorem \ref{th:abc} already follows from \cites{Mok15,KMSW} (see the end of \cite{KMSW}*{Section~3.3}). However, our proof is different and does not use any knowledge from the endoscopy theory for unitary groups. Furthermore, our method can actually be used to show the local-global compatibility at all places where $\pi_v$ is unramified as well, but the argument will implicitly relies on \cites{Mok15,KMSW} as we will need Lemma \ref{le:matching} for $\Box$ containing primes inert in $E$ (see Remark \ref{re:matching}).
\end{remark}

\subsection{Gan--Gross--Prasad and Ichino--Ikeda conjectures}
\label{ss:ggp}

In this subsection, we complete the proofs of Theorem \ref{th:ggp} (for the Gan--Gross--Prasad conjecture), Theorem \ref{th:nonvanishing}, and Theorem \ref{th:ii} (for the Ichino--Ikeda conjecture). We keep the setup in the previous two subsections.

We start from the following lemma as a preliminary on the descent of hermitian isobaric automorphic representations of $\GL_m(\dA_E)$.

\begin{lem}\label{le:descent}
Let $\Pi$ be a hermitian isobaric automorphic representation of $\GL_m(\dA_E)$ (Definition \ref{de:hermitian}) for some integer $m\geq 1$. Let $V$ be a hermitian space over $E$ of rank $m$ such that $\rU(V)$ is quasi-split. Then there exists a cuspidal automorphic representation $\pi$ of $\rU(V)(\dA_F)$ satisfying that for every prime $v$ of $F$ split in $E$ such that $\Pi_v$ is unramified, $\pi_v$ is unramified such that the base change of $\chi_{\pi_v}$ is $\chi_{\Pi_v}$. In particular, we have $\Pi\simeq\BC(\pi)$ and that the base change of $\chi_{\pi_\infty}$ is $\chi_{\Pi_\infty}$.
\end{lem}

\begin{proof}
The existence of $\pi$ follows from the automorphic descent construction \cite{GRS11}. The last assertion follows from Theorem \ref{th:abc}(4).
\end{proof}

\begin{proof}[Proof of Theorem \ref{th:ggp}]
There are two directions.

(2)$\Rightarrow$(1): We take $V=(V_n,V_{n+1})\in\fV$ and $\pi^V\coloneqq\pi_n\boxtimes\pi_{n+1}\in\fC(G^V,1)$. Put $\Pi\coloneqq\Pi_n\boxtimes\Pi_{n+1}\in\fC(G',1)$. By Theorem \ref{th:abc}(3), $\pi^V$ is not $(G^V,\tT_0)$-CAP. Thus, by the same argument for the situation (a) in Proposition \ref{pr:bc}, we obtain
\[
\sum_{\tilde\Pi\in\fC(G',1)}I_{\tilde\Pi}(\mu'\star f')\neq 0,
\]
in which the image of $\rR(\mu'\star f')$ is contained in $L^2_\cusp(G'(F)\backslash G'(\dA_F)/K',1)[\chi_{\Pi^\tT}]$. However, by the strong multiplicity one property \cite{Ram}*{Theorem~A}, we have $L^2_\cusp(G'(F)\backslash G'(\dA_F)/K',1)[\chi_{\Pi^\tT}]=L^2(G',\Pi)$. Thus, we have $I_\Pi(\mu'\star f')\neq 0$. By \eqref{eq:character1},
\[
\int_{H'_1(F)\backslash H'_1(\dA_F)}\phi(h_1)\rd h_1\neq 0
\]
for some cusp form $\phi$ in $\Pi$. Then by \cite{JPSS}, we have $L(\tfrac{1}{2},\Pi)\neq 0$.

(1)$\Rightarrow$(2): Again, put $\Pi\coloneqq\Pi_n\boxtimes\Pi_{n+1}$. Take $\tS$ to be the finite set of primes $v$ of $F$ at which either $E$ or $\Pi$ is ramified. For $m=n,n+1$, since $\Pi_m$ is hermitian, $L(s,\Pi_m,\As^{(-1)^{m+1}})$ has a pole at $s=1$, which implies that $\Pi_m$ is $\eta_{E/F}^{m+1}$-distinguished by $\GL_m(\dA_F)$. In particular, $\Pi$ belongs to $\fC(G',1)$; and there exists a cusp form $\phi$ in $\Pi$ that is fixed by $\prod_{v\not\in\tS}K'_v$ such that
\[
\int_{Z'(\dA_F)H'_2(F)\backslash H'_2(\dA_F)}\phi(h_2)\eta(h_2)\rd h_2\neq 0.
\]
Since $L(\tfrac{1}{2},\Pi)\neq 0$, by \cite{JPSS}, there exists a cusp form $\phi'$ in $\Pi$ that is fixed by $\prod_{v\not\in\tS}K'_v$ such that
\[
\int_{H'_1(F)\backslash H'_1(\dA_F)}\phi'(h_1)\rd h_1\neq 0.
\]
Together, we can find an element $f'=\bigotimes_vf'_v\in\cS(G'(\dA_F))$ with $f'_v\in\cS_\tr(G'(F_v))$ for every archimedean place $v$ of $F$ and $f'_v=\mathbf{1}_{K^\prime_v}$ for every prime $v$ of $F$ not in $\tS$, such that
\[
I_\Pi(f')\neq 0
\]
by \eqref{eq:character1}. After shrinking $K'_\tS$ if necessary, we may assume $f'\in\cS(G'(\dA_F))_{K'}$. Now we can apply Proposition \ref{pr:transfer1}(1) to $f'_v$ for every place $v$ of $F$ that is either archimedean or in $S$, to obtain elements $f^V_v\in\cS(G^V(F_v))$ as in there. For $V\in\fV_{(\tS)}$ and a prime $v$ not in $\tS$, we put $f^V_v\coloneqq\mathbf{1}_{K^V_v}$, hence obtain an element $f^V\coloneqq\bigotimes_vf^V_v\in\cS(G^V(\dA_F))_{K^V}$ for some $K^V_\tS$. For $V\not\in\fV_{(\tS)}$, we put $f^V=0$. By the relative fundamental lemma \cites{Yun11,BP3}, $f'$ and $(f^V)_{V\in\fV}$ are complete smooth transfer in the sense of Definition \ref{de:transfer}(3).

To pass to unitary groups, we consider $\fV^\circ_{(\tS)}=\fV_{(\tS)}$, $\chi'=\chi_{\Pi^\tT}$, and a $\tT$-character $\chi^V$ for $G^V$ whose base change is $\chi'$ for every $V\in\fV_{(\tS)}$ (which is possible by Lemma \ref{le:descent}), at the beginning of Subsection \ref{ss:base_change}. By the strong multiplicity one property \cite{Ram}*{Theorem~A}, we know that $\chi'$ is not $G'$-CAP. By Theorem \ref{th:isolation}, there is a $\tT$-multiplier $\mu'$ of $\cS(G'(\dA_F))_{K'}$ satisfying $\mu'(\chi')=1$ and such that $\rR(\mu'\star f')$ sends $L^2(G'(F)\backslash G'(\dA_F)/K',1)$ into $L^2_\cusp(G'(F)\backslash G'(\dA_F)/K',1)[\chi']$, which coincides with $L^2(G',\Pi)$. For each $V\in\fV_{(\tS)}$, we know that $\chi^V$ is not $G^V$-CAP, since otherwise $\pi^{V^*}$ would be $(G^V,\tT_0)$-CAP, which contradicts Theorem \ref{th:abc}(3); here $V^*\in\fV$ is the unique element such that $G^{V^*}$ is quasi-split and $\pi^{V^*}$ is a cuspidal automorphic representation of $G^{V^*}(\dA_F)$ as in Lemma \ref{le:descent}. Then by Theorem \ref{th:isolation}, for every $V\in\fV_{(\tS)}$, there is a $\tT$-multiplier $\mu^V$ of $\cS(G^V(\dA_F))_{K^V}$ satisfying $\mu^V(\chi^V)=1$ and such that $\rR(\mu^V\star f^V)$ sends $L^2(G^V(F)\backslash G^V(\dA_F)/K^V,1)$ into $L^2_\cusp(G^V(F)\backslash G^V(\dA_F)/K^V,1)[\chi^V]$. Moreover, by Lemma \ref{le:modify}, we may further assume that $\mu^V=\sf{bc}^V(\mu')$ for every $V\in\fV_{(\tS)}$. In all, we conclude that $\mu'\star f'$ and $((\mu^V\star f^V)_{V\in\fV_{(\tS)}},0,\dots)$ are complete smooth transfer by Proposition \ref{pr:transfer3}.

Now we run relative trace formulae from the general linear side to the unitary side. We have
\[
\sum_{\tilde\Pi\in\fC(G',1)}I_{\tilde\Pi}(\mu'\star f')=I_\Pi(\mu'\star f')=I_\Pi(f')\neq 0.
\]
Since $\mu'\star f'$ is quasi-cuspidal, we have
\[
\sum_{\gamma\in B'(F)}I_\gamma(\mu'\star f')=\sum_{\tilde\Pi\in\fC(G',1)}I_{\tilde\Pi}(\mu'\star f')\neq 0
\]
by Proposition \ref{pr:rtf1}. By Proposition \ref{pr:transfer2}, we have
\[
\sum_{V\in\fV_{(\tS)}^\circ}\sum_{\delta^V\in B^V(F)}J_{\delta^V}(\mu^V\star f^V)=\sum_{\gamma\in B'(F)}I_\gamma(\mu'\star f')\neq 0.
\]
Thus, we can choose some $V\in\fV_{(\tS)}$ such that
\[
\sum_{\delta^V\in B^V(F)}J_{\delta^V}(\mu^V\star f^V)\neq 0.
\]
Since $\mu^V\star f^V$ is quasi-cuspidal, we have
\[
\sum_{\pi\in\fC(G^V,1)}J_\pi(\mu^V\star f^V)=\sum_{\delta^V\in B^V(F)}J_{\delta^V}(\mu^V\star f^V)\neq 0
\]
by Proposition \ref{pr:rtf2}. Therefore, by the property of $\mu^V\star f^V$, we can find some element $\pi\in\fC(G^V,1)$ satisfying
\[
L^2(G^V,\pi)\cap L^2_\cusp(G^V(F)\backslash G^V(\dA_F)/K^V,1)[\chi^V]\neq\{0\},
\]
such that
\[
J_\pi(f^V)=J_\pi(\mu^V\star f^V)\neq 0.
\]
In particular, the weak automorphic base change of $\pi$ is isomorphic to $\Pi$, and $\sP$ is nonvanishing on cuspidal automorphic forms in $L^2(G^V,\pi)$ by \eqref{eq:character2}. Thus, (2) is achieved.

Theorem \ref{th:ggp} is proved.
\end{proof}

\begin{remark}
In the proof of Theorem \ref{th:ggp}, we actually obtain a stronger statement in the direction (1)$\Rightarrow$(2) by further requiring in (2) that
\begin{itemize}
  \item the base change of $\chi_{\pi_{m,\infty}}$ is $\chi_{\Pi_{m,\infty}}$ for $m=n,n+1$;

  \item the forms $\varphi_n\otimes\varphi_{n+1}$ is fixed by a relative hyperspecial maximal subgroup at every prime $v$ of $F$ that is unramified in $E$ and such that $\Pi_{n,v}\otimes\Pi_{n+1,v}$ is unramified.
\end{itemize}
\end{remark}

\begin{proof}[Proof of Theorem \ref{th:nonvanishing}]
Let $V=(V_n,V_{n+1})\in\fV$ be the unique element such that $G^V$ is quasi-split. Since $\Pi_{n+1}$ is hermitian, by Lemma \ref{le:descent}, we have an element $\pi_{n+1}\in\fC(G^V_{n+1},1)$ such that $\BC(\pi_{n+1})\simeq\Pi_{n+1}$. By the Burger--Sarnak trick \cite{Zha1}*{Proposition~2.14}, we can find another element $\pi_n\in\fC(G^V_n,1)$ that is supercuspidal at some prime of $F$ split in $E$, fulfilling the situation in Theorem \ref{th:ggp}(2). By Theorem \ref{th:abc}(2), $\BC(\pi_n)$ exists and is cuspidal and hermitian, which we denote by $\Pi_n$. Moreover, we have
\[
L(\tfrac{1}{2},\Pi_n\times\Pi_{n+1})\neq 0
\]
by Theorem \ref{th:ggp}. The theorem is proved.
\end{proof}

\begin{proof}[Proof of Theorem \ref{th:ii}]
We continue the proof of Theorem \ref{th:ggp}. Using the endoscopic classification for generic packets obtained in \cites{Mok15,KMSW} and the local Gan--Gross--Prasad \cite{BP0}, we arrive at the identity
\[
I_\Pi(f')=J_\pi(f^V)\neq 0
\]
for some $f'=\bigotimes_vf'_v\in\cS(G'(\dA_F))_{K'}$ and $f^V=\bigotimes_vf^V_v\in\cS(G^V(\dA_F))_{K^V}$ such that $f'_v$ and $f^V_v$ have matching orbital integrals for every place $v$. The remaining argument is the same as in the proof of \cite{BP2}*{Theorem~5}.
\end{proof}

\appendix

\section{Extending a result of Delorme to reductive groups}
\label{app}

In this appendix, we extend a theorem of Delorme \cite{Del86}*{Th\'{e}or\`{e}me~1.7} from semisimple groups to reductive groups. Let the setup be as in Subsection \ref{ss:multiplier_algebra}. In particular, $G=\ul{G}(\dR)$ for a connected reductive algebraic group $\ul{G}$ over $\dR$. Denote by $C^\infty_c(G)_{(K)}$ and $\cS(G)_{(K)}$ the subalgebras of bi-$K$-finite functions in $C^\infty_c(G)$ and $\cS(G)$, respectively. Recall from Definition \ref{de:multiplier} that $\cN(\fh^*_\dC)$ is the space of holomorphic functions on $\fh^*_\dC$ that have rapid decay on vertical strips.

\begin{proposition}\label{pr:delorme}
For every element $\mu\in\cN(\fh^*_\dC)$, there is a unique linear operator
\[
\mu\star\colon C^\infty_c(G)_{(K)}\to\cS(G)_{(K)},
\]
such that
\[
\pi(\mu\star f)=\mu(\chi_\pi)\cdot\pi(f)
\]
holds for every $f\in C^\infty_c(G)_{(K)}$ and every irreducible admissible representation $\pi$ of $G$.
\end{proposition}

When $G$ is semisimple, linear, and connected (in the analytic topology), this is exactly \cite{Del86}*{Th\'{e}or\`{e}me~1.7}. In general, we need to show that there is a bilinear map
\begin{align}\label{eq:delorme1}
\cN(\fh_\dC^*)^\sfW \times C^\infty_c(G)_{(K)} \to \cS(G)_{(K)}
\end{align}
sending $(\mu,f)$ to $\mu\star f$, satisfying the requirement in the proposition.

In what follows, for a real Lie group $H$, we denote by $H^0$ its neutral connected component. By the Iwasawa decomposition, the natural map $K/K^0\to G/G^0$ is an isomorphism, which gives rise to a decomposition
\begin{align}\label{eq:delorme2}
C^\infty_c(G)_{(K)}=\bigoplus_{k\in K/K^0}C^\infty_c(kG^0)_{(K^0)}.
\end{align}
Here, $C^\infty_c(kG^0)_{(K^0)}$ denotes the space of compactly supported smooth functions on $kG^0$ that are bi-$K^0$-finite. We first reduce the construction of \eqref{eq:delorme1} to the one for $G^0$:
\begin{align}\label{eq:delorme3}
\cN(\fh_\dC^*)^\sfW \times C^\infty_c(G^0)_{(K^0)} \to \cS(G^0)_{(K^0)}
\end{align}
sending $(\mu,f)$ to $\mu\star f$, satisfying $\pi(\mu\star f)=\mu(\chi_\pi)\cdot\pi(f)$ for every irreducible admissible representation $\pi$ of $G^0$. Indeed, once we have \eqref{eq:delorme3}, we may define
\[
\mu\star f\coloneqq\sum_{k\in K/K^0}\rL(k)(\mu\star \rL(k^{-1})f_k)
\]
for $(\mu,f)\in\cN(\fh_\dC^*)^\sfW \times C^\infty_c(G)_{(K)}$, where $f=\sum_{k\in K/K^0}f_k$ is the decomposition of $f$ under \eqref{eq:delorme2}, and $\rL$ denotes the left regular action. Since the restriction to $G^0$ of an irreducible admissible representation of $G$ is a finite direct sum of irreducible admissible representations of $G^0$ with the same infinitesimal character, it is easy to check that the above bilinear map satisfies the requirement in the proposition.

Now it remains to construct \eqref{eq:delorme3}. Denote by $\ul{Z}$ and $\ul{G}_\der$ the center and the derived subgroup of $\ul{G}$, respectively. Put $Z\coloneqq\ul{Z}(\dR)$ and $G_\der\coloneqq\ul{G}_\der(\dR)$. Put $K^0_Z\coloneqq K\cap Z^0$, $K^0_\der\coloneqq K\cap G_\der^0$, and $Z^0_\der\coloneqq Z^0\cap G_\der=K^0_Z\cap K^0_\der$. Then $Z^0_\der$ is a finite group, and the natural inclusions induce isomorphisms $(K^0_Z\times K^0_\der)/Z^0_\der\xrightarrow{\sim}K^0$ and
\begin{align}\label{eq:delorme4}
(Z^0\times G^0_\der)/Z^0_\der \xrightarrow{\sim} G^0.
\end{align}

For finite subsets $\fT_Z$ and $\fT_\der$ of $K^0_Z$-types and $K^0_\der$-types, respectively, we define $\fT$ to be the set of $K^0$-types whose inflation to $K^0_Z\times K^0_\der$ is an exterior tensor product of elements in $\fT_Z$ and $\fT_\der$. Similarly, for compact subsets $\Omega_Z$ and $\Omega_\der$ of $Z^0$ and $G^0_\der$ that are bi-invariant under $K^0_Z$ and $K^0_\der$, respectively, we let $\Omega$ be the image of $\Omega_Z\times\Omega_\der$ under the isomorphism \eqref{eq:delorme4}, which is a compact subset of $G^0$ bi-invariant under $K^0$. Clearly, we have
\[
C^\infty_c(G^0)_{(K^0)}=\bigcup_{\Omega}\bigcup_{\fT}C^\infty_\Omega(G^0)_{(\fT)},
\]
where $C^\infty_\Omega(G^0)_{(\fT)}$ denotes the subspace of smooth functions supported on $\Omega$ of bi-$K^0$-types in $\fT$. Thus, it suffices to construct \eqref{eq:delorme3} for a fixed subspace $C^\infty_\Omega(G^0)_{(\fT)}$ of $C^\infty_c(G^0)_{(K^0)}$. The isomorphism \eqref{eq:delorme4} induces a natural isomorphism
\begin{align}\label{eq:delorme5}
\(C^\infty_{\Omega_Z}(Z^0)_{(\fT_Z)}\widehat\otimes C^\infty_{\Omega_\der}(G^0_\der)_{(\fT_\der)}\)^{Z^0_\der}
\xrightarrow{\sim}C^\infty_\Omega(G^0)_{(\fT)}
\end{align}
of Fr\'{e}chet spaces, where $\widehat\otimes$ stands for the projective completed tensor product.\footnote{This is in general not correct if one does not fix the $K^0$-types and the support.}

On the other hand, let $\fz^*$ and $\fh^*_\der$ be the real vector spaces spanned by the weight lattices of $\ul{Z}_\dC$ and the abstract Cartan group of $\ul{G}_{\der,\dC}$, respectively.

\begin{lem}\label{le:pw}
The decomposition $\fh^*=\fz^*\oplus\fh^*_\der$ induced from \eqref{eq:delorme4} induces a natural isomorphism
\[
\cN(\fz^*_\dC)\widehat\otimes\cN(\fh^*_{\der,\dC})^\sfW\xrightarrow{\sim}\cN(\fh^*_\dC)^\sfW
\]
of Fr\'{e}chet spaces. Here, $\cN(\fz^*_\dC)$ and $\cN(\fh^*_{\der,\dC})$ are defined similarly as for $\cN(\fh^*_\dC)$.
\end{lem}

\begin{proof}
Let $Z^0_{\r{sp}}$, $H^0_\der$, and $H^0$ be the split real analytic tori whose cotangent spaces are $\fz^*$, $\fh^*_\der$, and $\fh^*$, respectively. By the classical Paley--Wiener theorem for Schwartz functions on split real analytic tori, we have canonical isomorphisms
\[
\cS(Z^0_{\r{sp}})\simeq\cN(\fz^*_\dC),\qquad
\cS(H^0_\der)\simeq\cN(\fh^*_{\der,\dC}),\qquad
\cS(H^0)\simeq\cN(\fh^*_\dC)
\]
induced by Fourier transforms. Since $H^0=Z^0_{\r{sp}}\times H^0_\der$, the natural map $\cS(Z^0_{\r{sp}})\widehat\otimes\cS(H^0_\der)\to\cS(H^0)$ is an isomorphism of Fr\'{e}chet spaces by \cite{AG10}*{Corollary~2.6.3}. By taking $\sfW$-invariants, we obtain the isomorphism in the lemma.
\end{proof}

Now take $\mu_1\in\cN(\fz^*_\dC)$, $\mu_2\in\cN(\fh^*_{\der,\dC})^\sfW$, and put $\mu\coloneqq\mu_1\otimes\mu_2\in\cN(\fh^*_\dC)^\sfW$. By \cite{Del86}*{Th\'{e}or\`{e}me~1.7}, there is a linear map
\[
\mu_2\star \colon C^\infty_{\Omega_\der}(G^0_\der)_{(\fT_\der)}\to \cS(G^0_\der)_{(\fT_\der)}
\]
sending $f$ to $\mu_2\star f$, satisfying $\pi_\der(\mu_2\star f)=\mu_2(\chi_{\pi_\der})\cdot\pi_\der(f)$ for every irreducible admissible representation $\pi_\der$ of $G^0_\der$. Similarly, by \cite{Sak}*{Theorem~2.1.2}, there is a linear map
\[
\mu_1\star \colon C^\infty_{\Omega_Z}(Z^0)_{(\fT_Z)}\to \cS(Z^0)_{(\fT_Z)}
\]
sending $f$ to $\mu_1\star f$, satisfying $\xi(\mu_1\star f)=\mu_1(\chi_\xi)\cdot\xi(f)$ for every smooth character $\xi$ of $Z^0$. Moreover, by the injectivity of the (operator valued) Fourier transform and the closed graph theorem, both $\mu_1\star$ and $\mu_2\star$ are continuous. Therefore, $(\mu_1\star)\otimes(\mu_2\star)$ extends uniquely to a continuous linear map
\begin{align*}
\mu\star\colon C^\infty_\Omega(G^0)_{(\fT)}
&\simeq\(C^\infty_{\Omega_Z}(Z^0)_{(\fT_Z)}\widehat\otimes C^\infty_{\Omega_\der}(G^0_\der)_{(\fT_\der)}\)^{Z^0_\der} \\
&\to\(\cS(G^0_\der)_{(\fT_Z)}\widehat\otimes\cS(Z^0)_{(\fT_\der)}\)^{Z^0_\der}
\to\cS(G^0)_{(\fT)}
\end{align*}
by \eqref{eq:delorme5}. Let $\pi$ be an irreducible admissible representation of $G^0$. By the decomposition \eqref{eq:delorme4}, we may write the inflation of $\pi$ to $Z^0\times G^0_\der$ as $\xi\boxtimes\pi_\der$ where $\xi$ is a smooth character of $Z^0$ and $\pi_\der$ is an irreducible admissible representation of $G^0_\der$. By construction, we have $\pi(\mu\star f)=\mu(\chi_\pi)\pi(f)$ for every $f\in C^\infty_{\Omega_Z}(Z^0)_{(\fT_Z)}\otimes C^\infty_{\Omega_\der}(G^0_\der)_{(\fT_\der)}$ that is $Z^0_\der$-invariant. Since the map $f\in C^\infty_\Omega(G^0)_{(\fT)}\mapsto \pi(f)$ is continuous when we equip the space of continuous endomorphisms of $\pi$ with the topology of pointwise convergence, we deduce by density that $\pi(\mu\star f)=\mu(\chi_\pi)\pi(f)$ holds for every $f\in C^\infty_\Omega(G^0)_{(\fT)}$. In summary, we have constructed our desired bilinear map from $\cN(\fz^*_\dC)\otimes\cN(\fh^*_{\der,\dC})^\sfW\times C^\infty_\Omega(G^0)_{(\fT)}$ to $\cS(G^0)_{(\fT)}$. Now again by the injectivity of the (operator valued) Fourier transform and the closed graph theorem, we know that, for every $f\in C^\infty_\Omega(G^0)_{(\fT)}$, the map
\begin{align}\label{eq:delorme6}
\star f\colon\cN(\fz^*_\dC)\times\cN(\fh^*_{\der,\dC})^\sfW\to\cS(G^0)_{(\fT)}
\end{align}
sending $(\mu_1,\mu_2)$ to $(\mu_1\otimes\mu_2)\star f$ is a separately continuous bilinear map. Therefore, \eqref{eq:delorme6} extends uniquely to a continuous map $\star{f}\colon\cN(\fz^*_\dC)\widehat\otimes\cN(\fh^*_{\der,\dC})^\sfW\to\cS(G^0)_{(\fT)}$ by \cite{Tre67}*{Theorem~34.1}. By Lemma \ref{le:pw}, we obtain the desired bilinear map $\cN(\fh_\dC^*)^\sfW \times C^\infty_\Omega(G^0)_{(\fT)}\to \cS(G^0)_{(\fT)}$ sending $(\mu,f)$ to $\mu\star f$, satisfying $\pi(\mu\star f)=\mu(\chi_\pi)\cdot\pi(f)$ for every irreducible admissible representation $\pi$ of $G^0$. Taking union over all $\fT$ and $\Omega$, we obtain \eqref{eq:delorme3}.

Proposition \ref{pr:delorme} is proved.

\end{document}